\documentclass[11pt]{article}
\usepackage[centering,includeheadfoot,margin=2 cm]{geometry}
\usepackage{graphics,enumitem,epsfig}
\usepackage{amsfonts,amsmath,amssymb,euscript,color,stackrel,amsthm}
\usepackage{hyperref}

\begin{document}

\def\ALERT#1{{\large\bf $\clubsuit$#1$\clubsuit$}}

\numberwithin{equation}{section}
\newtheorem{defin}{Definition}
\newtheorem{theorem}{Theorem}
\newtheorem{proposition}{Proposition}
\newtheorem{notice}{Notice}
\newtheorem*{convention}{Convention}
\newtheorem{hypothesis}{Hypothesis}
\newtheorem{lemma}{Lemma}
\newtheorem{cor}{Corollary}
\newtheorem{example}{Example}
\newtheorem{remark}{Remark}
\newtheorem{conj}{Conjecture}
\def\begproof{\noindent{\bf Proof: }}
\def\endproof{\par\rightline{\vrule height5pt width5pt depth0pt}\medskip}
\def\div{\nabla\cdot}
\def\divergence{{\rm div}}
\def\rot{\nabla\times}
\def\sign{{\rm sign}}
\def\arsinh{{\rm arsinh}}
\def\arcosh{{\rm arcosh}}
\def\diag{{\rm diag}}
\def\const{{\rm const}}
\def\d{\,\mathrm{d}}

\newcommand{\eps}{\varepsilon}
\newcommand{\Lm}{\mathcal{L}}
\newcommand{\Hm}{\mathcal{H}}
\newcommand{\E}{\mathcal{E}}
\newcommand{\T}{\mathcal{T}}
\newcommand{\Pol}{\mathcal{P}}
\newcommand{\X}{\mathcal{X}}
\newcommand{\RR}{\mathbb{R}}
\newcommand{\NN}{\mathbb{N}}

\def\eps{\varepsilon}
\def\phi{\varphi}
\def\theta{\vartheta}
\def\N{\mathbb{N}}
\def\R{\mathbb{R}}
\def\C{\hbox{\rlap{\kern.24em\raise.1ex\hbox
      {\vrule height1.3ex width.9pt}}C}}
\def\P{\hbox{\rlap{I}\kern.16em P}}
\def\Q{\hbox{\rlap{\kern.24em\raise.1ex\hbox
      {\vrule height1.3ex width.9pt}}Q}}
\def\M{\hbox{\rlap{I}\kern.16em\rlap{I}M}}
\def\Z{\hbox{\rlap{Z}\kern.20em Z}}
\def\({\begin{eqnarray}}
\def\){\end{eqnarray}}
\def\[{\begin{eqnarray*}}
\def\]{\end{eqnarray*}}
\def\part#1#2{\frac{\partial #1}{\partial #2}}
\def\partk#1#2#3{\frac{\partial^#3 #1}{\partial #2^#3}} 
\def\mat#1{{D #1\over Dt}}
\def\dx{\nabla_x}
\def\dv{\nabla_v}
\def\grad{\nabla}
\def\curl{\mathrm{curl}}

\def\Norm#1{\left\| #1 \right\|}
\def\pmb#1{\setbox0=\hbox{$#1$}
  \kern-.025em\copy0\kern-\wd0
  \kern-.05em\copy0\kern-\wd0
  \kern-.025em\raise.0433em\box0 }
\def\bar{\overline}
\def\lbar{\underline}
\def\fref#1{(\ref{#1})}
\def\half{\frac{1}{2}}
\def\oo#1{\frac{1}{#1}}

\def\tot#1#2{\frac{\d #1}{\d #2}} 
\def\laplace{\Delta}
\def\d{\,\mathrm{d}}
\def\N{\mathbb{N}}
\def\R{\mathbb{R}}
\def\supp{\mbox{supp }}
\def\eps{\varepsilon}
\def\phi{\varphi}

\def\E{\mathcal{E}}
\def\bbE{\mathbb{E}}

\def\Ikn{\I_{k_n}}
\def\sIkn{(\Ikn)_{n\in\N}}

\def\O{\mathcal{O}}
\def\A{\mathcal{A}}
\def\calG{\mathcal{G}}
\def\calF{\mathcal{F}}
\def\calX{\mathcal{X}}
\def\calY{\mathcal{Y}}
\def\calT{\mathcal{T}}
\def\calH{\mathcal{H}}
\def\calR{\mathcal{R}}
\def\bbX{\mathbb{X}}

\def\bZ{{\bar Z}}

\def\laplaceD{\laplace}

\def\blueJH#1{{\textcolor{blue}{#1}}}

\def\comment#1{\textbf{[Comment: #1]}}

\def\flux{u}


\centerline{{\huge Notes on a PDE System}}
\centerline{{\huge for Biological Network Formation}}
\vskip 7mm


\vskip 5mm

\centerline{
{\large Jan Haskovec}\footnote{Mathematical and Computer Sciences and Engineering Division,
King Abdullah University of Science and Technology,
Thuwal 23955-6900, Kingdom of Saudi Arabia; 
{\it jan.haskovec@kaust.edu.sa}}\qquad
{\large Peter Markowich}\footnote{Mathematical and Computer Sciences and Engineering Division,
King Abdullah University of Science and Technology,
Thuwal 23955-6900, Kingdom of Saudi Arabia; 
{\it peter.markowich@kaust.edu.sa}}\qquad
{\large Beno\^\i t Perthame}\footnote{Sorbonne Universit\'eŽs, UPMC Univ Paris 06, Inria, Laboratoire Jacques-Louis Lions  UMR CNRS 7598, F-75005, Paris, France,
{\it benoit.perthame@upmc.fr}}\qquad
{\large Matthias Schlottbom}\footnote{Institute for Computational and Applied Mathematics,
University of M\"unster, Einsteinstr. 62, 48149 M\"unster, Germany; {\it schlottbom@uni-muenster.de}}
}
\vskip 6mm


\noindent{\bf Abstract.}
We present new analytical and numerical results for the
elliptic-parabolic system of partial differential equations proposed by
Hu and Cai \cite{Hu, Hu-Cai}, which models the formation
of biological transport networks.
The model describes the pressure field using a Darcy's type
equation and the dynamics of the conductance network under pressure force effects.
Randomness in the material structure is represented by a linear diffusion term 
and conductance relaxation by an algebraic decay term. 
The analytical part extends the results of \cite{HMP15}
regarding the existence of weak and mild solutions to the
whole range of meaningful relaxation exponents. 
Moreover, we prove finite time extinction 
or break-down of solutions in the spatially one-dimensional setting
for certain ranges of the relaxation exponent.
We also construct stationary solutions for the case of vanishing diffusion
and critical value of the relaxation exponent,
using a variational formulation and a penalty method.

The analytical part is complemented by extensive numerical simulations.
We propose a discretization based on mixed finite elements and study
the qualitative properties of network structures for various parameters values.
Furthermore, we indicate numerically that some analytical results
proved for the spatially one-dimensional setting are likely to be valid also in several space dimensions.

\vskip 5mm

\vspace{2mm}

\noindent{\bf Math. Class. No.:} 35K55; 35B32; 92C42

\tableofcontents

\section{Introduction}
\label{sec:introduction}

In \cite{HMP15} we presented a mathematical analysis of the PDE system modeling formation of biological transportation networks
\(
    -\div [(rI + m\otimes m)\grad p] &=& S,   \label{eq01}\\
    \part{m}{t} - D^2\laplace m - c^2(m\cdot\grad p)\grad p + \alpha |m|^{2(\gamma-1)}m &=& 0, \label{eq02}
\)
for the scalar pressure $p=p(t,x)\in\R$ of the fluid transported within the network and vector-valued conductance $m=m(t,x)\in\R^d$
with $d\leq 3$ the space dimension.
The parameters are $D\geq 0$ (diffusivity), $c>0$ (activation parameter), $\alpha>0$ and $\gamma\in\R$;
in particular, we restricted ourselves to $\gamma\geq 1$ in \cite{HMP15}.
The scalar function $r=r(x) \geq r_0 > 0$ describes the isotropic background permeability of the medium.
The source term $S=S(x)$ is assumed to be independent of time and $\gamma\in\R$ is a parameter crucial for the type of networks formed \cite{Hu-Cai}.
In particular, experimental studies of scaling relations of conductances (diameters) of parent and daughter edges
in realistic network modeling examples suggest that $\gamma=1/2$ can be used to model blood vessel systems in the human body
and $\gamma=1$ is adapted to  leaf venation \cite{Hu, Hu-pc}.
For the details on the modeling which leads to \eqref{eq01}, \eqref{eq02} we refer to \cite{MMM}.

The system was originally derived in \cite{Hu, Hu-Cai} as the formal gradient flow of the continuous version
of an energy functional describing formation of biological transportation networks on discrete graphs.
We pose \eqref{eq01}, \eqref{eq02} on a bounded domain $\Omega\subset\R^d$ with smooth boundary $\partial\Omega$,
subject to homogeneous Dirichlet boundary conditions on $\partial\Omega$ for $m$ and $p$:
\(   \label{BC_0}
   m(t,x) = 0,\qquad p(t,x)=0 \qquad \mbox{for } x\in\partial\Omega,\; t\geq 0,
\)
and subject to the initial condition for $m$:
\( \label{IC_0}
   m(t=0,x)=m^0(x)\qquad\mbox{for } x\in\Omega.
\)

The main mathematical interest of the PDE system for network formation stems from the highly unusual nonlocal coupling of the 
elliptic equation \eqref{eq01} for the pressure $p$ to the reaction-diffusion equation \eqref{eq02} for the conductance vector $m$
via the pumping term $+c^2 (\grad p\otimes\grad p) m$ and the latter term's potential equilibriation with the decay term
$- |m|^{2(\gamma-1)}m$. 
A major observation concerning system \eqref{eq01}--\eqref{eq02} is that it represents
the formal $L^2(\Omega)$-gradient flow associated with the highly non-convex energy-type functional
\(   \label{energy}
   \E(m) := \frac12 \int_\Omega \left(D^2 |\grad m|^2 + \frac{\alpha}{\gamma}|m|^{2\gamma} + c^2|m\cdot\grad p[m]|^2  + c^2 r(x) |\grad p[m]|^2 \right) \d x,
\)
where $p=p[m] \in H^1_0(\Omega)$ is the unique solution of the Poisson equation \eqref{eq01} with given $m$,
subject to the homogeneous Dirichlet boundary condition on $\partial\Omega$.
Note that \eqref{energy} consists of, respectively, the diffusive energy term, metabolic (relaxation) energy,
and the last two terms account for network-fluid interaction energy.
We have:

\begin{lemma}[Lemma 1 in \cite{HMP15}]\label{lem:energy}
Let $\E(m^0) < \infty$. Then the energy $\E(m(t))$ is nonincreasing along smooth solutions of \eqref{eq01}--\eqref{eq02} and satisfies
\[
    \tot{}{t} \E(m(t)) = - \int_\Omega \left( \part{m}{t}(t,x)\right)^2 \d x.
\] 
\end{lemma}
As usual, along weak solutions,  we obtain a weaker form of energy dissipation, see formula \eqref{energy_ineq} below.

In \cite{HMP15} we provided the following analytical results for \eqref{eq01}--\eqref{IC_0} in the case $\gamma\geq 1$:
\begin{itemize}
 \item Existence of global weak solutions in the energy space
 \item Existence and uniqueness of local in time mild solutions (global in 1d)
 \item Existence of nontrivial (i.e., $m\not\equiv 0$) stationary states and analysis of their stability (nonlinear in 1d, linearized in multiple dimensions)
 \item The limit $D\to 0$ in the 1d setting
\end{itemize}

The purpose of this paper is to extend the analysis of the network formation system by providing several new results, in particular:
\begin{itemize}
 \item Existence of global weak solutions in the energy space for $1/2 \leq \gamma < 1$ and
 of local in time mild solutions for $1/2 < \gamma < 1$ (Section \ref{S1}).
 \item Analysis of the system in the 1d setting: finite time breakdown of solutions for $\gamma<1/2$,
 infinite time extinction for $1/2 \leq \gamma \leq 1$ with small sources, nonlinear stability analysis for $\gamma\geq 1/2$ and $D=0$
 (Section \ref{S2}).
 \item Construction of stationary solutions in the case $\gamma=1$ and $D=0$ (Section \ref{S3}).
\end{itemize}
The analytical part is complemented by extensive numerical examples in Section \ref{sec:Num}.
We propose a discretization based on mixed finite elements and study
the qualitative properties of network structures for various parameters values.
Furthermore, we indicate numerically that some analytical results
proved for the spatially one-dimensional setting are likely to be valid also in several space dimensions.

\subsection{Scaling analysis}
We introduce the rescaled variables
\[
   x_s := \frac{x}{\bar x}, \qquad t_s := \frac{t}{\bar t},\qquad m_s := \frac{m}{\bar m}, \qquad p_s := \frac{p}{\bar p}, \qquad S_s := \frac{S}{\bar S}
\]
and choose
\[
   \bar x := \mbox{diam}(\Omega), \qquad  \bar m := \sup_{x\in\Omega} |m^0(x)|, \qquad \bar t := \frac{1}{\alpha\bar m^{2(\gamma-1)}},\qquad 
   \bar S := \sup_{x\in\Omega} |S(x)|,\qquad \bar p := \frac{\bar x^2\bar S}{\bar m^2}
\]
which leads to $S_s = \mathcal{O}(1)$, $m_s(t=0) = \mathcal{O}(1)$ and the following rescaled version of \eqref{eq01}--\eqref{eq02},
\[
    -\grad_{x_s}\cdot [(r_sI + m\otimes m)\grad_{x_s} p_s] &=& S_s,   \\
    \part{m_s}{t} - D_s^2\laplace_{x_s} m_s - c_s^2(m_s\cdot\grad_{x_s} p_s)\grad p_s + |m_s|^{2(\gamma-1)}m_s &=& 0,
\]
with
\[
   r_s = \frac{r}{\bar m^2},\qquad D_s^2 = \frac{\bar p\,\bar m^2}{\bar x^2\bar s},\qquad c_s^2 = \frac{c^2\bar p^2}{\alpha \bar x^2\bar m^{2(\gamma-1)}}.
\]
Dropping the index $s$ in the scaled variables, we obtain the system
\(
    -\div [(rI + m\otimes m)\grad p] &=& S,   \label{eq1}\\
    \part{m}{t} - D^2\laplace m - c^2(m\cdot\grad p)\grad p + |m|^{2(\gamma-1)}m &=& 0, \label{eq2}
\)
that we will study in this paper.
Moreover, for simplicity, we set $r(x)\equiv 1$ in the analytical part (Sections \ref{S1}--\ref{S3}).

\begin{convention}
In the following, generic, not necessarily equal, constants will be denoted by $C$.
Moreover, we will make specific use of the Poincar\'e constant  $C_\Omega$, i.e.,
\[
    \Norm{u}_{L^2(\Omega)} \leq C_\Omega \Norm{\grad u}_{L^2(\Omega)} \qquad\mbox{for all } u\in H_0^1(\Omega).
\]
\end{convention}

\section{Existence of global weak solutions for $1/2 \leq \gamma < 1$}\label{S1}
In \cite{HMP15}, Section 2, we provided the proof of existence of global weak solutions for $\gamma\geq 1$
based on the Leray-Schauder fixed point theorem for a regularized version of \eqref{eq1}--\eqref{eq2}
that preserves the energy dissipation structure, and consequent
limit passage to remove the regularization. We now extend the proof to the case $1/2 \leq \gamma < 1$.
However, the case $\gamma=1/2$ requires special care
since the algebraic term in \eqref{eq2} formally becomes $m/|m|$
and an interpretation has to be given for $m=0$.
In particular, \eqref{eq2} has to be substituted by
the differential inclusion
\( \label{m-incl}
    \partial_t m - D^2\laplace m - c^2(m\cdot\grad p[m])\grad p[m] \in -\partial \mathcal{R}(m),
\)
where $\partial\mathcal{R}$ is the subdifferential of $\mathcal{R}(m):=\int_\Omega |m| \d x$, in particular,
\( \label{subdiffR}
   \partial\mathcal{R}(m) = \{ u\in L^\infty(\Omega)^d;\, u(x) &=& m(x)/ |m(x)| \mbox{ if } m(x) \neq 0,\\ |u(x)| &\leq& 1 \mbox{ if } m(x)=0 \}.
\)

\begin{theorem}[Extension of Theorem 1 of \cite{HMP15}]
Let $\gamma \geq 1/2$, $S\in L^2(\Omega)$ and $m^0 \in H_0^1(\Omega)^d \cap L^{2\gamma}(\Omega)^d$.
Then the problem \eqref{eq1}--\eqref{eq2}, \eqref{BC_0}--\eqref{IC_0} (with \eqref{m-incl} instead of \eqref{eq2} if $\gamma=1/2$)
admits a global weak solution $(m,p[m])$ with $\E(m)\in L^\infty(0,\infty)$ and with
\[
     m \in L^\infty(0,\infty; H_0^1(\Omega)) \cap L^\infty(0,\infty; L^{2\gamma}(\Omega)), &\qquad&
    \partial_t m \in L^2((0,\infty)\times\Omega),\\
    \grad p \in L^\infty(0,\infty; L^2(\Omega)), &\qquad&
    m\cdot\grad p \in L^\infty(0,\infty; L^2(\Omega)).
\]
This solution satisfies the energy dissipation inequality,  with $\E$ given by \eqref{energy}, 
\(    \label{energy_ineq}
    \E(m(t)) + \int_0^t \int_\Omega \left( \part{m}{t}(s,x)\right)^2 \d x\d s \leq \E(m^0) \qquad\mbox{for all } t \geq 0.
\)
\label{thm:global_ex}
\end{theorem}

The proof proceeds along the lines of Section 2 of \cite{HMP15}, i.e.,
for $\gamma>1/2$ and $\eps>0$ we consider the regularized system
\(
    -\div [\grad p + m(m\cdot\grad p)\ast\eta_\eps ] &=& S,   \label{eq_pert1}\\
    \part{m}{t} - D^2\laplace m - c^2[(m\cdot\grad p)\ast\eta_\eps]\grad p + |m|^{2(\gamma-1)}m &=& 0, \label{eq_pert2}
\)
with $(\eta_\eps)_{\eps>0}$ the $d$-dimensional heat kernel $\eta_\eps(x) = (4\pi\eps)^{-d/2} \exp(-|x|^2/4\eps)$ and $m$, $p$ are extended by $0$ outside $\Omega$ so that the convolution is well defined.
For $\gamma=1/2$, equation \eqref{eq_pert2} has to be substituted by the differential inclusion
\[ \label{eq_pert2-incl}
    \part{m}{t} - D^2\laplace m - c^2[(m\cdot\grad p)\ast\eta_\eps]\grad p \in -\partial \mathcal{R}(m).
\]
Weak solutions of the regularized system are constructed by an application
of the Leray-Schauder fixed point theorem as in Section 2 in \cite{HMP15},
the only change that needs to be done is a slight modification of the proof of Lemma 3 in \cite{HMP15}.
In particular, for $\gamma>1/2$ we construct weak solutions of the auxiliary problem
\(   \label{m-aux}
    \part{m}{t} - D^2\laplace m - f = -|m|^{2(\gamma-1)}m,
\)
subject to the initial and boundary conditions
\(   \label{m-aux-IC-BC}
   m(t=0)=m^0\quad\mbox{in }\Omega,\qquad m=0\quad\mbox{on }\partial\Omega.
\)
Again, for $\gamma=1/2$ we have to consider the following differential inclusion instead,
\(   \label{aux-incl}
    \part{m}{t} - D^2\laplace m -f \in -\partial\mathcal{R}(m).
\)
In the subsequent lemma we construct the so-called \emph{slow solution} of \eqref{aux-incl},
which is the unique weak solution of the PDE
\( \label{m-slow}
    \part{m}{t} - D^2\laplace m - f = -r(m)
\)
with
\begin{equation}    \label{r}
   [r(m)](x) =
   \left\{ \begin{array}{ll}
    m(x)/ |m(x)|  &\mbox{ when } m(x) \neq 0, \\ 0  &\mbox{ when } m(x)=0.
   \end{array} \right.  
\end{equation}

\begin{lemma}[Extension of Lemma 3 in \cite{HMP15}]\label{lem:m-aux}
For every $D>0$, $\gamma> 1/2$, $T>0$ and $f\in L^2((0,T)\times\Omega)^d$, the problem \eqref{m-aux}--\eqref{m-aux-IC-BC}
with $m^0\in H_0^1(\Omega)^d$ admits a unique weak solution
$m\in L^\infty(0,T; H_0^1(\Omega))^d \cap L^2(0,T; H^2(\Omega))^d \cap L^\infty(0,T; L^{2\gamma}(\Omega))^d$
with $\partial_t m \in L^2((0,T)\times\Omega)^d$
and the estimates hold
\(
   \Norm{m}_{L^\infty(0,T; H_0^1(\Omega))} &\leq& C  \left( \Norm{f}_{L^2((0,T)\times\Omega)} + \Norm{m^0}_{H_0^1(\Omega)} \right), \label{m-aux-est4a}\\
   \Norm{\laplace m}_{L^2((0,T)\times\Omega)} &\leq& C \left( \Norm{f}_{L^2((0,T)\times\Omega)} + \Norm{m^0}_{H^1_0(\Omega)} \right). \label{m-aux-est4b}
\)
The same statement is true for the problem \eqref{m-slow}--\eqref{r}, \eqref{m-aux-IC-BC} in the case $\gamma=1/2$.
\end{lemma}

\begin{proof}
In both cases (i.e. $\gamma\geq 1/2$) we construct the solution of the differential inclusion
\(   \label{proof-incl}
    \part{m}{t} + \partial \mathcal{I}_\gamma(m) \ni f
\)
with the functional $\mathcal{I}_\gamma: L^2(\Omega) \to [0,+\infty]$ given by
\[
    \mathcal{I}_\gamma(m) := \frac{D^2}2 \int_\Omega |\grad m|^2 \d x + \frac{1}{2\gamma} \int_\Omega |m|^{2\gamma}\d x,\qquad\mbox{if } m\in H^1_0(\Omega)\cap L^{2\gamma}(\Omega),
\]
and $\mathcal{I}_\gamma(m):= +\infty$ otherwise.
It can be easily checked that for $\gamma\geq 1/2$ the functional $\mathcal{I}_\gamma$ is proper with dense domain,
strictly convex and lower semicontinuous on $H^1_0(\Omega)$.
By the Rockafellar theorem \cite{Phelps}, the Fr\'echet subdifferential $\partial\mathcal{I}_\gamma(m)$
is a maximal monotone operator and the standard theory \cite{Aubin-Cellina} then provides the existence
of a unique solution $m\in L^2(0,T; H^1_0(\Omega))^d \cap L^{2\gamma}((0,T)\times\Omega)^d$ of \eqref{proof-incl}.
Clearly, for $\gamma>1/2$, $m$ is the unique weak solution of \eqref{m-aux}--\eqref{m-aux-IC-BC}.
For $\gamma=1/2$, $m$ is the so-called \emph{slow solution}, meaning that
the velocity $\tot{m}{t}$ is the element of minimal norm
in $\partial\mathcal{R}(m)$, i.e.,
\[
    \tot{m}{t} = - \mbox{argmin} \{ \Norm{u}_{H^1_0(\Omega)}; u \in \partial\mathcal{R}(m) \}.
\]
Therefore, $m$ is a weak solution of \eqref{m-slow}--\eqref{r}.

To prove the higher regularity estimates \eqref{m-aux-est4a}, \eqref{m-aux-est4b}, we use
(formally, but easily justifiable) $\laplace m$ as a test function, which after integration by parts leads to
\(    \label{testByLaplace}
    \frac12 \tot{}{t} \int_\Omega |\grad m|^2 \d x + D^2\int_\Omega |\laplace m|^2 \d x
       - \int_\Omega (|m|^{2(\gamma-1)}m)\cdot\laplace m \d x = \int_\Omega f\cdot\laplace m \d x.
\)
Then, denoting $\phi(m):=|m|^{2(\gamma-1)}m$, we have
\(   \label{nonnegativeTerm}
      - \int_\Omega (|m|^{2(\gamma-1)}m)\cdot\laplace m \d x = - \int_\Omega \phi(m)\cdot\laplace m \d x = \int_\Omega \grad m\cdot D\phi(m)\grad m \d x
\)
with
\[
   D\phi(m) = |m|^{2(\gamma-1)}\left( 2(\gamma-1)\frac{m}{|m|}\otimes\frac{m}{|m|} + I\right).
\]
Clearly, $D\phi(m)$ is a nonnegative matrix for $\gamma\geq 1/2$,
so that the term \eqref{nonnegativeTerm} is nonnegative.
The identity \eqref{testByLaplace} together with a standard density argument
gives directly the required regularity and the estimates \eqref{m-aux-est4a}, \eqref{m-aux-est4b}.
\end{proof}

The rest of the proof of existence of solutions of the regularized problem
\eqref{eq_pert1}--\eqref{eq_pert2} is identical to Section 2 of \cite{HMP15}.
For the limit $\eps\to 0$, we only need to provide the following result
for the case $\gamma=1/2$.

\begin{lemma}\label{lem:m-algebraic}
Let $m^k\to m$ strongly in $L^1((0,T)\times\Omega)$ as $k\to\infty$,
and denote $h^k:= r(m^k)$ with $r$ given by \eqref{r}.
Then there exists $h\in\partial\mathcal{R}(m)$ such that, for a whileuence,
\[
    h^k \rightharpoonup^* h \qquad\mbox{weakly* in } L^\infty((0,T)\times\Omega) \mbox{ as } k\to\infty.
\]
\end{lemma}

\begproof
Because $h^k=r(m^k)$ is uniformly bounded in $L^\infty((0,T)\times\Omega)$, there exists a
subsequence, still denoted by $h^k$, converging to $h\in L^\infty((0,T)\times\Omega)$ weakly*.
Due to the strong convergence of $m^k$ in $L^1$, there exists a subsequence converging almost everywhere to $m$.
Consequently, $h^k$ converges to $m/|m|$ almost everywhere on $\{m\neq 0\}$.
On $\{m=0\}$, we have $|h|\leq 1$, so that $h\in\partial\mathcal{R}(m)$ defined by \eqref{subdiffR}.
\endproof

Note that in the case $\gamma=1/2$, due to Lemma \ref{lem:m-algebraic},
we only obtain weak solutions of the system \eqref{eq1}, \eqref{m-incl}.
We conjecture that $m$ is in fact a slow solution of \eqref{m-incl}, i.e.,
that it solves
\[
    \partial_t m - D^2\laplace m - c^2(m\cdot\grad p[m])\grad p[m] = r(m)
\]
with $r(m)$ given by \eqref{r}. 

\begin{remark}
The proof of local in time existence of mild solutions
(Theorem 3 of \cite{HMP15}) carries over to the case $1/2 < \gamma < 1$
without modifications. However, the proof of uniqueness of mild solutions
by a contraction mapping argument requires $\gamma\geq 1$
and it is not clear how to adapt it for values of $\gamma$ less than one.
\end{remark}

\section{Analysis in the 1d setting}\label{S2}
Much more can be proved about the system \eqref{eq1}--\eqref{eq2}
in the spatially one dimensional setting. Then, and without loss of generality, we can consider it on the interval $\Omega:=(0,1)$.
The system reads
\(
   - \partial_x \bigl(\partial_x p + m^2 \partial_x p \bigr) &=& S,  \label{eq1D1}\\
   \partial_t m - D^2 \partial_{xx}^2 m - c^2(\partial_x p)^2 m + |m|^{2(\gamma-1)}m &=& 0,  \label{eq1D2}
\)
Additionally, throughout this section we assume $S>0$ a.e. on $(0,1$),
and for mathematical convenience we prescribe the mixed boundary conditions for $p$,
\[   
    \partial_x p(0) = 0, \qquad p(1) = 0,
\]
and homogeneous Neumann boundary condition for $m$.
Then, integrating \eqref{eq1D1} with respect to $x$, we obtain
\[
    (1+m^2)\partial_x p = - \int_0^x S(y) \d y.
\]
Denoting $B(x) := \int_0^x S(y)\d y$, we have
\(   \label{partx_p}
    \partial_x p = - \frac{B(x)}{1+m^2},
\)
so that the system \eqref{eq1D1}--\eqref{eq1D2} is rewritten as
\(   \label{meq}
    \partial_t m - D^2 \partial_{xx}^2 m = \left( \frac{c^2 B(x)^2}{(1+m^2)^2} - |m|^{2(\gamma-1)} \right) m.
\)

\subsection{Extinction of solutions for $-1 \leq \gamma\leq 1$ and small sources}\label{sec:finite_time_breakdown}
We show that if the source term $S$ is small enough in a suitable sense,
then solutions of \eqref{meq} converge to zero, either in infinite time for $1/2 \leq \gamma \leq 1$,
or in finite time for $-1 \leq \gamma < 1/2$.
In the latter case it means that the solutions can only exist on finite time intervals,
since the algebraic term $|m|^{2(\gamma-1)}m$ is singular at $m=0$
and solutions of \eqref{meq} cannot be extended
beyond the point where they reach zero.

\begin{lemma}\label{lem:extFT}
Let $D\geq 0$, $-1 \leq \gamma \leq 1$, $m^0\in L^\infty(0,1)$ and $c\Norm{B}_{L^\infty(0,1)} < Z_\gamma$ with
 \(   \label{Z_gamma}
     Z_\gamma := \frac{2}{\gamma+1} \left(\frac{1-\gamma}{1+\gamma} \right)^\frac{\gamma-1}{2}, \qquad
     Z_1 :=1.
 \)
Without loss of generality, let $\inf_{x\in(0,1)} |m^0(x)|>0$, and $m$ be a weak solution
of \eqref{meq} with homogeneous Neumann boundary conditions and $m(t=0)=m^0$.

Then, for $1/2 \leq \gamma \leq  1$, $\Norm{m}_{L^1(0,1)}$ converges to zero as $t\to\infty$.
For $-1 \leq \gamma < 1/2$ there exists a finite break-down time $T_0 > 0$
such that $\inf_{x\in(0,1)} |m(T_0,x)| = 0$.
\end{lemma}

\begin{proof}
 For $m>0$ we define the positive function
 \( \label{h_gamma}
    h_\gamma(m) := m^{2(\gamma -1)} \;  (1+ m^2)^2.
 \)
 It can be easily shown that for all $-1 \leq \gamma \leq  1$,
 $$
     \inf_{0 < m < \infty} h_\gamma(m) = Z_\gamma^2 > 0.
 $$
 Consequently, the assumption $c \|B\|_{L^\infty(0,1)} <  Z_\gamma$ implies
 $$
   \frac{c^2 \|B\|^2_{L^\infty(0,1)}}{(1+m^2)^2} - |m|^{2(\gamma-1)} <0 \qquad\mbox{for all } m\in\R.
 $$
As a consequence of the maximum principle for \eqref{meq}, we have the a-priori bound
 $$
 \| m(t)\|_{L^\infty(0,1)} \leq M:=  \| m^0\|_{L^\infty(0,1)}  \qquad\mbox{for all } t\geq 0.
 $$
Now, we can conclude that there exists a $\delta > 0$ such that 
\(  \label{delta_ineq}
  \frac{c^2 \|B\|^2_{L^\infty(0,1)}}{(1+m^2)^2} - |m|^{2(\gamma-1)} < -\delta \qquad \mbox{for }  |m| \leq M.
\)
As long as $\inf_{x\in(0,1)} |m(t)|>0$, we can multiply \eqref{meq} with $\sign(m)$ and integrate over $\Omega=(0,1)$,
\[
    \tot{}{t} \int_0^1 |m| \d x = D^2 \int_0^1 (\partial^2_{xx} m)\,\sign(m) \d x 
       + \int_0^1 \left( \frac{c^2B^2  |m|  }{(1+m^2)^2} - |m|^{2\gamma-1} \right)\d x.
\]
On the one hand, the Kato inequality \cite{Kato} for the first term of the right-hand side yields
\[
   D^2 \int_0^1 (\partial^2_{xx} m)\,\sign(m) \d x \leq
    D^2 \int_0^1 \partial^2_{xx} |m| \d x =
    D^2 \Bigl[ \partial_x |m| \Bigr]_{x=0}^1 =
    D^2 \Bigl[ (\partial_{x} m)\,\sign(m) \Bigr]_{x=0}^1 = 0,
\]
where the boundary term vanishes due to the homogeneous Neumann boundary conditions.
Now, for $1/2 \leq \gamma \leq  1$, \eqref{delta_ineq} implies
\[
   \int_0^1 \left( \frac{c^2B^2  |m|  }{(1+m^2)^2} - |m|^{2\gamma-1} \right)\d x < -\delta \int_0^1 |m| \d x,
\]
so that
\[
    \tot{}{t} \int_0^1 |m| \d x < -\delta \int_0^1 |m| \d x
\]
and by Gronwall lemma we conclude exponential convergence of $\Norm{m}_{L^1(0,1)}$
(or, due to the maximum principle, any $L^q$-norm of $m$ with $q<\infty$)
to zero as $t\to\infty$.

For $-1 \leq \gamma < 1/2$, from \eqref{delta_ineq} and the behaviour
near $m\approx 0$ it follows that there exists a $\widetilde\delta > 0$ such that
\[
   \frac{c^2\|B^2\|_{L^\infty(0,1)}   |m|  }{(1+m^2)^2} - |m|^{2\gamma-1} < - \widetilde \delta \qquad\mbox{for } |m| \leq M.
\]
Therefore, we have
\[ 
    \tot{}{t} \int_0^1 |m| \d x \leq - \widetilde \delta
\]
and conclude the result with $T_0 < \Norm{m^0}_{L^1(0,1)}/ \widetilde \delta $. 
\end{proof}

\begin{remark}
For $\gamma < -1$, there exists a unique positive solution $m_b$ of the equation
\[
    \frac{c^2B^2}{(1+m^2)^2} - |m|^{2(\gamma-1)} = 0
\]
for each $cB>0$. Consequently, the claim of Lemma \ref{lem:extFT}
cannot be extended to the case $\gamma < -1$ in a straightforward way.
It can be done under the smallness assumption on the initial datum $|m^0(x)| < m_b(x)$
for all $x\in (0,1)$, but we will skip the technical details here.
\end{remark}

\subsection{Nonlinear stability analysis for $D=0$}
\label{sec:stability1D}
In Section 6.1 of \cite{HMP15} we studied the nonlinear asymptotic stability
of the 1d network formation system with $D=0$ and $\gamma\geq 1$.
We now extend that analysis to values $\gamma\geq 1/2$.

Setting $D=0$ in \eqref{meq}, we obtain
\(   \label{ODE1}
    \partial_t m = \left( \frac{c^2 B(x)^2}{(1+m^2)^2} - |m|^{2(\gamma-1)} \right) m,
\)
which we interpret as a family of ODEs for $m=m(t)$ with the parameter $x$.
Assuming that $S>0$ on $(0,1)$, we have $B(x)>0$ on $(0,1)$.

Clearly, $m=0$ is a steady state for \eqref{ODE1}; with $\gamma=1/2$ we interpret $m/|m|=0$ for $m=0$.
To find nonzero steady states, we solve the algebraic equation
\[   \label{eq_m_s}
   \frac{c^2 B(x)^2}{(1+m^2)^2} - |m|^{2(\gamma-1)} = 0,
\]
in other words, we look for the roots of $h_\gamma(m) - c^2B^2(x) = 0$ with $h_\gamma$ given by \eqref{h_gamma}.
We distinguish the cases:
\begin{itemize}
 \item $\gamma>1$: 
 The ODE \eqref{ODE1} has three stationary points: \emph{unstable} $m_0=0$ and \emph{stable} $\pm m_s$.
 Therefore, the asymptotic steady state for \eqref{ODE1} subject to the initial datum $m^0=m^0(x)$ on $(0,1)$
 is $m_s(x)\sign(m^0(x))$.
 
 \item $\gamma=1$:
 \begin{itemize}
 \item[*] If $c|B(x)| > 1$, then there are three stationary points, \emph{unstable} $m_0=0$ and \emph{stable} $\pm \sqrt{c|B(x)|-1}$.
 \item[*] If $c|B(x)| \leq 1$, then there is the only \emph{stable} stationary point $m=0$.
 \end{itemize}
 Thus, the solution of \eqref{ODE1} subject to the initial datum $m^0=m^0(x)$ on $(0,1)$
 converges to the asymptotic steady state $\chi_{\{c|B(x)| > 1\}}(x) \sign(m^0(x)) \sqrt{c|B(x)|-1}$.
 
 \item For $1/2 \leq \gamma < 1$ the picture depends on the size of $c|B(x)|$ relative to $Z_\gamma$
 defined in \eqref{Z_gamma}.
 \begin{itemize}
 \item[*] If $c|B(x)| > Z_\gamma$, then \eqref{ODE1} has five stationary points, \emph{stable} $m_0=0$,
    \emph{unstable} $\pm m_u$ and \emph{stable} $\pm m_s$, with $0 < m_u < m_s$.
 \item[*] If $c|B(x)| = Z_\gamma$, then zero is a \emph{stable} stationary point and
 there are two symmetric nonzero stationary points (attracting from $\pm\infty$ and repulsing towards zero).
 \item[*] If $c|B(x)| < Z_\gamma$, then there is the only \emph{stable} stationary point $m=0$.
 \end{itemize}


 \end{itemize}

\begin{remark}
The above asymptotic stability result for the case $1/2 \leq \gamma < 1$
shows that, at least in the case $D=0$, the assumption $c\Norm{B}_{L^\infty(0,1)} < Z_\gamma$
of Lemma \ref{lem:extFT} is optimal.
\end{remark}

\section{Stationary solutions in the multidimensional setting for $D=0$}\label{S3}
In the multidimensional setting we are able to construct
\emph{pointwise} stationary solutions of \eqref{eq1}--\eqref{eq2}.
Regarding the number of possible solutions, we obtain
the same picture as in the previous Section \ref{sec:stability1D}.
However, we are not able to provide a stability analysis.

We denote $\flux:=(I+m\otimes m)\grad p$, so that \eqref{eq1} gives
\[
    - \grad\cdot \flux = S
\]
and
\(   \label{grad_p-F}
    \grad p = (I+m\otimes m)^{-1}\flux = \left(I - \frac{m\otimes m}{1+|m|^2}\right)\flux.
\)
The activation term $c^2 (m\cdot\grad p)\grad p$ in \eqref{eq2} is then expressed in terms of $\flux$ as
\[
    c^2 (m\cdot\grad p)\grad p = c^2 \frac{m\cdot \flux}{1+|m|^2} \left(I - \frac{m\otimes m}{1+|m|^2}\right)\flux.
\]
Therefore, stationary solutions of \eqref{eq1}--\eqref{eq2} with $D=0$ satisfy
\(   \label{m-F-eq}
    c^2 \frac{m\cdot \flux}{1+|m|^2} \flux = \left( c^2 \frac{(m\cdot \flux)^2}{(1+|m|^2)^2} + |m|^{2(\gamma-1)} \right) m.
\)
Clearly, $m(x)=0$ is a solution for any $\flux\in\R^d$.
On the other hand, if $m(x)\neq 0$, then there exists a nonzero scalar $\beta(x)\in\R\setminus\{0\}$
such that $m(x) = \beta(x)\flux(x)$. Denoting $z:=\beta(x)|\flux(x)|$ and inserting into \eqref{m-F-eq}
gives
\[
    \frac{c^2 |\flux|^2}{1+z^2} = \frac{c^2 |\flux|^2 z^2}{(1+z^2)^2} + |z|^{2(\gamma-1)},
\]
which further reduces to
\(   \label{F-z-eq}
    c |\flux| = |z|^{\gamma-1}(1+z^2).
\)
We now distinguish the cases:
\begin{itemize}
 \item For $\gamma> 1$ the equation \eqref{F-z-eq} has exactly one positive solution $z>0$
 for every $|\flux|>0$.
 \item For $\gamma=1$ the equation \eqref{F-z-eq} has exactly one positive solution $z>0$
 for every $|\flux|>1/c$ and no positive solutions for $|\flux|\leq 1/c$.
 \item For $1 > \gamma \geq 1/2$ (in fact for $\gamma>-1$, but we discard the values of $\gamma<1/2$),
 if $c|\flux| > Z_\gamma$ with $Z_\gamma$ given by \eqref{Z_gamma}, there exist exactly
 two positive solutions $z_1, z_2 >0$ of \eqref{F-z-eq} for every $c|\flux|>0$.
 If $c|\flux| = Z_\gamma$, there is one positive solution $z>0$, and if $c|\flux| < Z_\gamma$,
 \eqref{F-z-eq} has no solutions.
\end{itemize}

Let us recall that in \cite{HMP15} we considered stationary solutions $(m_0, p_0)$ of \eqref{eq1}--\eqref{eq2} in the case $D=0$, $\gamma>1$.
These are constructed by fixing measurable disjoint sets
$\A_+\subseteq\Omega$, $\A_-\subseteq\Omega$ and setting
\(  \label{m0}
    m_0(x) := \left(\chi_{\A_+}(x) - \chi_{\A_-}(x)\right) c^\frac1{\gamma-1}|\grad p_0(x)|^\frac{2-\gamma}{\gamma-1} \grad p_0(x),
\)
where $p_0\in H_0^1(\Omega)$ solves the nonlinear Poisson equation
\(   \label{p0}
   -\grad\cdot\left[\left(1 + c^\frac{2}{\gamma-1}|\grad p_0(x)|^\frac{2}{\gamma-1}\chi_{\A_+\cup\A_-}(x) \right)\grad p_0(x) \right] = S,
\)
subject to - say - homogeneous Dirichlet boundary condition.
The steady states $p_0\in H^1_0(\Omega) \cap W_0^{1,2\gamma/(\gamma-1)}(\A_+\cup\A_-)$
were found as the unique minimizers of the uniformly convex and coercive functional
\[   \label{calF}
   \calF_\gamma[p] := \frac12 \int_\Omega |\grad p|^2 \d x + c^\frac{2}{\gamma-1} \frac{\gamma-1}{2\gamma} \int_{\A_+\cup\A_-} |\grad p|^{\frac{2\gamma}{\gamma-1}} \d x
     - \int_\Omega p S \d x,
\]
see Theorem 6 in \cite{HMP15}.
Let us remark that the linearized stability analysis performed in Section 6.2 of \cite{HMP15}
implies that in the case $D=0$, $\gamma>1$ the \emph{linearly} stable (in the sense of G\^{a}teaux derivative) networks
fill up the whole domain due to the necessary condition $\mbox{meas}(\A_+ \cup \A_-) = \mbox{meas}(\Omega)$.
In the 1d case, the nonlinear stability analysis of Section \ref{sec:stability1D} above implies
that the same holds also for the (nonlinearly) stable stationary solution.
On the other hand, for $\gamma=1$ the stationary solution $m_0$
\emph{must} vanish on the set $\{x\in\Omega;\; |\flux(x)|\leq 1/c\}$.
We shall return to this case below.

\subsection{Stationary solutions in the multidimensional setting for $D=0$, $1/2 \leq \gamma <1$} 
Inserting the formula $m(x) = \frac{z(x)}{|\flux(x)|}\flux(x)$ into \eqref{grad_p-F} gives
\[
    \grad p(x) &=& \flux(x) \qquad\mbox{for } m(x)=0,\\
       &=& \frac{\flux(x)}{1+z(x)^2} \qquad\mbox{for } m(x)\neq 0.
\]
Consequently, we choose mutually disjoint measurable sets $\A_0$, $\A_1$, $\A_2$
such that $\A_0\cup\A_1\cup\A_2 = \Omega$, and construct the stationary pressure gradient as
\[
    \grad p(x) = a(x,|\flux(x)|^2) \flux(x)
\]
with
\(   \label{a}
    a(x,r) &=& \chi_{\{r<\bZ\}}
       + \chi_{\{r\geq \bZ\}} \left(
          \chi_{\A_0}(x) + \frac{\chi_{\A_1}(x)}{1+z_1(r)^2}
          + \frac{\chi_{\A_2}(x)}{1+z_2(r)^2} \right),
\)
where we denoted $\bZ:=Z_\gamma^2/c^2$,
and $z_1(r)$, $z_2(r)$ are the two positive solutions of \eqref{F-z-eq} with $r=|\flux|^2$,
i.e.,
\(   \label{zr-eq}
    c^2 r = |z(r)|^{2(\gamma-1)}(1 + z(r)^2)^2.
\)
We denote by $z_1(r)$ the branch of solutions of \eqref{zr-eq} that is decreasing in $r$,
while $z_2(r)$ is increasing. Consequently,
\[
    \mbox{range}(z_1) = (0,z(\bZ)],\qquad \mbox{range}(z_2) = [z(\bZ),\infty),
\]
where we denoted $z(\bZ):=z_1(\bZ) = z_2(\bZ)$.

We assume $\int_\Omega S(x)\d x = 0$ and prescribe the homogeneous Neumann boundary condition
for $p$,
\[
    \grad p\cdot\nu = 0 \qquad\mbox{on } \partial\Omega, 
\]
where $\nu$ is the outer normal vector to $\partial\Omega$.
We perform the Helmholtz decomposition of $\flux$ as
\[
    \flux = \grad\phi + \curl\;U,
\]
where $\phi$ solves
\[
    -\laplace\phi &=& S\qquad\mbox{in }\Omega,\\
    \grad\phi\cdot\nu &=& 0\qquad\mbox{on }\partial\Omega.
\]
The identity $\curl\;\grad p = \curl(a(x,|\flux|)\flux) = 0$ gives the equation
\(  \label{curl-eq}
    \curl\left[ a(x,|\grad\phi + \curl\;U|^2)(\grad\phi + \curl\;U) \right] = 0,
\)
subject to the boundary condition $\curl\;U\cdot\nu = 0$ on $\partial\Omega$.

We define $A(x,r) := \int_0^r a(x,s) \d s \geq 0$ and for given $\phi=\phi(x)$ the functional
\(   \label{var_form}
   \mathcal{I}(U) := \frac12 \int_\Omega A(x,|\grad\phi + \curl\;U|^2) \d x.
\)
It is easily checked that \eqref{curl-eq} is the Euler-Lagrange equation
corresponding to critical points of $\mathcal{I}$.

We will now check whether $\mathcal{I}$ is convex.
For any fixed vector $\xi\in\R^d$ have
\[
   \frac12 \xi\cdot D^2_{ww} A(x, |\grad\phi+w|^2) \xi &=&
   \frac12 \sum_{i=1}^d \sum_{j=1}^d \partial^2_{w_i w_j} A(x, |\grad\phi+w|^2) \xi_i\xi_j \\
   &=& 2 \frac{\partial^2 A}{\partial r^2}(x, |\grad\phi+w|^2) (\xi\cdot(\grad\phi + w))^2
      + \frac{\partial A}{\partial r}(x, |\grad\phi+w|^2) |\xi|^2,
\]
where $\frac{\partial A}{\partial r} = \frac{\partial A}{\partial r}(x,r)$ is the derivative of $A$
with respect to the second variable.
Therefore, convexity of $\mathcal{I}(U)$ is equivalent to the condition
\(   \label{convexityA}
    2 \frac{\partial^2 A}{\partial r^2}(x, |v|^2) (\xi\cdot v)^2
    + \frac{\partial A}{\partial r}(x, |v|^2) |\xi|^2 \geq 0
    \qquad\mbox{for all } v,\,\xi\in\R^d.
\)
Using the decomposition $\xi = \lambda v + v^\perp$ with $\lambda = \frac{\xi\cdot v}{|v|^2}$ and $v^\perp\cdot v=0$
gives
\[
   \lambda^2 |v|^2 \left( 2 \frac{\partial^2 A}{\partial r^2}(x, |v|^2) |v|^2 + \frac{\partial A}{\partial r}(x, |v|^2) \right)
     + \frac{\partial A}{\partial r}(x, |v|^2) |v^\perp|^2 \geq 0.
\]
Consequently, \eqref{convexityA} is equivalent to the conditions
\[
    \frac{\partial A}{\partial r}(x, r) \geq 0,\qquad
    2 \frac{\partial^2 A}{\partial r^2}(x, r) r + \frac{\partial A}{\partial r}(x, r) \geq 0
\]
for all $x\in\Omega$, $r>0$.
Note that $\frac{\partial A}{\partial r}(x, r) = a(x,r) \geq 0$ due to \eqref{a},
so that we only need to verify the second condition.
We choose a compactly supported nonnegative test function $\psi\in C^\infty_0(\Omega,[0,\infty))$
and integrate by parts to obtain
\[
    \int_0^\infty \int_\Omega \left( 2 \frac{\partial^2 A}{\partial r^2}(x, r) r + \frac{\partial A}{\partial r}(x, r) \right) \psi(x,r)\d x\d r
     = - \int_0^\infty \int_\Omega a(x,r) \left( 2r\frac{\partial\psi}{\partial r} + \psi\right) \d x\d r.
\]
Inserting for $a=a(x,r)$ the expression \eqref{a}, we calculate
\[
    - \int_0^\infty \int_\Omega a(x,r) \left( 2r\frac{\partial\psi}{\partial r} + \psi\right) \d x\d r &=&
    \int_0^\bZ \int_\Omega \psi(x,r) \d x\d r + \int_\bZ^\infty \int_{\A_0} \psi(x,r)\d x\d r
      - 2\bZ \int_{\A_1\cup\A_2} \psi(x,\bZ)\d x \\
      &+& \int_\bZ^\infty \int_{\A_1} \frac{\psi(x,r)}{1+z_1(r)^2} \d x\d r 
      + \int_\bZ^\infty \int_{\A_2} \frac{\psi(x,r)}{1+z_2(r)^2} \d x\d r \\
    &+& \frac{2\bZ}{1+z_1(\bZ)^2} \int_{\A_1} \psi(x,\bZ) \d x 
     + \frac{2\bZ}{1+z_2(\bZ)^2} \int_{\A_2} \psi(x,\bZ) \d x \\
     &+& \frac{4}{c^2} \int_0^{z_1(\bZ)} \int_{\A_1} s^{2\gamma-1} \psi(x,z_1^{-1}(s))\d s
     - \frac{4}{c^2} \int_{z_2(\bZ)}^\infty \int_{\A_2} s^{2\gamma-1} \psi(x,z_2^{-1}(s))\d s.
\]
Note that even if we set $\A_2:=\emptyset$, the third term of the right-hand side
cannot be, in general, balanced by the other ones.
Consequently, we do not have convexity of $\mathcal{I}(U)$.

To check coercivity, we take $r > \bZ$, then
\[
   A(x,r) = \bZ + (r-\bZ)\chi_{\A_0}(x) + \chi_{\A_1}(x) \int_\bZ^r \frac{\d s}{1+z_1(s)^2} + \chi_{\A_2}(x) \int_\bZ^r \frac{\d s}{1+z_2(s)^2}.
\]
With \eqref{zr-eq} we have for both branches $z_1$, $z_2$,
\[
   c^2 r = |z|^{1(\gamma+1)}\left( 1 + 2z^{-2} + z^{-4} \right).
\]
Consequently, the increasing branch $z_2(r)\to\infty$ when $r\to\infty$
and $c^2 r \sim z_2(r)^{2(1+\gamma)}$, so that
\[
    \frac{1}{1+z_2(r)^2} \sim c^{-\frac{2}{1+\gamma}} r^{-\frac{1}{1+\gamma}}
\]
and
\[
   \int^r \frac{\d s}{1+z_2(s)^2} \sim r^\frac{\gamma}{1+\gamma}.
\]
Noting that $r=|\flux|^2=|\curl\; U + \grad\phi|^2$, the energy estimate
gives (at least) control of $|\curl\; U|^\frac{2\gamma}{1+\gamma}$.
For $1/2\leq\gamma<1$ this gives the range $2/3\leq\frac{2\gamma}{1+\gamma} < 1$
which is not enough to obtain usable coercivity estimates.
Thus, the existence of stationary points of the functional $\mathcal{I}$
remains open for $1/2\leq\gamma<1$,
however, the corresponding variational formulation \eqref{var_form} can be used
as an alternative method for numerical simulations.

\subsection{Stationary solutions in the multidimensional setting for $D=0$, $\gamma=1$} \label{sec:stable_stationary}
In the case $\gamma=1$, the stationary version of \eqref{eq2} with $D=0$ reads
\[   \label{stat_gamma=1}
   c^2(\grad p_0\otimes\grad p_0) m_0  = m_0,
\]
i.e., $m_0$ is either the zero vector or an eigenvector of the matrix $c^2(\grad p_0\otimes\grad p_0)$
with eigenvalue $1$.
The spectrum of $c^2(\grad p_0\otimes\grad p_0)$ consists of zero and $c^2|\grad p_0|^2$,
so that $m_0\neq 0$ is only possible if $c^2|\grad p_0|^2 = 1$.
Therefore, for every stationary solution there exists a measurable function $\lambda=\lambda(x)$
such that
\[
   m_0(x) = \lambda(x)\chi_{\{c^2|\grad p_0|^2=1\}}(x) \grad p_0(x)
\]
and $p_0$ solves the highly nonlinear Poisson equation
\[  \label{nonlinPoisson}
   -\grad\cdot\left[ \left(1+\frac{\lambda(x)^2}{c^2}\chi_{\{c^2|\grad p_0|^2=1\}}(x) \right) \grad p_0 \right] = S
\]
subject to the homogeneous Dirichlet boundary condition $p_0=0$ on $\partial\Omega$.

A simple consideration suggests that \emph{stable} stationary solutions of \eqref{eq1}--\eqref{eq2} with $D=0$ should be constructed as
\(
   -\grad\cdot\left[(1+a(x)^2)\grad p_0\right] &=& S,\qquad p_0 \in H^1_0(\Omega), \label{NLP1} \\
   c^2 |\grad p_0(x)|^2 &\leq& 1,\qquad \mbox{a.e. on }\Omega, \label{NLP2} \\
   a(x)^2 \left[c^2|\grad p_0(x)|^2-1 \right] &=& 0,\qquad \mbox{a.e. on }\Omega, \label{NLP3}
\)
for some measurable function $a^2=a(x)^2$ on $\Omega$ which is the Lagrange multiplier
for the condition \eqref{NLP2}.
This condition follows from the nonpositivity of the eigenvalues of the matrix $c^2(\grad p_0\otimes\grad p_0) - I$,
which is heuristically a necessary condition for linearized stability of the stationary solution
of \eqref{eq1}--\eqref{eq2} with $D=0$.
The function $\lambda=\lambda(x)$ can be chosen as $\lambda(x):=ca(x)$.

\subsubsection{Variational formulation}
We claim that solutions of \eqref{NLP1}--\eqref{NLP3} are minimizers of the energy functional
\(   \label{functional_g1}
    \mathcal{J}[p]:=\int_\Omega \left( \frac{|\grad p|^2}{2} - Sp  \right) \d x
\)
on the set $\mathcal{M}:=\{p\in H_0^1(\Omega), c^2|\grad p|^2\leq 1 \mbox{ a.e. on } \Omega\}$.

\begin{lemma}\label{lem:NLP}
Let $S\in L^2(\Omega)$. There exists a unique minimizer of the functional \eqref{functional_g1}
on the set $\mathcal{M}$.
It is the unique weak solution of the problem \eqref{NLP1}--\eqref{NLP3}
with homogeneous Dirichlet boundary conditions on $\Omega$
and with $a\in L^2(\Omega)$.
\end{lemma}

\begin{proof}
The functional $\mathcal{J}$ is convex and, due to the Poincar\'e inequality, coercive on $H_0^1(\Omega)$.
Therefore, a unique minimizer $p_0\in H_0^1(\Omega)$ exists
on the closed, convex set $\mathcal{M}$.
Clearly, \eqref{NLP1}--\eqref{NLP3} is the Euler-Lagrange system corresponding to this
constrained minimization problem, so that $p_0$ is its weak solution.
Moreover, using $p_0$ as a test function and an application of the Poincar\'e
inequality yields
\[
   \int_\Omega (1+a^2)|\grad p_0|^2 \d x &=& \int_\Omega Sp_0 \d x \\
     &\leq& \int_\Omega |\grad p_0|^2 \d x + C\int_\Omega S^2 \d x.
\]
With \eqref{NLP3} we have then
\[
   \int_\Omega a^2 \d x = c^2 \int_\Omega a^2 \grad |p_0|^2 \d x \leq c^2 C \int_\Omega S^2 \d x,
\]
so that $a\in L^2(\Omega)$.

Next, we prove that any weak solution $p_0\in H_0^1(\Omega)$, $a^2\in L^1(\Omega)$
of \eqref{NLP1}--\eqref{NLP3} is a minimizer of \eqref{functional_g1}
on the set $\mathcal{M}$.
Indeed, we consider any $q\in\mathcal{M}$ and use $(p_0-q)$ as a test function for \eqref{NLP1},
\[
    \int_\Omega (1+a^2)\grad p_0\cdot\grad(p_0-q) \d x = \int_\Omega S(p_0-q) \d x.
\]
The Cauchy-Schwarz inequality for the term $\grad p_0\cdot\grad q$ gives
\[
    \frac12 \int_\Omega (1+a^2)|\grad p_0|^2 \d x \leq \frac12 \int_\Omega (1+a^2) |\grad q|^2 \d x + \int_\Omega (p_0-q)S \d x.
\]
Moreover, \eqref{NLP3} gives $a^2 |\grad p_0|^2 = a^2/c^2$, and with $|\grad q|\leq 1/c^2$ we have
\[
    \int_\Omega \left( \frac{|\grad p_0|^2}{2} - p_0 S \right) \d x + \frac{1}{2c^2} \int a^2 \d x \leq
       \int_\Omega \frac{|\grad q|^2}{2} - qS \d x + \frac{1}{2c^2} \int a^2 \d x,
\]
so that $\mathcal{J}[p_0] \leq \mathcal{J}[q]$.

Finally, let $p_i\in H_0^1(\Omega)$, $a_i\in L^2(\Omega)$, $i=1,2$, be two weak solutions of \eqref{NLP1}--\eqref{NLP3}.
We take the difference of \eqref{NLP1} for $p_1$ and $p_2$
and test by $(p_1-p_2)$:
\[
    \int_\Omega \left[ (1+a_1^2)\grad p_1 - (1+a_2^2)\grad p_2 \right]\cdot(\grad p_1-\grad p_2) \d x = 0.
\]
We use the Cauchy-Schwarz inequality for
\[
    \int_\Omega (a_1^2 + a_2^2) (\grad p_1\cdot\grad p_2) \d x \leq
      \frac12 \int_\Omega (a_1^2 + a_2^2) |\grad p_1|^2 \d x + \frac12 \int_\Omega (a_1^2 + a_2^2) |\grad p_2|^2 \d x \leq
      \frac{1}{c^2}\int_\Omega (a_1^2 + a_2^2),
\]
where the second inequality comes from \eqref{NLP2}. 
Consequently, we have
\[
    \int_\Omega |\grad p_1-\grad p_2|^2 \d x + \int_\Omega a_1^2 |\grad p_1|^2 + a_2^2 |\grad p_2|^2 \leq
      \frac{1}{c^2}\int_\Omega (a_1^2 + a_2^2). 
\]
Finally, using \eqref{NLP3} we obtain
\[
    \int_\Omega |\grad p_1-\grad p_2|^2 \d x \leq 0
\]
and conclude that $p_1= p_2$ a.e. on $\Omega$.
\end{proof}

\begin{remark}
The gradient constrained variational problem \eqref{functional_g1}
was studied in \cite{Caffarelli} as a model for twisting
of an elastic-plastic cylindrical bar.
There it was shown that the unique solution
has $C^{1,1}$-regularity in $\Omega$;
see also \cite{Safdari1, Safdari2}.
\end{remark}

\subsubsection{A penalty method for $D=0$, $\gamma=1$}
Solutions of \eqref{NLP1}--\eqref{NLP3} can also be constructed via a penalty approximation.
Although the variational formulation used in the previous section is a short, effective and elegant
way to prove existence of solutions, we provide the alternative penalty method here,
since it provides approximations of the solution and 
since we find the related analytical techniques interesting on their own.
For the following we assume $\Omega$ to be the unit cube $(0,1)^d$ and
prescribe periodic boundary conditions on $\partial\Omega$
to discard of cumbersome boundary terms.

We consider the penalized problem
\(   \label{penalizedPoisson}
   -\grad\cdot \left[\left( 1+ \frac{(|\grad p_\eps|^2-1/c^2)_+}{\eps}\right) \grad p_\eps \right]= S, \qquad p_\eps \in \bar H^1(\Omega),
\)
where $A_+ := \max(A,0)$ denotes the positive part of $A$ and
$\bar H^1(\Omega) = \left\{u\in H^1_\mathrm{per}(\Omega),\; \int_\Omega u\d x = 0\right\}$.
Here $H^1_\mathrm{per}(\Omega)$ denotes the space of $H^1_\mathrm{loc}(\R^d)$-functions
with $(0,1)^d$-periodicity.

\begin{theorem} \label{thm:penalizedPoisson}
For any $S\in L^2(\Omega)$ with $\int_\Omega S(x)\d x =0$ there exists a unique weak solution
$p\in \bar H^1(\Omega)$ of \eqref{penalizedPoisson}.
\end{theorem}

\begproof
The functional $\calF_\eps: \bar H^1(\Omega) \to \R \cup \{\infty\}$,
\[
   \calF_\eps[p] := \frac12 \int_\Omega |\grad p|^2 \d x + \frac{1}{4\eps} \int_\Omega (|\grad p|^2-1/c^2)_+^2 \d x
     - \int_\Omega p S \d x,
\]
is uniformly convex and coercive on $\bar H^1(\Omega)$.
The classical theory (see, e.g., \cite{Evans}) provides the existence of a unique minimizer $p_\eps\in \bar H^1(\Omega)$ of $\calF_\eps$,
which is a weak solution of the corresponding Euler-Lagrange equation \eqref{penalizedPoisson}.
The uniqueness of solutions follows from the monotonicity of the function $F:\R^d\to\R^d$, $F(x)=(|x|^2-1/c^2)_+ x$.
\endproof

We will prove convergence of a subsequence of $\{p_\eps\}_{\eps>0}$, solutions of \eqref{penalizedPoisson},
towards a solution of \eqref{NLP1}--\eqref{NLP3} as $\eps\to 0$.
We introduce the notation
\(  \label{a_eps}
   a_\eps:=\frac{(|\grad p_\eps|^2-1/c^2)_+}{\eps}.
\)

\begin{theorem}\label{thm:convergenceNLP}
Let $S\in H^1(\Omega)$ with $\int_\Omega S(x)\d x = 0$, $d\leq 3$ and
$(p_\eps)_{\eps>0}\subset \bar H^1(\Omega)$ be a family of solutions of \eqref{penalizedPoisson}
constructed in Theorem \ref{thm:penalizedPoisson}, and $(a_\eps)_{\eps>0}\subset L^2(\Omega)$ given by \eqref{a_eps}.
Then there exist a subsequence of $(p_\eps,a_\eps)$ and $p\in \bar H^1(\Omega)$, $a^2\in L^2(\Omega)$ such that,
as $\eps\to 0$,
\begin{itemize}
 \item $\grad p_\eps \rightarrow \grad p$ strongly in $L^2(\Omega)$ and strongly in $L^4(\Omega)$.
 \item $(1+a_\eps)\grad p_\eps \rightharpoonup (1+a^2)\grad p$ weakly in $L^1(\Omega)$, so that \eqref{NLP1} is satisfied in the weak sense.
 \item $(|\grad p_\eps|^2-1/c^2)_+ = \eps a_\eps \to 0$ strongly in $L^2(\Omega)$,
 so that $(|\grad p|^2-1/c^2)_+ = 0$ and \eqref{NLP2} is satisfied a.e.
 \item $a_\eps \left[|\grad p_\eps|^2-1/c^2 \right] = \eps a_\eps^2 \rightarrow 0$ strongly in $L^1(\Omega)$,
 and $a_\eps \left[|\grad p_\eps|^2-1/c^2 \right] \rightharpoonup a \left[|\grad p|^2-1/c^2 \right]$ weakly in $L^1(\Omega)$ if $d\leq 3$,
 so that \eqref{NLP3} is satisfied a.e.
\end{itemize}
\end{theorem}

The proof of the above Theorem is based on the following a priori estimates:

\begin{lemma}\label{lem:NLPest1}
The family $(p_\eps)_{\eps>0}$ constructed in Theorem \ref{thm:penalizedPoisson} is uniformly bounded in $L^\infty(\Omega)$.
\end{lemma}

\begproof
This is a direct consequence of the maximum principle for \eqref{penalizedPoisson}.
\endproof

\begin{lemma}\label{lem:NLPest2}
Let $S\in H^1(\Omega)$. Then the solutions $p_\eps$ of \eqref{penalizedPoisson} satisfy
\[
  \label{estBernstein}
  \int_{\Omega} |\grad^2 p_\eps|^2 d x\leq C \Norm{\grad S}_{L^2(\Omega)}^2,\qquad
   \int_{\Omega} (1+a_\eps)|\grad^2 p_\eps|^2 \d x \leq C \Norm{\grad S}_{L^2(\Omega)}^2.
\]
\end{lemma}

\proof We use the short-hand notation $\partial_i := \partial_{x_i}$
and denote $p_i:=\partial_{i} p_\eps$ and $p_{ij} :=\partial_{ij}^2 p_\eps$.
Moreover, we set $w_\eps(x) := |\grad p_\eps(x)|$.
We use the Bernstein method and differentiate \eqref{penalizedPoisson} with respect to $x_j$:
\[
  - \partial_i \left[ \left( 1+a_\eps \right) p_{ij} \right]
  - \frac{1}{\eps} \partial_i \left[ \chi_{\{(cw-1)_+\}} (\partial_j w^2) p_i \right]
  = S_j.
\]
Then we multiply by $p_j$ and integrate by parts
\[
   \int_\Omega \left(1 + a_{\eps} \right) p_{ij}^2 \d x
   + \frac1\eps \int_{\{(cw-1)_+\}} (\partial_j w^2) p_i p_{ij} \d x = \int_\Omega S_j p_j \d x.
\]
Now we use the identity $2 p_i p_{ij} = (\partial_j w^2)$, so that the second term of the left-hand side becomes
\[
   \frac1{2\eps} \int_{\{(cw-1)_+\}} (\partial_j w^2)^2 \d x
     = \frac1{2\eps} \int_\Omega \left[ \partial_j (w^2 - 1/c^2)_+ \right]^2 \d x.
\]
Therefore, we have
\[
   \int_\Omega \left(1 + a_{\eps} \right) p_{ij}^2 \d x
   + \frac1{2\eps} \int_\Omega \left[ \partial_j (w^2 - 1/c^2)_+ \right]^2 \d x = \int_\Omega S_j p_j \d x.
\]
Using $a_\eps \geq 0$ and the nonnegativity of the second term of the left-hand side, we write
\[
   \frac12 \int_\Omega \left(1 + a_{\eps} \right) p_{ij}^2 \d x
   + \frac12 \int_\Omega p_{ij}^2 \d x \leq
    \int_\Omega \left(1 + a_{\eps} \right) p_{ij}^2 \d x
   \leq \int_\Omega S_j p_j \d x. 
\]
The claim follows by using a Cauchy-Schwarz and Poincar\'e inequality (with constant $C_\Omega$) in the right-hand side,
\[
   \int_\Omega S_j p_j \d x \leq \frac{1}{2\delta} \int_\Omega S_j^2 \d x + \frac{\delta}{2} \int_\Omega p_j^2 \d x
   \leq \frac{1}{2\delta} \int_\Omega |\grad S|^2 \d x + \frac{\delta C_\Omega}{2} \int_\Omega |\grad^2 p|^2 \d x,
\]
and choosing $\delta$ such that $\frac{\delta C_\Omega}{2} < 1/2$,
\[
   \frac12 \int_\Omega \left(1 + a_{\eps} \right) p_{ij}^2 \d x
   + C \int_\Omega p_{ij}^2 \d x
   \leq \frac{1}{2\delta} \int_\Omega |\grad S|^2 \d x, 
\]
with $C=\frac12 - \frac{\delta C_\Omega}{2} > 0$.
\endproof

\begin{lemma}\label{lem:NLPest3}
The solutions $p_\eps$ of \eqref{penalizedPoisson} and $a_\eps$ given by \eqref{a_eps} satisfy
\[
    \int_\Omega (1+a_\eps)^2 |\grad^2 p_\eps|^2 \d x \leq \Norm{S}_{L^2(\Omega)}^2, \qquad
    \int_\Omega |\grad a_\eps\cdot\grad p_\eps|^2 \d x \leq 2 \Norm{S}_{L^2(\Omega)}^2.
\]
\end{lemma}

\begproof
We again use the short-hand notation from the previous proof, and, moreover,
denote $a_i := \partial_{i} a_\eps$.
We take the square of \eqref{penalizedPoisson},
\[
    \int_\Omega S^2 \d x &=& \int_\Omega \left[\sum_{i=1}^d \partial_i [(1+a_\eps)p_i]\right]
          \left[\sum_{j=1}^d \partial_j [(1+a_\eps)p_j]\right] \d x\\
       &=& \sum_{i=1}^d \sum_{j=1}^d \int_\Omega [\partial_j [(1+a_\eps)p_i]] [\partial_i (1+a_\eps)p_j]] \d x \\
       &=& \sum_{i=1}^d \sum_{j=1}^d \int_\Omega [a_j p_i + (1+a_\eps)p_{ij}] [a_i p_j + (1+a_\eps)p_{ij}]] \d x \\
       &=& \left( \sum_{i=1}^d \int_\Omega a_i p_i \d x \right)^2
            + \sum_{i=1}^d \sum_{j=1}^d \int_\Omega (1+a_\eps)p_{ij} (a_j p_i + a_i p_j) \d x + \int_\Omega (1+a_\eps)^2 |\grad^2 p_\eps|^2 \d x.
\]
In the middle term of the last line we use the identity $\partial_{x_i} |\grad p|^2 = 2\sum_{j=1}^d p_j p_{ij}$, so that
\[
    \sum_{i=1}^d \sum_{j=1}^d \int_\Omega (1+a_\eps)p_{ij} (a_j p_i + a_i p_j) \d x =
       \sum_{i=1}^d \int_\Omega (1+a_\eps) a_i \partial_{x_i} |\grad p|^2 \d x.
\]
Morever, denoting $w:=|\grad p|^2$, we realize that $a_\eps(w) = \frac{(w-1/c^2)_+}{\eps}$
is a nondecreasing function of $w$, so that $a_i \partial_{x_i} |\grad p|^2 = (\partial_{x_i} a_\eps(w))(\partial_{x_i} w)
= a_\eps'(w)(\partial_{x_i} w)^2 \geq 0$.
Consequently, the middle term is nonnegative and we have
\[
   \int_\Omega (1+a_\eps)^2 |\grad^2 p_\eps|^2 \d x \leq \int_\Omega S^2 \d x.
\]
Now, knowing that $(1+a_\eps)\laplace p_\eps$ is bounded in $L^2(\Omega)$
due to Lemma \ref{lem:NLPest2},
and expanding the derivatives in \eqref{penalizedPoisson},
\[
    (1+a_\eps)\laplace p_\eps + \grad p_\eps\cdot\grad a_\eps = -S,
\]
we conclude
\[
    \Norm{\grad p_\eps\cdot\grad a_\eps}_{L^2(\Omega)} \leq \Norm{S}_{L^2(\Omega)} + \Norm{(1+a_\eps)\laplace p_\eps}_{L^2(\Omega)}
      \leq 2 \Norm{S}_{L^2(\Omega)}.
\]
\endproof

\begin{lemma}\label{lem:NLPest4}
The sequence $(a_\eps)_{\eps>0}$ defined in \eqref{a_eps} is uniformly bounded in $L^2(\Omega)$.
\end{lemma}

\begproof
We multiply \eqref{penalizedPoisson} by $a_\eps p_\eps$ and integrate by parts. This gives
\[
    \int_\Omega (1+a_\eps)a_\eps |\grad p_\eps|^2 \d x + \int_\Omega (1+a_\eps) p_\eps \grad p_\eps\cdot\grad a_\eps \d x = \int_\Omega S p_\eps a_\eps \d x.
\]
We write the first term as
\[
   \int_\Omega (1+a_\eps)a_\eps |\grad p_\eps|^2 \d x &=&
       \int_\Omega (1+a_\eps)a_\eps (|\grad p_\eps|^2 - 1/c^2) \d x + \frac{1}{c^2} \int_\Omega (1+a_\eps)a_\eps \d x \\
       &=& \eps \int_\Omega (1+a_\eps) a_\eps^2 \d x + \frac{1}{c^2} \int_\Omega (1+a_\eps)a_\eps \d x.
\]
Due to the nonnegativity of the first term, we have
\[
   \frac{1}{c^2} \int_\Omega a_\eps^2 \d x &\leq& \frac{1}{c^2} \int_\Omega (1+a_\eps)a_\eps \d x \leq
      \int_\Omega (1+a_\eps)a_\eps |\grad p_\eps|^2 \d x \\
      &\leq& 
      \Norm{S}_{L^2(\Omega)}\Norm{p_\eps}_{L^\infty(\Omega)}\Norm{a_\eps}_{L^2(\Omega)}
         + \Norm{\grad p_\eps\cdot\grad a_\eps}_{L^2(\Omega)} \Norm{p_\eps}_{L^\infty(\Omega)} \Norm{1+a_\eps}_{L^2(\Omega)} \\
      &\leq&
      \frac{1}{2\delta}\Norm{S}_{L^2(\Omega)}^2\Norm{p_\eps}_{L^\infty(\Omega)}^2 + \frac{\delta}{2}\Norm{a_\eps}_{L^2(\Omega)}^2 \\
       && + \frac{1}{2\delta} \Norm{\grad p_\eps\cdot\grad a_\eps}_{L^2(\Omega)}^2 \Norm{p_\eps}_{L^\infty(\Omega)}^2 +
           \delta \left( |\Omega|^2 + \Norm{a_\eps}_{L^2(\Omega)}^2 \right)
\]
for any $\delta>0$.
Then, an application of Lemmata \ref{lem:NLPest1} and \ref{lem:NLPest3} gives a constant $C>0$ such that
\[
    \frac{1}{c^2} \Norm{a_\eps}_{L^2(\Omega)}^2 \leq C(1+\delta^{-1}) + \frac{3\delta}{2}\Norm{a_\eps}_{L^2(\Omega)}^2
\]
and choosing $\delta>0$ small enough, we conclude.
\endproof

Now we are ready to prove Theorem \ref{thm:convergenceNLP}:

\begproof
 From Lemma \ref{lem:NLPest2} we conclude that $\grad^2 p_\eps$ is uniformly bounded in $L^2(\Omega)$,
 so, in $d\leq 3$ spatial dimensions, $\grad p_\eps$ converges strongly in $L^4(\Omega)$
 to $\grad p$ due to the compact Sobolev embedding.
 
 Since $a_\eps$ is bounded in $L^2(\Omega)$ by Lemma \ref{lem:NLPest4}, there exists a
 weakly converging subsequence to $a^2\in L^2(\Omega)$.
 Thus, due to the strong convergence of $\grad p_\eps$, the product $(1+a_\eps)\grad p_\eps$
 converges weakly in $L^1(\Omega)$ to $(1+a^2)\grad p$.
 
 Clearly, $\eps a_\eps\to 0$ strongly in $L^2(\Omega)$.
 Moreover, due to the inequality $|a_+ - b_+| \leq |a-b|$, we have
 \[
    0 \leq \int_\Omega \left| (|\grad p_\eps|^2 - 1/c^2)_+ - (|\grad p|^2 - 1/c^2)_+ \right| \d x
      \leq \int_\Omega \left| |\grad p_\eps|^2 - |\grad p|^2 \right| \d x,
 \]
 so that the strong convergence of $\grad p_\eps$ in $L^2(\Omega)$ implies
 \[
     \eps a_\eps = (|\grad p_\eps|^2 - 1/c^2)_+ \to (|\grad p|^2 - 1/c^2)_+ \qquad\mbox{strongly in } L^1(\Omega).
 \]
 Consequently, $(|\grad p|^2 - 1/c^2)_+ = 0$ a.e.
 
 Clearly, $\eps a_\eps^2\to 0$ strongly in $L^1(\Omega)$.
 The strong convergence of $\grad p_\eps$ in $L^4(\Omega)$ and weak convergence (of a subsequence of)
 $a_\eps$ in $L^2(\Omega)$ implies
 \[
    \eps a_\eps^2 = a_\eps \left[|\grad p_\eps|^2-1/c^2 \right] \rightharpoonup a^2 \left[|\grad p|^2-1/c^2 \right]
    \qquad\mbox{weakly in } L^1(\Omega).
 \]
 Consequently, $a^2 \left[|\grad p|^2-1/c^2 \right] = 0$ a.e.
\endproof

Finally, let us note that uniqueness of solutions of the system \eqref{NLP1}--\eqref{NLP3}
was already proved in Lemma \ref{lem:NLP}.

\subsection{Stationary solutions via the variational formulation for $D=0$, $1/2 \leq \gamma < 1$}\label{sec:instable_gamma_less_1}
Let us recall that in \cite{HMP15} we constructed stationary solutions $(m_0, p_0)$ of \eqref{eq1}--\eqref{eq2} in the case $\gamma>1$, $D=0$
by using \eqref{m0} and employing the variational formulation of the nonlinear Poisson equation \eqref{p0}.
Clearly, this approach fails for $\gamma<1$ due to the singularity of the term $|\grad p_0(x)|^\frac{2-\gamma}{\gamma-1} \grad p_0(x)$
at $|\grad p_0|=0$ and the resulting non-boundedness from below of the associated functional.
However, stationary solutions can be constructed by ``cutting off'' small values of $|\grad p_0|$.
For simplicity, we set the activation parameter $c^2:=1$ in this section.

We fix a measurable set $\A\subset\Omega$ and a constant $\alpha>0$ to be specified later, and define the stationary solution
of \eqref{eq1}, \eqref{eq2} for $D=0$:
\(   \label{m0_cutoff}
   m_0 = \chi_{\A} \chi_{\{|\grad p_0|>\alpha\}} |\grad p_0|^\frac{2-\gamma}{\gamma-1} \grad p_0,
\)
where $p_0$ solves
\[
   -\grad\cdot \left[ (1+m_0\otimes m_0) \grad p_0 \right] = S.
\]
This is the Euler-Lagrange equation corresponding to the functional $\calF_\alpha: H^1_0(\Omega)\to \R$,
\(    \label{fF}
    \calF_\alpha[p] = \int_\Omega \frac{|\grad p|^2}{2} + \frac{\gamma-1}{2\gamma} \chi_{\A}(x) \left( |\grad p|^\frac{2\gamma}{\gamma-1}
      - \alpha^\frac{2\gamma}{\gamma-1} \right)_- \d x.
\)
We examine the convexity of $\calF$ in dependence on the value of $\alpha$.
Defining $F:\R^d\to\R$,
\[
   F(\xi):=\frac{|\xi|^2}{2} + \frac{\gamma-1}{2\gamma} \left( |\xi|^\frac{2\gamma}{\gamma-1}-\alpha^\frac{2\gamma}{\gamma-1} \right)_-,
\]
we calculate the Hessian matrix
\[
    D^2 F(\xi) = F''(|\xi|)\frac{\xi\otimes \xi}{|\xi|^2} + F'(|\xi|)\left( \frac{I}{|\xi|} - \frac{\xi\otimes \xi}{|\xi|^3} \right).
\]
This has the eigenvectors $\xi$ and $\xi^\perp$ and a quick inspection reveals that
$F$ is convex as a function of $\xi$ if and only if $F'(|\xi|)\geq 0$ and $F''(|\xi|)\geq 0$.
Writing $r:=|\xi|$, we have
\[
   F'(r) &=& r + r^\frac{\gamma+1}{\gamma-1},\qquad r>\alpha,\\
     &=& r,\qquad r<\alpha,
\]
and
\[
   F''(r) &=& 1 + \delta(r-\alpha) \alpha^\frac{\gamma+1}{\gamma-1} + \frac{\gamma+1}{\gamma-1} r^\frac{2}{\gamma-1},\qquad r\geq \alpha,\\
     &=& 1,\qquad r<\alpha.
\]
Thus, $F$ is a uniformly convex function on $\R^d$ if and only if $\alpha>\alpha_\gamma$ with
\(  \label{alpha_gamma}
    \alpha_\gamma := \left( \frac{1-\gamma}{1+\gamma} \right)^\frac{\gamma-1}{2}.
\)

\begin{lemma}\label{lem:unst_contr}
Let $1/2 \leq \gamma < 1$ and $\alpha>\alpha_\gamma$ with $\alpha_\gamma$ given by \eqref{alpha_gamma}.
Then the functional \eqref{fF} is coercive and uniformly convex on $H_0^1(\Omega)$
and the unique minimizer $p_0$ with $m_0$ given by \eqref{m0_cutoff} is a stationary solution of \eqref{eq1}--\eqref{eq2}.
\end{lemma}

Unfortunately, in the spatially one-dimensional case we are able to show that the above construction delivers
the nonlinearly \emph{unstable} steady states (as long as $m_0\not\equiv 0$), so these solutions will never be observed
in the long time limit of the system \eqref{eq1}--\eqref{eq2}.
For simplicity, we set $\Omega:=(0,1)$ and $\A=\emptyset$.
Let us recall that the 1d problem with $D=0$ reduces to the ODE family
\( \label{m1d}
    \partial_t m = \left( \frac{B(x)^2}{(1+m^2)^2} - |m|^{2(\gamma-1)} \right) m.
\)
As calculated in Section \ref{sec:stability1D}, the nonlinearly \emph{stable}
stationary solution of \eqref{m1d} for $1/2 < \gamma <1$ is
\[
    m(x) &=& 0,\qquad\mbox{if } |B(x)|\leq Z_\gamma,\\
    m(x) &\in& \{0, \pm m_s(x)\},\qquad\mbox{if } |B(x)|> Z_\gamma,
\]
where $Z_\gamma$ is given by \eqref{Z_gamma} and $m_s>0$ is the \emph{largest} solution of
\[
    \frac{B(x)^2}{(1+m^2)^2} = |m|^{2(\gamma-1)},
\]
with $B(x)=\int_0^x S(y) \d y >0$.
Note that for $|B(x)|>Z_\gamma$ the above algebraic equation has
four solutions $\pm m_u$, $\pm m_s$ with $0 < m_u < m_s$,
and $m_u$ is the unstable, $m_s$ stable solution for \eqref{m1d}.
Moreover, note that $$Z_\gamma = \frac{2}{1+\gamma}\alpha_\gamma,$$
so that $Z_\gamma > \alpha_\gamma$.

Now, let $(m_0, p_0)$ be the solution constructed in Lemma \ref{lem:unst_contr},
i.e., $p_0$ is the unique minimizer of \eqref{fF} and $m_0$ is given by \eqref{m0_cutoff}.
Clearly, to have a \emph{nonzero stable} state $m_0=m_0(x)$ for some $x\in\Omega$, it is necessary that
\[
   |B(x)|>Z_\gamma \qquad\mbox{and}\qquad |\partial_x p_0(x)| > \alpha > \alpha_\gamma.
\]
Moreover, if $m_0(x)\neq 0$, we have the formulas
\[
   |m_0| = |\partial_x p_0|^\frac{1}{\gamma-1}  \qquad\mbox{and}\qquad
      |B(x)|=|m_0|^{\gamma-1}(1+|m_0|^2) = |\partial_x p_0|\left( 1 + |\partial_x p_0|^\frac{2}{\gamma-1}\right).
\]
Thus, defining $f_\gamma(u) := u\left(1+u^\frac{2}{\gamma-1}\right)$ for $u>0$,
we have $|B(x)| = f_\gamma(|\partial_x p_0|)$.
The function $f_\gamma(u)$ has a unique strict minimum on $\R^+$ at $u=\alpha_\gamma>0$
and $f_\gamma(\alpha_\gamma) = Z_\gamma$. Since $\lim_{u\to 0} f_\gamma(u) = \lim_{u\to+\infty} f_\gamma(u) = +\infty$,
for each $|B(x)|>Z_\gamma$ there exist $0 < u_1 < \alpha_\gamma < u_2$
such that $f_\gamma(u_1) = f_\gamma(u_2) = |B(x)|$.
Clearly, $u_1$, $u_2$ correspond to the nonzero steady states $m_u$, $m_s$ of \eqref{m1d}, and since
$|m_0| = |\partial_x p_0|^\frac{1}{\gamma-1}$ is a decreasing function of $u=|\partial_x p_0|$
and $m_u < m_s$, the \emph{unstable} steady state $m_u$ corresponds to $u_2$,
while the \emph{stable} steady state $m_s$ corresponds to $u_1$.
Since, by construction, we have to choose $u=|\partial_x p_0|>\alpha_\gamma$ in order to have $m_0(x)\neq 0$,
we are in fact choosing $u_2$ and thus the \emph{unstable} steady state $m_u$.

\section{Numerical Method and Examples}\label{sec:Num}

The model has the ability to generate fascinating patterns and we illustrate this with numerical experiments performed in two space dimensions using a Galerkin framework.
These interesting patterns show up if the diffusivity in the system is low, i.e. $D\ll 1$, and the pressure gradient is large.
In this context let us also mention  \cite{MMM,Hu}, where numerical simulations for Eqs.~\eqref{eq1}--\eqref{eq2} have been presented.

Furthermore, we want to demonstrate that the results of the analysis in one dimension are also relevant for the two dimensional setting. 
For instance, for $\gamma<1/2$ we are interested in extinction in finite time of the solution, cf. Section~\ref{sec:finite_time_breakdown}, and for $1/2\leq \gamma<1$ we demonstrate instability of solutions constructed in Section~\ref{sec:instable_gamma_less_1}.

\subsection{Mixed variational formulation}\label{sec:mixed}
Since we are interested in the case $D\ll 1$, it turns out to be useful to reformulate Eq.~\eqref{eq2} as a mixed problem.
Consequently, setting $\sigma=\nabla m$, we consider
\begin{align*}
-\div [( rI + m \otimes m) \nabla p ] &= S\quad\text{in }\Omega \times (0,T),\\
              p &=0\quad\text{ on }\Gamma \times (0,T),\\
       \nu\cdot ( rI + m \otimes m) \nabla p &=0\quad\text{ on }\partial\Omega\setminus\Gamma\times (0,T),\\
\partial_t m - D^2 \div \sigma &= c^2 (\nabla p \otimes \nabla p) m - |m|^{2(\gamma-1)} m\quad\text{in } \Omega\times (0,T),\\
 \sigma-\nabla  m &= 0 \quad\text{in } \Omega\times (0,T),\\
  m &=0\quad\text{on }\partial\Omega\times (0,T),
\end{align*}
with $m(t=0)=m^0$ in $\Omega$.
Here, $\Gamma\subset\partial\Omega$ denotes the Dirichlet part of the boundary, and we denote by $H^1_{0,\Gamma}(\Omega)=\{p\in H^1(\Omega): p_{\mid\Gamma}=0\}$ the space of Sobolev functions vanishing on $\Gamma$. Additionally, we need the space $H(\divergence)=\{\mu \in L^2(\Omega)^2:\ \div \mu\in L^2(\Omega)\}$.
As a starting point for our Galerkin framework we consider the following weak formulation: Find $(p,m,\sigma)\in L^\infty(0,T;H^1_{0,\Gamma}(\Omega))\times L^2(0,T;L^2(\Omega)^2)\times L^2(0,T;H(\divergence)^2)$ such that 
\begin{align*}
 \int_\Omega ( rI + m \otimes m) \nabla p \cdot \nabla q \d x &= \int_\Omega S q \d x,\\
 \int_\Omega \partial_t m\cdot v \d x - \int_\Omega D^2 \div\sigma\cdot v \d x &= \int_\Omega f_{\gamma,c}(m,\nabla p)\cdot v \d x,\\
 \int_\Omega \sigma\cdot \mu \d x +\int_\Omega m \cdot\div \mu \d x&= 0,
\end{align*}
for all $(q,v,\mu)\in H^1_{0,\Gamma}(\Omega)\times L^2(\Omega)^2\times H(\divergence)^2$, and $m(t=0)=m^0$ in $\Omega$. Here, we use the abbreviation
$$
 f_{\gamma,c}(m,\nabla p)= c^2 (\nabla p \otimes \nabla p) m - |m|_\rho^{2(\gamma-1)} m,
$$
where $|m|_\rho=\sqrt{m_{1}^2 + m_2^2 +\rho}$ is a regularized absolute value with regularization parameter $\rho\geq 0$.
Any strong solution $(m,p)$ to \eqref{eq1}--\eqref{eq2} satisfying the above boundary conditions yields a solution to the flux based weak formulation in case $\rho=0$ and $\gamma>1/2$.
Homogeneous Neumann boundary conditions for $m$ result in homogeneous Dirichlet boundary conditions for $\sigma\cdot \nu$ on $\partial \Omega$, and the function space for $\sigma$ has to be adapted accordingly.

\subsection{Space discretization}
To obtain a space discretization, we let $\{\T_h\}$ be a family of regular quasi-uniform triangulations of $\Omega$ with $h=\max_{T\in\T_h} h_T$ and $h_T=\sqrt{|T|}$ for all $T\in\T_h$.
For the approximation of the pressure $p$ we choose standard Lagrangian finite elements, i.e. continuous, piecewise linear functions
$$
 P_h =\{ v_h \in C^0(\overline{\Omega}): v_{h\mid T}\in \Pol_1(T)\ \forall T\in \T_h, v_{h\mid \Gamma}=0 \} \subset H^1_{0,\Gamma}(\Omega).
$$
For the approximation of the conductance vector $m$ we choose piecewise constant functions
$$
 M_h =\{ v_h \in L^2(\Omega): v_{h\mid T}\in \Pol_0(T)\ \forall T\in \T_h\},
$$
and for the approximation of $\sigma$ we choose lowest order Raviart-Thomas elements
$$
 V_h =\{ v_h\in H(\divergence) : v_{h\mid T}\in RT_0(T)=\Pol_0(T)+x \Pol_0(T)\ \forall T\in \T_h\}.
$$
The resulting Galerkin approximation is then as follows.
Find $(p_h,m_h,\sigma_h)\in L^\infty(0,T;P_h)\times L^2(0,T;M_h^2)\times L^2(0,T;V_h^2)$ such that for a.e. $t\in (0,T]$
\begin{align*}
 \int_\Omega ( rI + m_h(t) \otimes m_h(t)) \nabla p_h(t) \cdot \nabla q_h \d x &= \int_\Omega S q_h \d x,\\
\int_\Omega \partial_t m_h(t)\cdot v_h \d x - \int_\Omega D^2 \div\sigma_h(t) \cdot v_h \d x &= \int_\Omega f_{\gamma,c}(m_h(t),\nabla p_h(t)) \cdot v_h \d x,\\
 \int_\Omega \sigma_h(t)\cdot \mu_h \d x+\int_\Omega m_h(t)\cdot \div \mu_h \d x&= 0,
\end{align*}
for all $(q_h,v_h,\mu_h)\in P_h\times M_h^2\times V_h^2$, and $m_h(t=0)=m_h^0$ in $\Omega$. Here, $m_h^0$ denotes the $L^2$-projection of $m^0$ onto $M_h$.
Assuming sufficient regularity of solutions, the method applied to the stationary problem
is of first order in the $L^2(\Omega)$-norm and also first order in the $L^2(\Omega)$-norm for $\grad p$ and $\sigma$.
The $L^2(\Omega)$-projection of $m$ is approximated with second order in $L^2(\Omega)$ if $f_{\gamma,c}(m,\nabla p)\in H^1(\Omega)$.
Our analytical results do not provide such regularity; however,
even for regular solutions the error estimates are in practice not very helpful since the constants depend on norms of derivatives
of solutions which are locally very large in the small diffusion - large activation regime.
For details on the approximation spaces, mixed finite elements and corresponding error estimates see for instance \cite{Boffi}.

\subsection{Time discretization}
For the discretization of the time variable let $0=t^0<t^1<\ldots<t^K=T$ denote a partition of $[0,T]$. By $m_h^k\approx m_h(t^k)$, $p_h^k\approx p_h(t^k)$ and $\sigma_h^k \approx \sigma_h(t^k)$, $0\leq k\leq K$, we denote the corresponding approximation in time, which is obtained by solving the following implicit-explicit (IMEX)
first-order Euler scheme
\begin{align}
 \int_\Omega (r I + m_h^k \otimes m_h^k) \nabla p_h^k \cdot \nabla q_h \d x &= \int_\Omega S q_h \d x, \label{eq:IMEX1}\\
 \int_\Omega m_h^{k+1} \cdot v_h  - \delta^{k+1} D^2 \div \sigma_h^{k+1}\cdot v_h \d x &=  \int_\Omega \big(m_h^{k}+\delta^{k+1} f_{\gamma,c}(m^k_h,\nabla p_h^k)\big)\cdot v_h \d x,\label{eq:IMEX2}\\
 \int_\Omega \sigma_h^{k+1}\cdot \mu_h\d x +\int_\Omega m_h^{k+1}\cdot \div \mu_h \d x&= 0,\label{eq:IMEX3}
\end{align}
for all $(q_h,v_h,\mu_h)\in P_h\times M_h^2\times V_h^2$, and $\delta^{k+1}=t^{k+1}-t^k$.
We note that for $D=0$ Eq.~\eqref{eq2} is an ODE, and \eqref{eq:IMEX2} amounts to an explicit Euler scheme for approximating $m_h(t)$ on each triangle $T\in\T_h$.
In addition, there is no coupling between the different triangles in this case, and consequently no numerical diffusion is introduced into the system.

The discrete counterpart of the energy defined in \eqref{energy} is defined as follows
$$
 \E_h(m^k_h)=\frac{1}{2} \int_{\Omega} D^2 |\sigma_h^k|^2 + \frac{|m_h^k|_\rho^{2\gamma}}{\gamma} + c^2 |m^k_h\cdot\nabla p^k_h|^2 + c^2|\nabla p_k^h|^2 \d x.
$$
Our main guideline for obtaining a stable scheme is to ensure that $\E_h(m^{k+1}_h)\leq \E_h(m^{k}_h)$ for all $k\geq 0$, which is inspired by but weaker than \eqref{energy_ineq}.
We choose an adaptive time-stepping according to the following rule: $\delta^1=1/(2 c^2 \| |\nabla p^0_h|\|_\infty^2)$, $t_1=t_0+\delta^1$, $\delta^2=\delta^1$. Let $\delta^k$ be given.
If $\delta^k\in (1/(20 c^2 \| |\nabla p^k_h|\|_\infty^2),9/(10c^2 \| |\nabla p^k_h|\|_\infty^2))$ then $\delta^{k+1}=\delta^k$, otherwise set $\delta^{k+1}=1/(2 c^2 \| |\nabla p^k_h|\|_\infty^2)$, and $t_{k+1}=t^k+\delta^{k+1}$. 
Moreover, we let $\delta^k$ be sufficiently small.
This choice of time-step is motivated by the solution of the ODE system 
\begin{align*}
 \tilde m_t = \begin{pmatrix} 0 & 0 \\0 & c^2|\nabla p|^2\end{pmatrix} \tilde m,\qquad \tilde m(0)=\tilde m_0,
\end{align*}
which is obtained from Eq.~\eqref{eq2} with $D=0$ and no relaxation term through diagonalization. Assuming $\nabla p(t)$ does not depend on $t$, the solution is given by
\begin{align*}
 \tilde m_1(t) =\tilde m_{0,1},\quad \tilde m_2(t) = \exp(c^2|\nabla p|^2 t) \tilde m_{0,2}.
\end{align*}
Since $c^2\nabla p_h^k\otimes\nabla p_h^k$ is positive semi-definite the explicit Euler method is unstable for all choices of $\delta^k$. Stability might however be retained through the relaxation term as soon as $c^2\nabla p_h^k\otimes\nabla p_h^k-|m_h^k|^{2(\gamma-1)}I$ is negative definite, i.e. if $c^2|\nabla p_h^k|^2< |m_h^k|^{2(\gamma-1)}$; cf. Section~\ref{sec:stable_stationary}.
Besides this stability issue there is an additional linearization error by treating $\nabla p_h^k$ explicitly. Hence, if $\nabla p_h^k$ is changing rapidly, then $\delta^k$ should be sufficiently small to obtain a reasonable accuracy.
A detailed investigation of stable and accurate time-stepping schemes is however out of the scope of this paper and is left for further research; let us mention \cite{Koto,Ruuth} for IMEX schemes in the context of reaction-diffusion equations.

%

\subsection{Setup}\label{sec:setup}
\begin{figure} \centering
 \includegraphics[width=.24\textwidth]{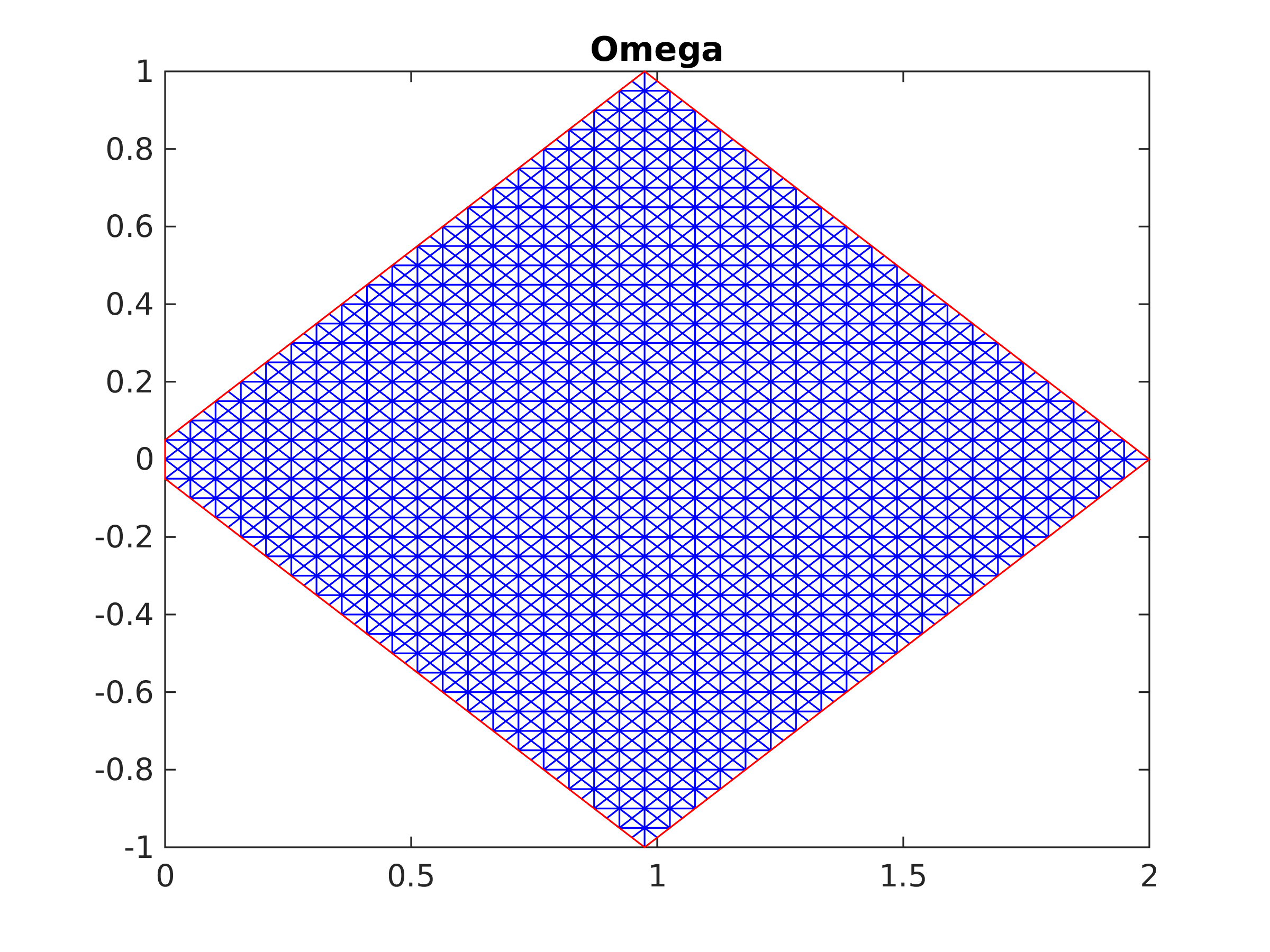}
 \includegraphics[width=.24\textwidth]{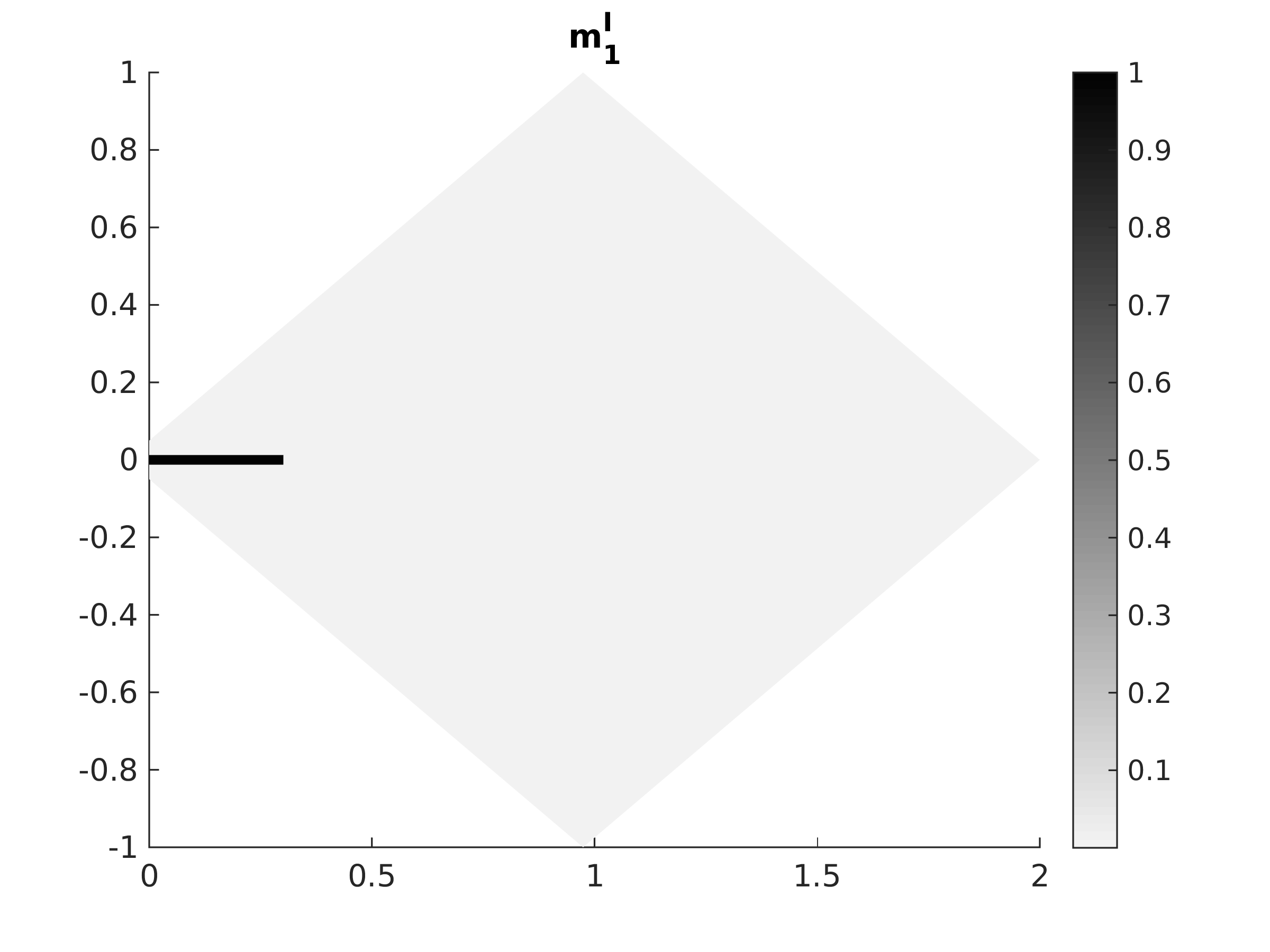}
 \includegraphics[width=.24\textwidth]{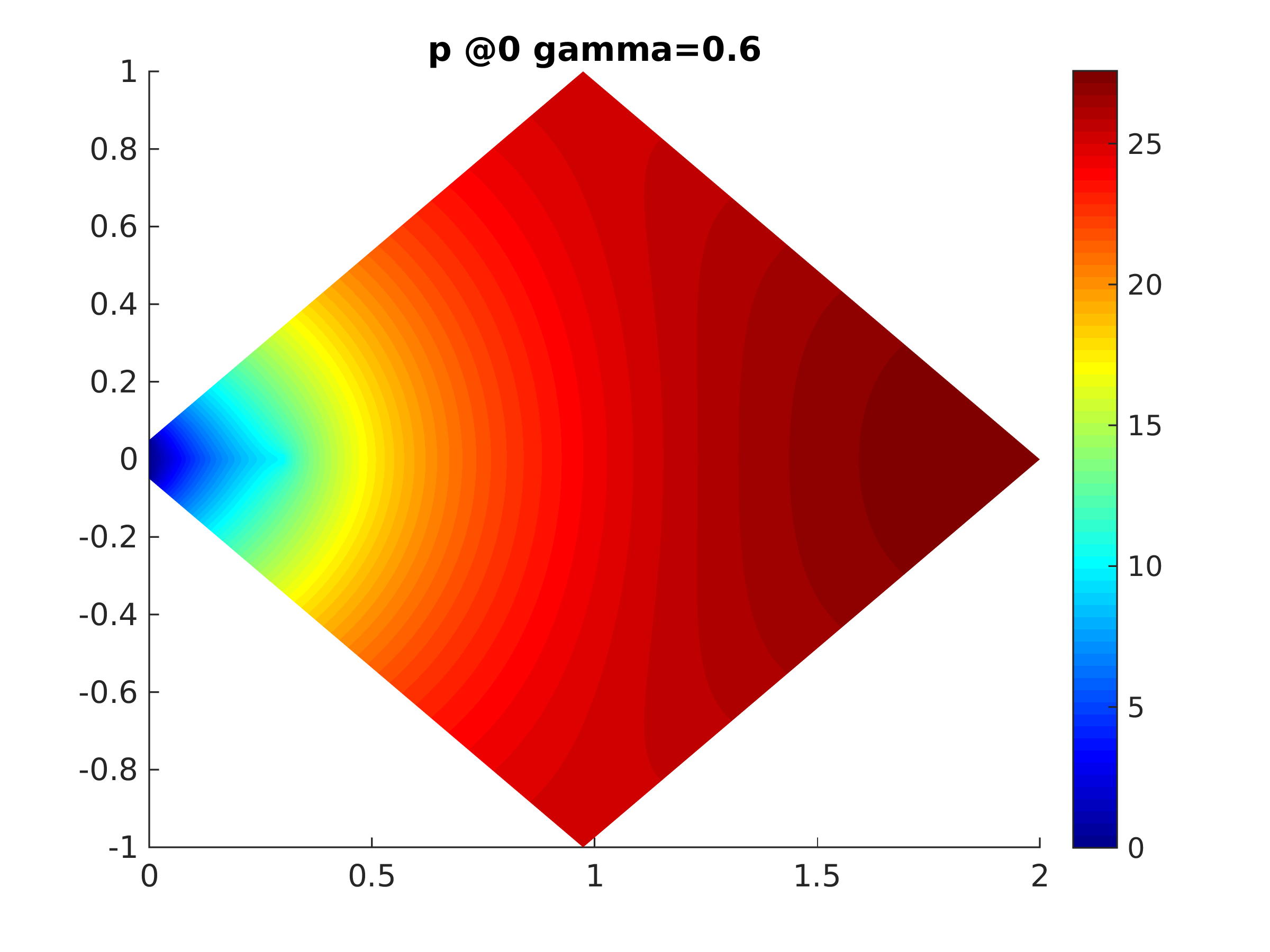}
  \includegraphics[width=.24\textwidth]{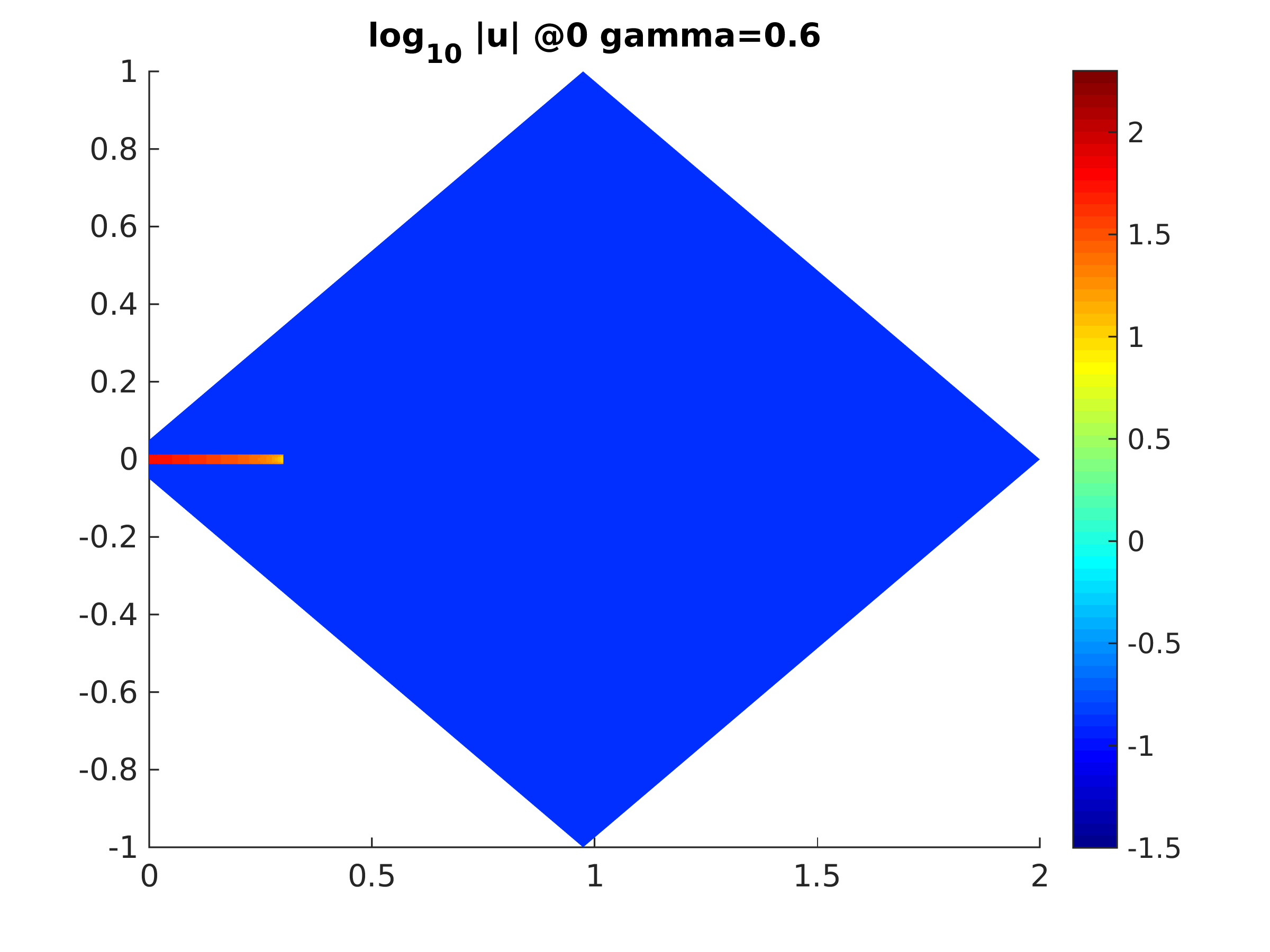}
  \caption{From left to right: Triangulation $\mathcal{T}_h$ of $\Omega$ with $1,678$ vertices and $3,196$ triangles; initial datum $m_1^0$, while $m_2^0=0$; initial pressure $p^0_h$; decadic logarithm of the absolute value of the initial velocity $|u^0_h|$. Note that $p^0_h$ and $|u^0_h|$ are the same for all values of $\gamma$.
 \label{fig:domain}}
\end{figure}

As a computational domain we consider a diamond shaped two-dimensional domain with one edge cut, see Figure~\ref{fig:domain}. 
We use a refined triangulation with $102,905$ vertices and $204,544$ triangles, which corresponds to $h\approx 0.0032$. 
The Dirichlet boundary is defined as $\Gamma=\{x\in\RR^2: x_1=0\}\cap\partial\Omega$.
If not stated otherwise, we let
\begin{align*}
 S=1,\quad r=\frac{1}{10},\quad c=50,\quad D=\frac{1}{1000},\quad \gamma=\frac{1}{2},\quad \rho=10^{-12},
\end{align*}
where $\rho$ is the regularization parameter defined in Section~\ref{sec:mixed}, and define the initial datum as
\begin{align*}
 m^0_1(x)=\begin{cases} 1, & x\leq 0.3 \text{ and } |y| \leq 0.0125,\\ 0, &\text{ else,} \end{cases}\qquad m^0_2=0.
\end{align*}
The main quantity of interest is the discrete velocity defined as
\begin{align*}
 u_h^k = (rI + m_h^k\otimes m_h^k)\nabla p_h^k,
\end{align*}
see also Section~\ref{S3}.
The initial velocity $u_h^0$ and the initial pressure $p_h^0$ do not depend on $\gamma$ or $D$, and they are depicted in Figure~\ref{fig:domain}.
Since the numerical simulation is computationally expensive, we could not compute a stationary state in many examples below. However, in order to indicate that the presented solutions are near a stationary state, we define the stationarity measures
\begin{align*}
 \E_{h,t}^{k} = \frac{\E_h(m^k_h)-\E_h(m^{k-1}_h)}{\delta^k}\quad\text{and}\quad m_{h,t}^k=\frac{\|m^k_h-m^{k-1}_h\|_{L^2(\Omega)}}{\delta^k}.
\end{align*}
Furthermore, we define the quantity
\begin{align*}
 s_k = \frac{\|u_h^k\|_{L^2(\Omega)}}{\|u_h^k\|_{L^1(\Omega)}},
\end{align*}
which measures the sparsity of the network.
In order to demonstrate the dependence of the solution on the different parameters in the system we first present some simulations for varying parameter values.

\subsection{Varying $D$}\label{sec:varying_D}
The proliferation of the network and its structure is crucially influenced by diffusion. 
In the limit of vanishing diffusion $D=0$ the support of the conductance vector cannot grow.
If diffusion is too large, interesting patterns will not show up in the stationary network.
In the following we investigate the influence of different values for $D\in \{\frac{1}{2},\frac{1}{10},\frac{1}{100},\frac{1}{1000}\}$ on the network formation,
while keeping $\gamma=1/2$ fixed.
For $D\geq 1/10$, the obtained velocities are dominated by diffusion and no fine scale structures appear in the network, see Figures~\ref{fig:evolution_D05} and~\ref{fig:evolution_D01}.
For $D=1/100$ the resulting velocity is still diffusive, but some large scale structure is visible, see Figure~\ref{fig:evolution_D001}.
Decreasing the diffusion coefficient even further to $D=1/1000$, the network builds fine scale structures, see Figure~\ref{fig:evolution_D0001}.
We note that the velocity is not near a stationary state here.
For this very small $D=1/1000$, we have to be careful in interpreting the results. Our simulations have shown a strong mesh dependence for this case,
which is not apparent for $D\geq 1/2$.
We are not able to fully explain this behavior, but the comparison on different meshes for large diffusion and moderate values of $c$, which makes in turn $c\nabla p$ moderate, suggest that the mesh is too coarse to be able to resolve the diffusion process properly in the presence of strong activation. Here, one should use finer meshes to resolve this issue, which however also leads to prohibitively long computation times. 
A modification of the existing scheme to cope with this issue is left to further research. 

Nonetheless, we believe that the velocities presented here are qualitatively correct, as they structurally show the right behavior, i.e. the smaller the diffusion $D$ the finer the scales in the network are, and the higher the sparsity index $s_k$ is, see also Table~\ref{tab:sparsity_D}.
Let us remark that the closer $\gamma$ is to $1/2$, the sparser the structures should be in a stationary state, which complies with the well-known fact that $L^1$-norm minimization promotes sparse solutions (note that the metabolic energy term in \eqref{energy} becomes a multiple
of $\Norm{m}_{L^1(\Omega)}$ for $\gamma=1/2$).
Furthermore, even though the results in Figure~\ref{fig:evolution_D0001} are quantitatively very different, they possess qualitatively the same properties; namely the thickness of the primary, secondary and tertiary branches. Let us again emphasize that the results of Figure~\ref{fig:evolution_D0001} are far from being stationary, and the networks are likely to change their structure when further evolving.


\begin{figure} \centering
 \includegraphics[width=.32\textwidth]{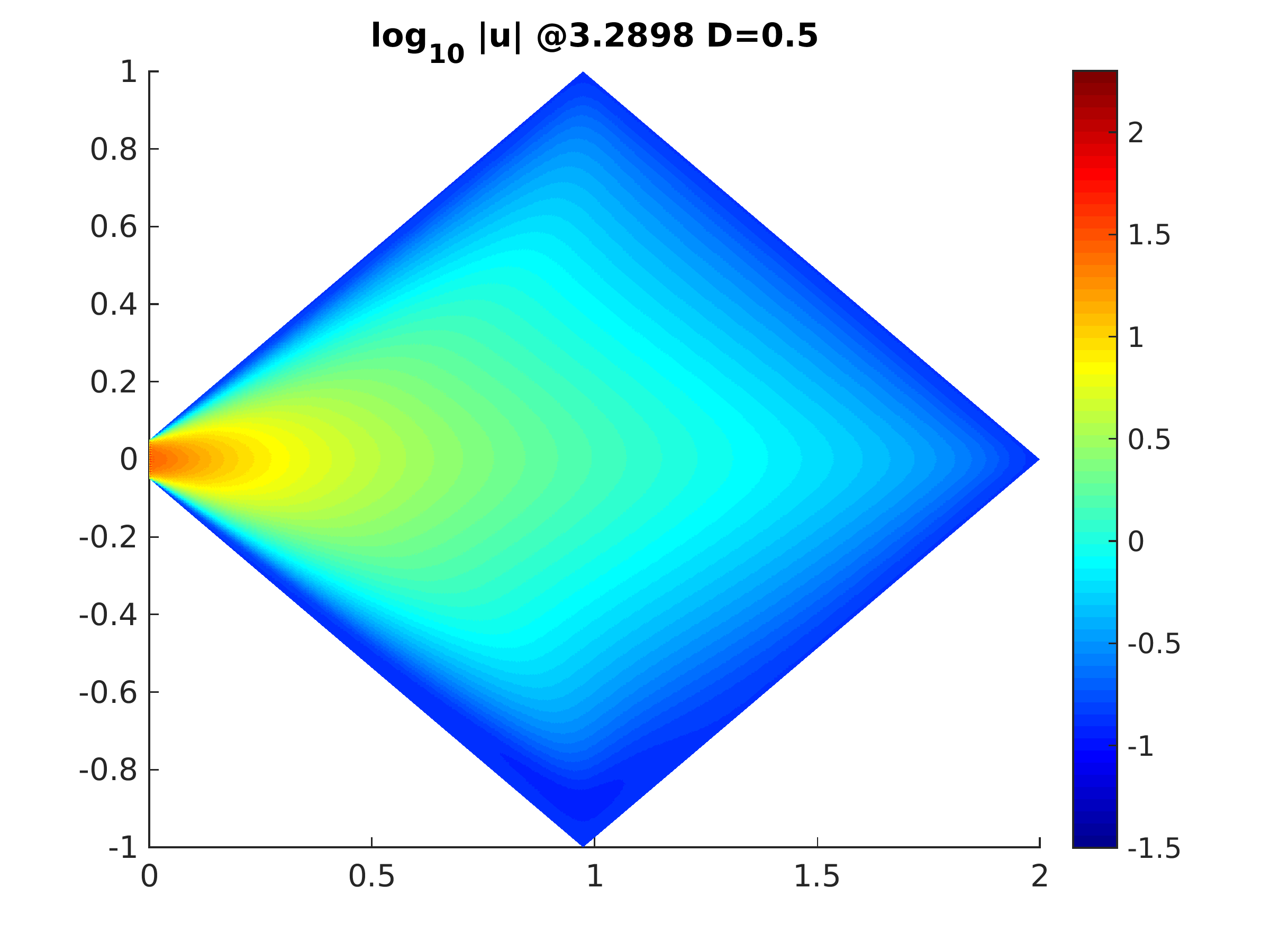}
 \includegraphics[width=.32\textwidth]{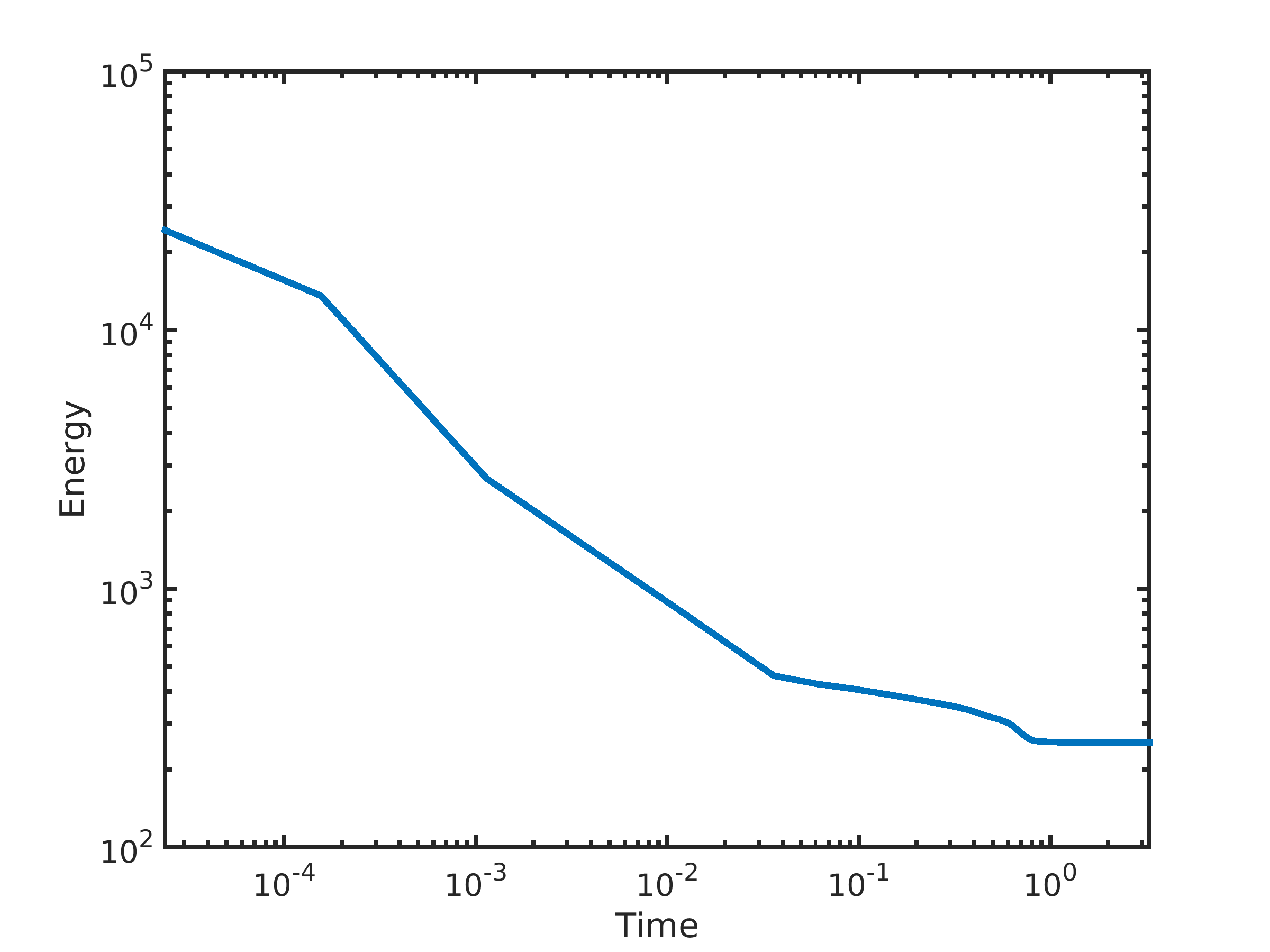}
 \includegraphics[width=.32\textwidth]{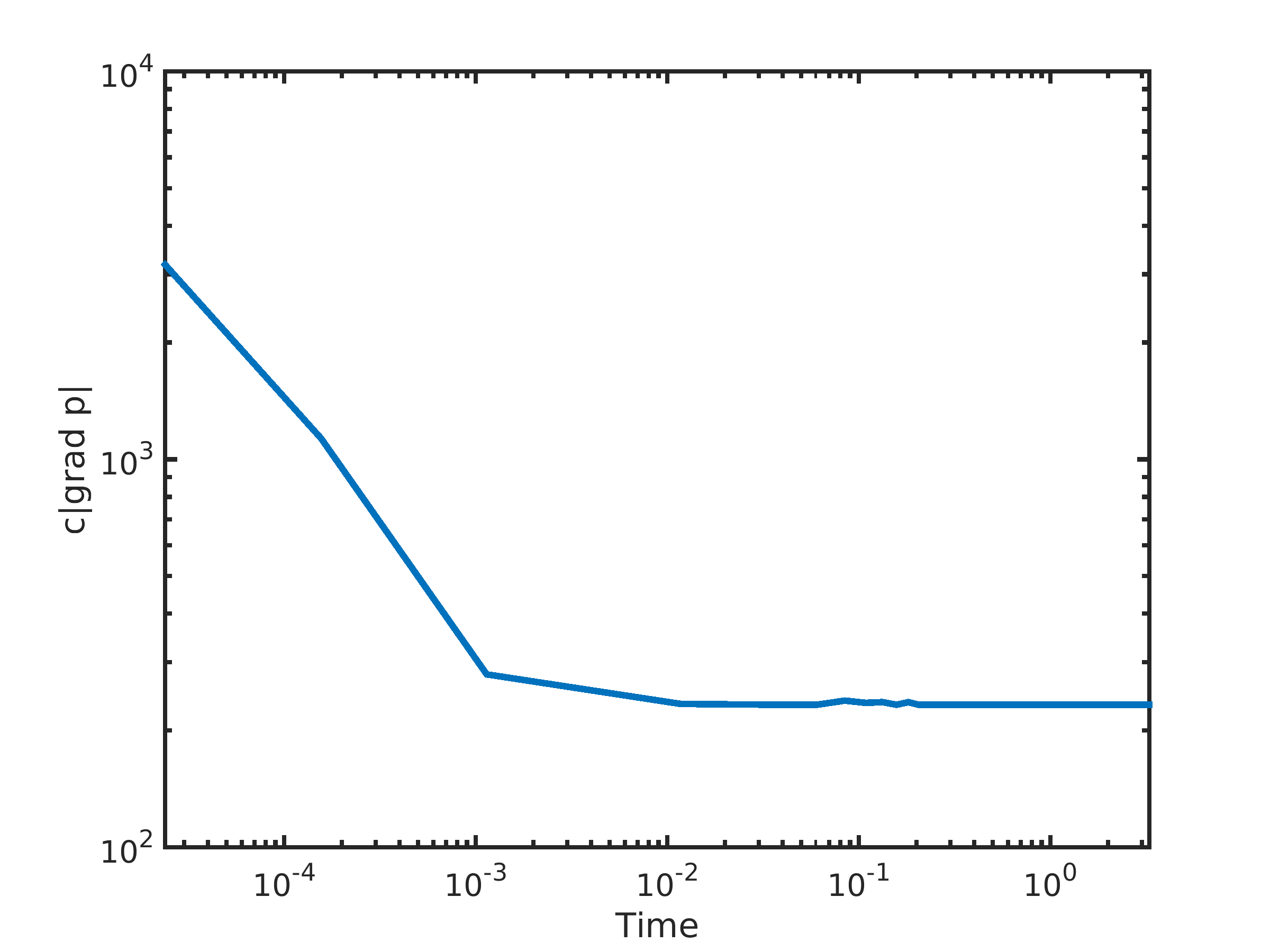}
 \caption{Stationary velocity $|u_h^k|$ for $\gamma=\frac{1}{2}$ and $D=\frac{1}{2}$ in a Log$_{10}$-scale, and corresponding evolution of $\mathcal{E}_h(m_h^k)$.
 The stationary state is reached for $t=3.2898$. 
  \label{fig:evolution_D05}}
\end{figure}

\begin{figure} \centering
 \includegraphics[width=.32\textwidth]{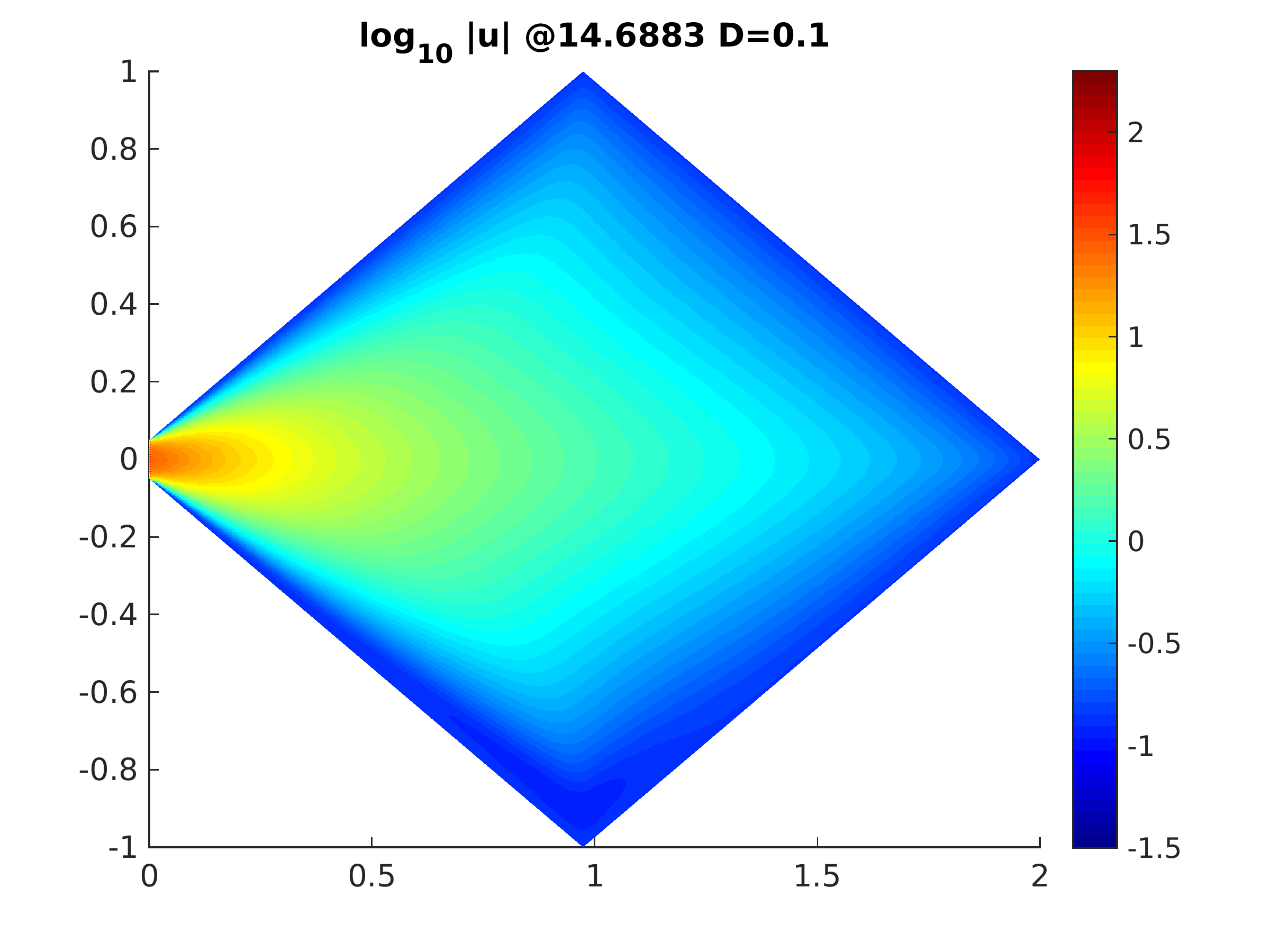}
 \includegraphics[width=.32\textwidth]{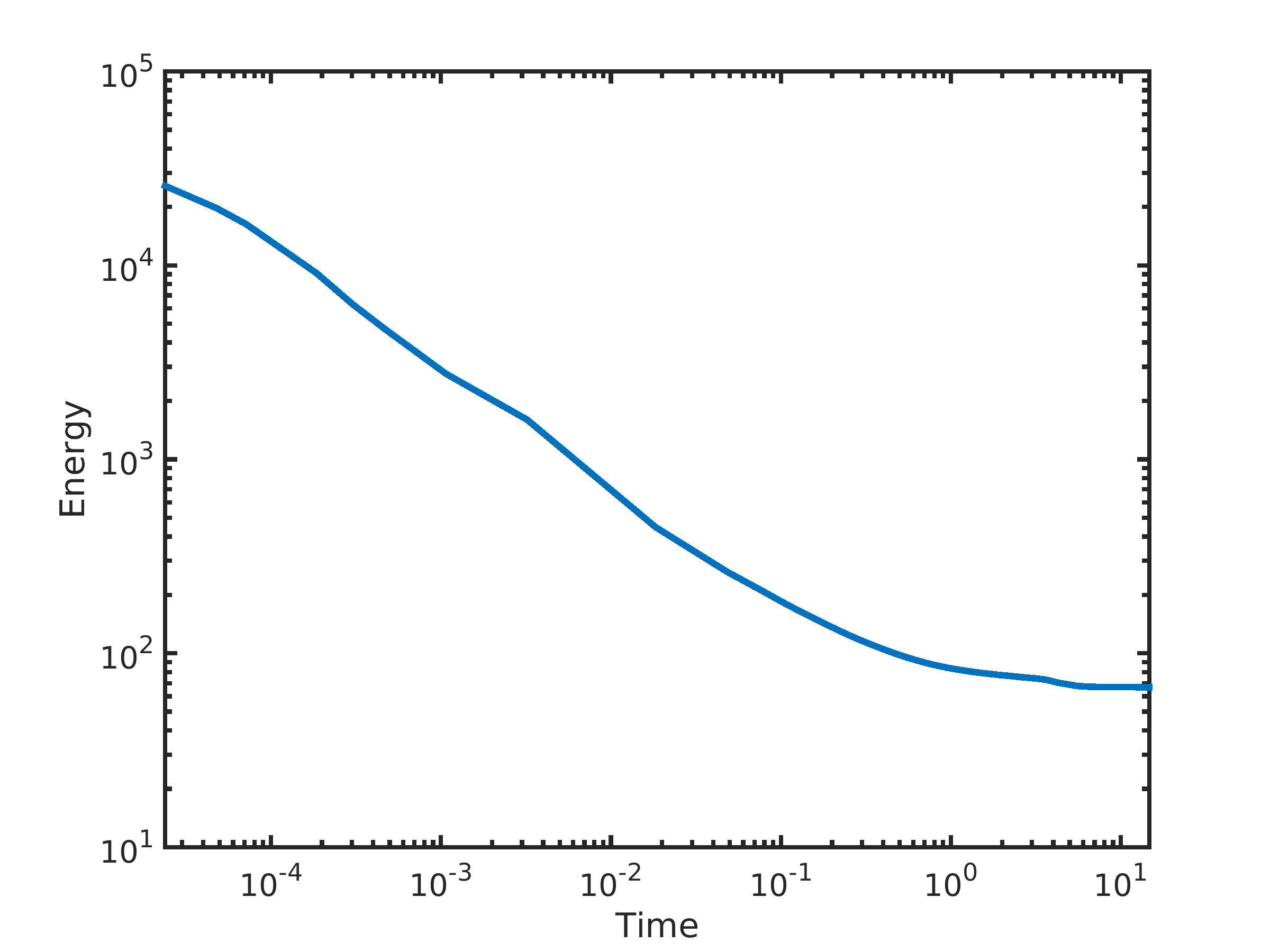}
 \includegraphics[width=.32\textwidth]{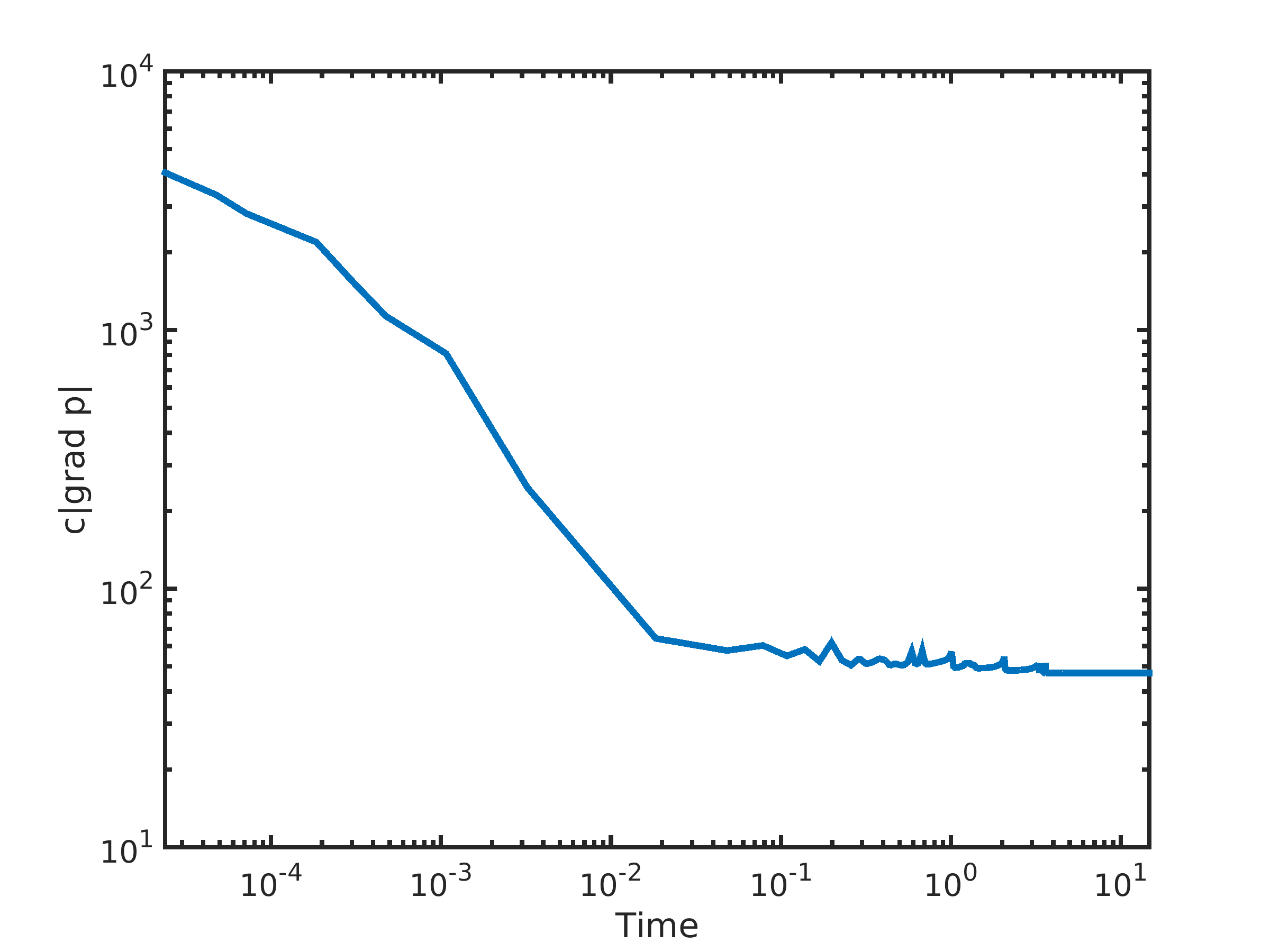}
 \caption{Near stationary velocity $|u_h^k|$ for $\gamma=\frac{1}{2}$ and $D=\frac{1}{10}$ in a Log$_{10}$-scale, and corresponding evolution of $\mathcal{E}_h(m_h^k)$ and $\|c|\nabla p_h^k|\|_{L^\infty(\Omega)}$.
 \label{fig:evolution_D01}}
\end{figure}


\begin{figure} \centering
 \includegraphics[width=.32\textwidth]{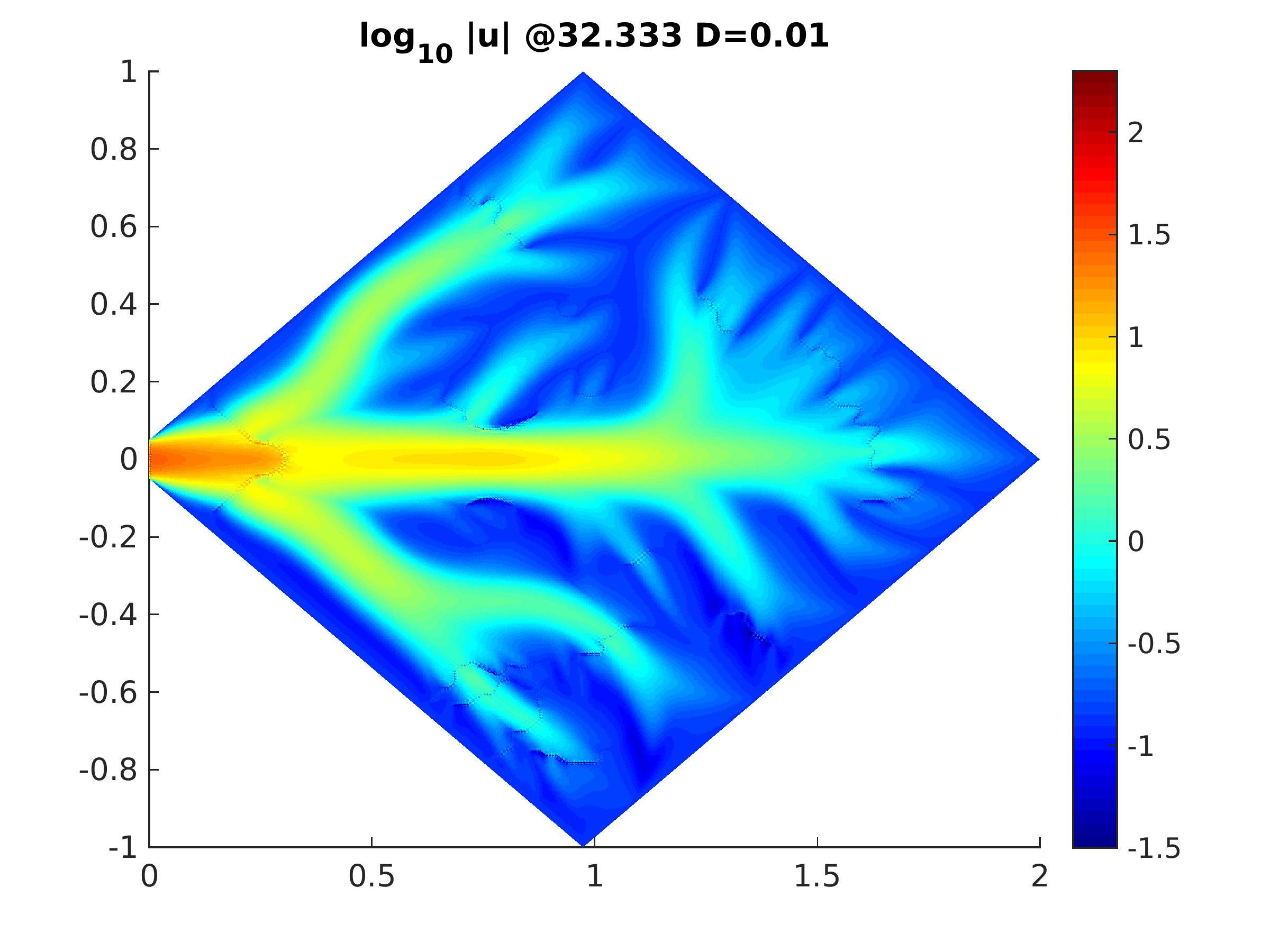}
 \includegraphics[width=.32\textwidth]{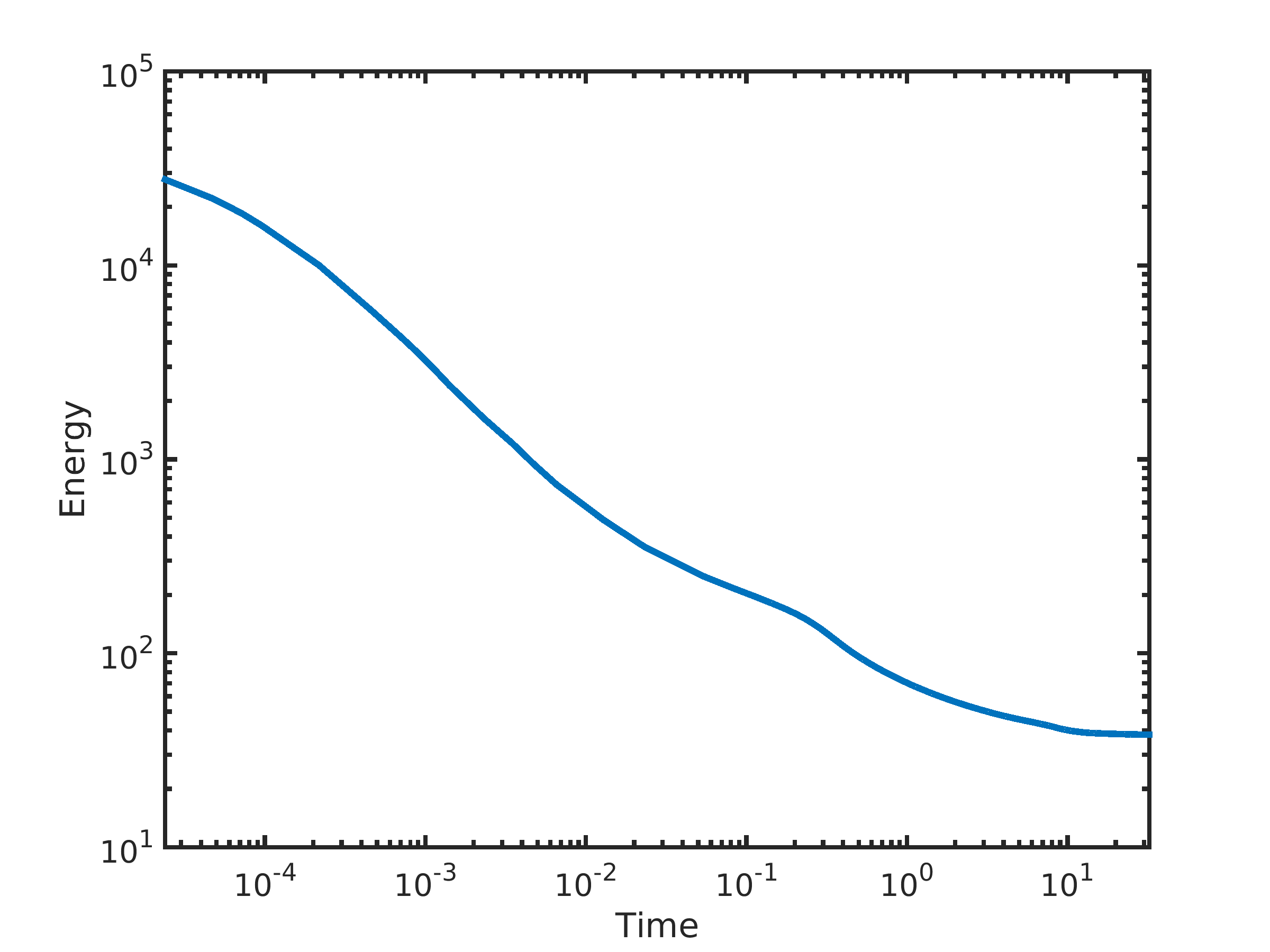}
 \includegraphics[width=.32\textwidth]{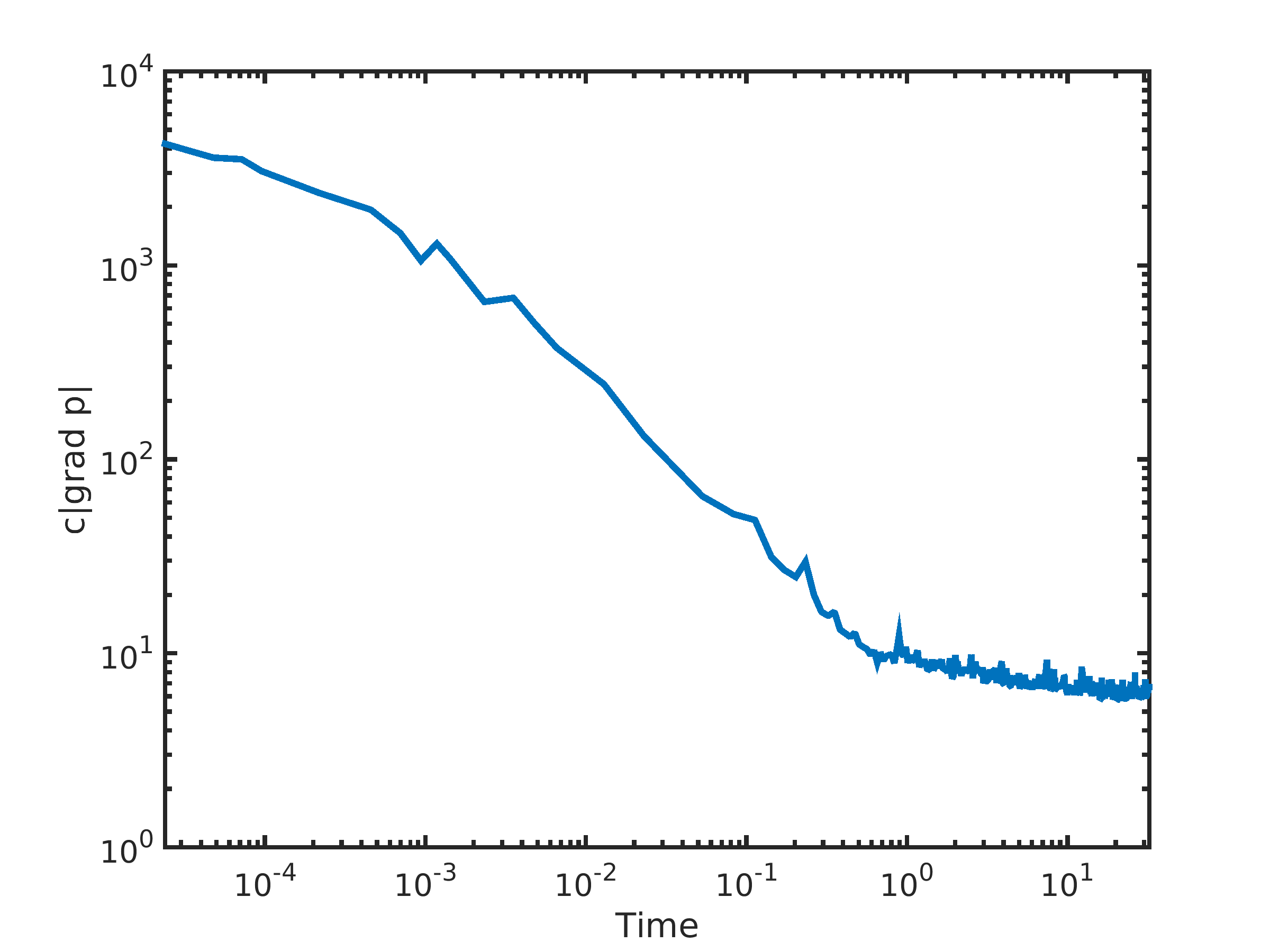}
 \caption{Near stationary velocity $|u_h^k|$ for $\gamma=\frac{1}{2}$ and $D=\frac{1}{100}$ in a Log$_{10}$-scale, and corresponding evolution of $\mathcal{E}_h(m_h^k)$ and $\|c|\nabla p_h^k|\|_{L^\infty(\Omega)}$.\label{fig:evolution_D001}}
\end{figure}



\begin{figure} \centering
 \includegraphics[width=.32\textwidth]{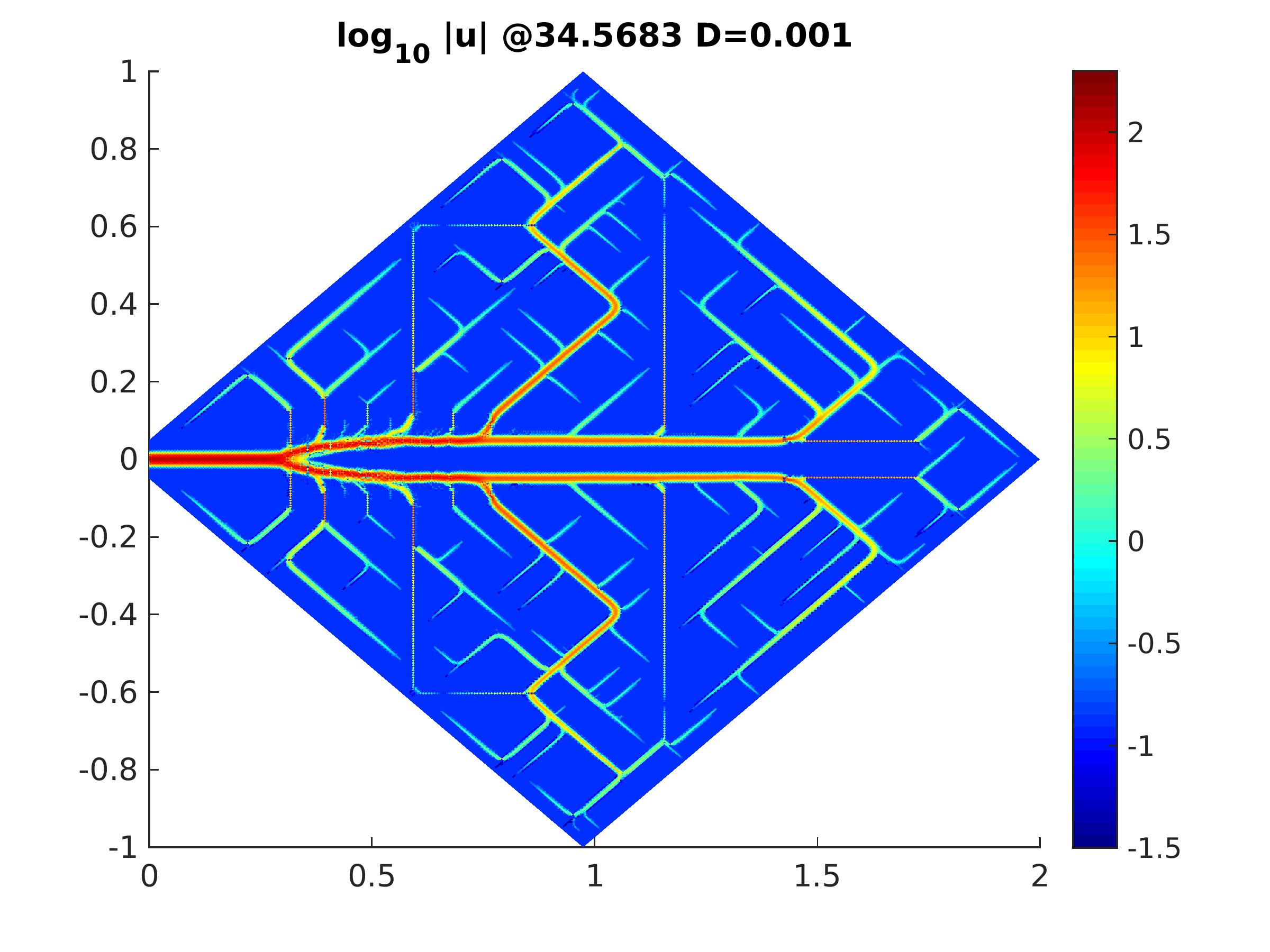}
 \includegraphics[width=.32\textwidth]{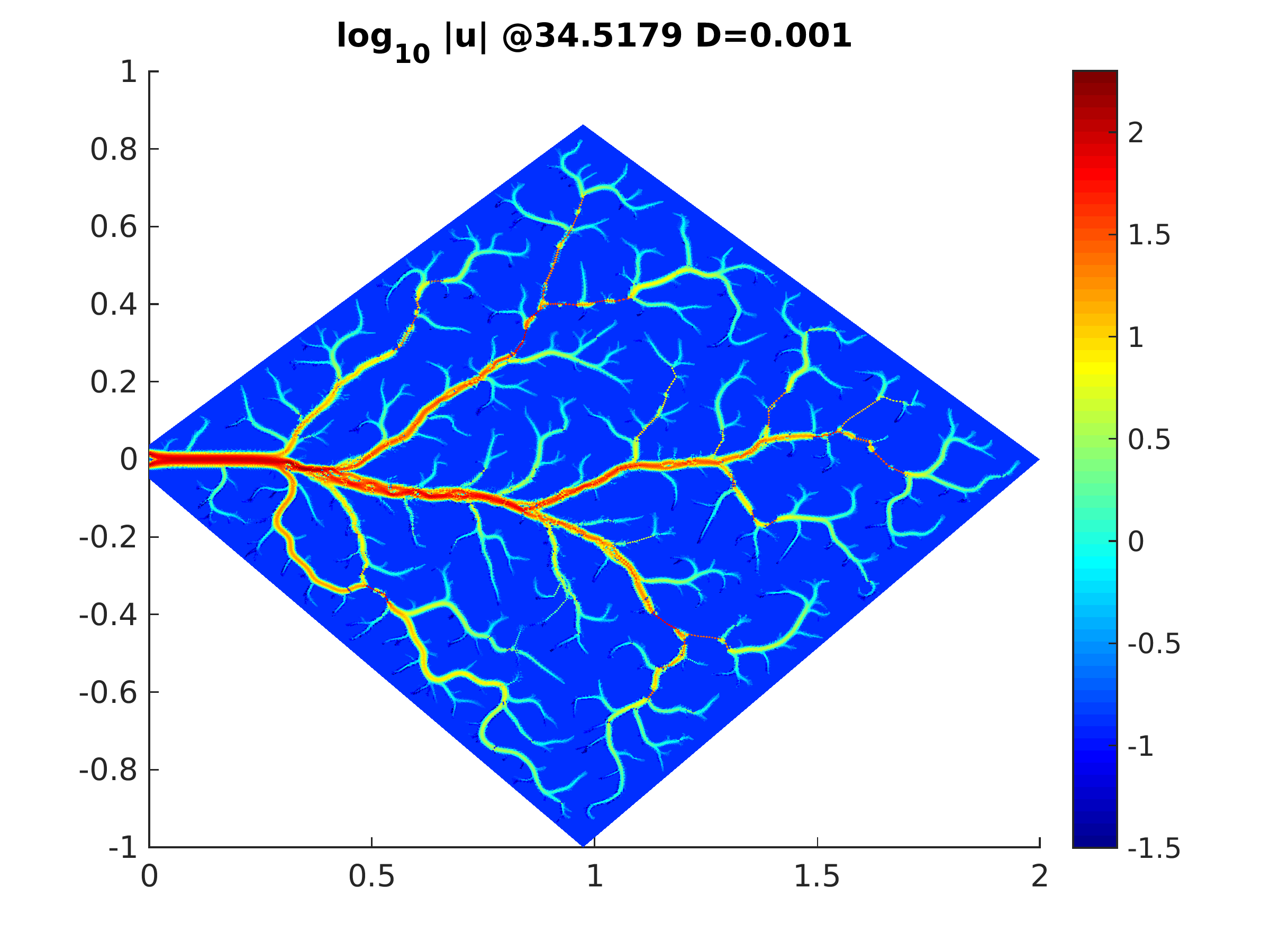}
 \includegraphics[width=.32\textwidth]{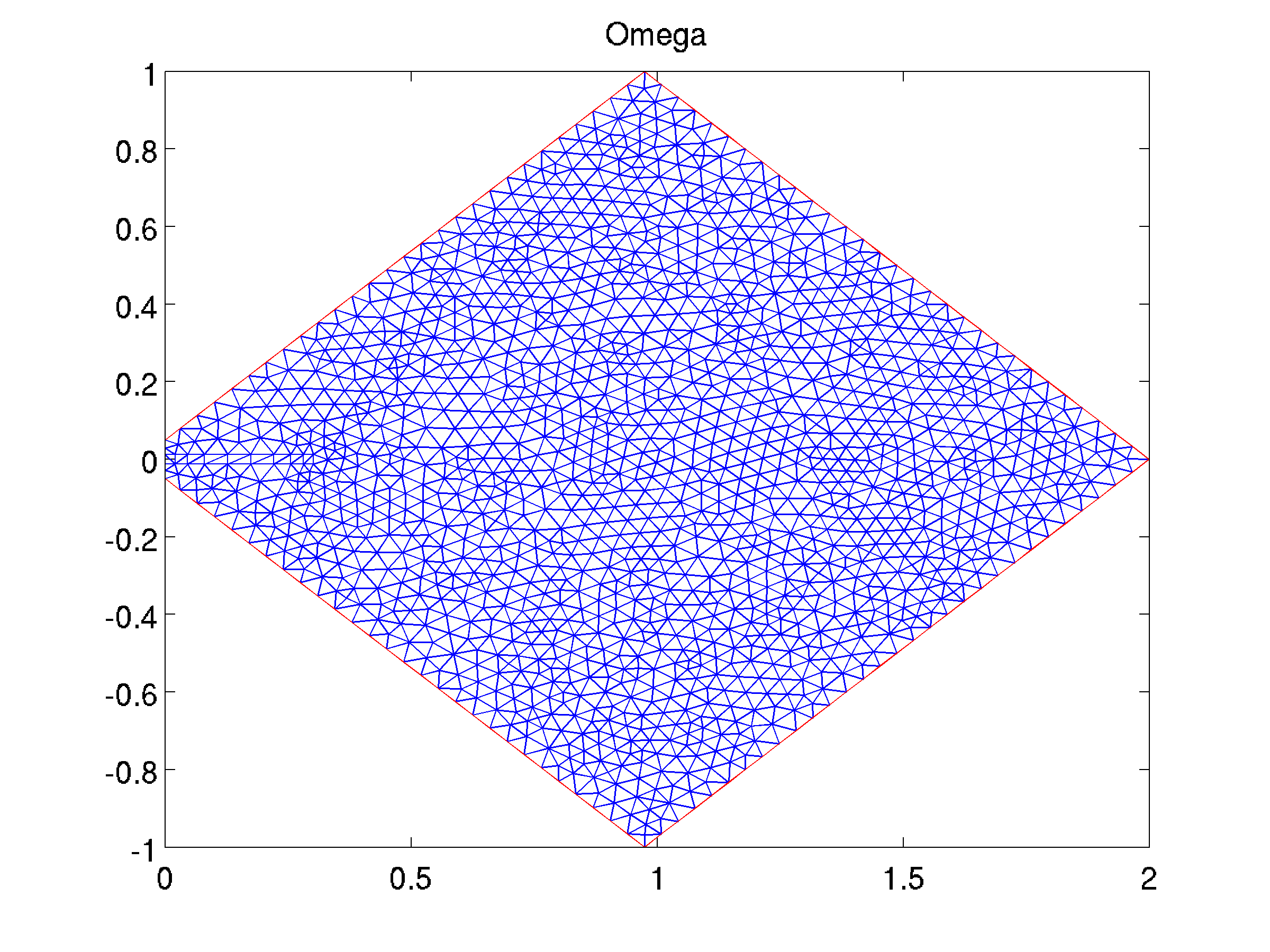}
 \caption{Left panel: Velocity $|u_h^k|$ for $\gamma=\frac{1}{2}$ and $D=\frac{1}{1000}$ in a Log$_{10}$-scale computed on a refined triangulation from the uniform grid shown in Figure~\ref{fig:domain} with $102,905$ vertices, cf. Section~\ref{sec:setup}. Corresponding velocity (middle panel) computed on a refinement with $107,009$ vertices of the non-uniform grid shown in the right panel. For the refined non-uniform grid we have $\min h_T=0.0017$ and $\max h_T=0.0043$.\label{fig:evolution_D0001}}
\end{figure}

\begin{table}
\centering
 \begin{tabular}{c| r r r }
  $D$ & $\frac{1}{2}$ & $\frac{1}{10}$ & $\frac{1}{100}$ \\
  \hline
  $t^k$           &           $3.2898$   &            $14.6883$  &             $32.333$ \\
  $\E_{h,t}^{k}$  & $-1.1\times 10^{-5}$ &  $-2.9\times 10^{-6}$ &  $-1.5\times 10^{-2}$\\
  $m_{h,t}^k$     &  $6.2\times 10^{-7}$ &  $4.5\times 10^{-2}$  &   $2.8\times 10^{-1}$\\
  $s_k$           &           $1.344153$ &           $1.348375$  &            $1.576061$\\
 \end{tabular}
 \caption{Stationarity and sparsity measures for different values of $D$ and $\gamma=\frac{1}{2}$, see Section~\ref{sec:varying_D}. 
\label{tab:sparsity_D}}
\end{table}

\subsection{Varying $\gamma$}\label{sec:varying_gamma}

In order to demonstrate the dependence of the network formation process on the relaxation term, we let $\gamma\in\{\frac{3}{5}, \frac{3}{4}, 1, \frac{3}{2}, 2\}$; for $\gamma=\frac{1}{2}$ see Section~\ref{sec:varying_D}.
For $\gamma\in \{\frac{3}{2},2\}$ the results are depicted in Figure~\ref{fig:evolution_gamma15} and Figure~\ref{fig:evolution_gamma2}. From the evolution of the energies and from Table~\ref{tab:sparsity_gamma} we may conclude that for these two values of $\gamma$ we are near a stationary state, and that for $\gamma>1$ no fine scale structures built up.
For $\gamma=1$, depicted in Figure~\ref{fig:evolution_gamma1}, network structures appear for large times. This may also be indicated by the oscillating behavior of $\grad p$ for larger times. The changes in energy are however already small.
In view of Section~\ref{sec:stable_stationary}, stable solutions should satisfy $c\||\nabla p_h^k|\|_{L^\infty(\Omega)}\leq 1$ in the limiting case $D=0$.
Here, $D=1/1000$, and $c|\nabla p_h^k|\leq 2$ is in accordance with this analysis.

For $\gamma\in \{\frac{3}{5},\frac{3}{4}\}$ we observe fine scale structures, which are depicted in Figure~\ref{fig:evolution_gamma06} and Figure~\ref{fig:evolution_gamma075}.
As remarked in the previous section, the network evolution is influenced by the underlying grid due to coarse discretization, very small diffusion, and very large activiation terms. Note that for small times we have $c\||\nabla p\|_{L^\infty(\Omega)}\approx 4000$, which enters quadratically in the activation term. 
The closer $\gamma$ is to $1$ the less the relaxation term
promotes sparsity.
This might explain that for $\gamma\geq \frac{3}{4}$ we see two branches originating from the Dirichlet boundary $\Gamma=\partial \Omega\cap \{x_1=0\}$. Since the pressure gradient is very large at the transition of Dirichlet to Neumann boundary, artificial conductance is created. Notice that these two branches do not appear for $\gamma\in \{\frac{1}{2},\frac{3}{5}\}$. In this context let us mention that in several situations $L^1$-type minimization can be performed exactly by soft-shrinkage, where values below a certain threshold, which corresponds to $\delta^k$ here, are set to $0$. Hence, small values of $m$, due to round-off or small diffusion, will less affect the evolution.

%
\begin{figure} \centering
 \includegraphics[width=.32\textwidth]{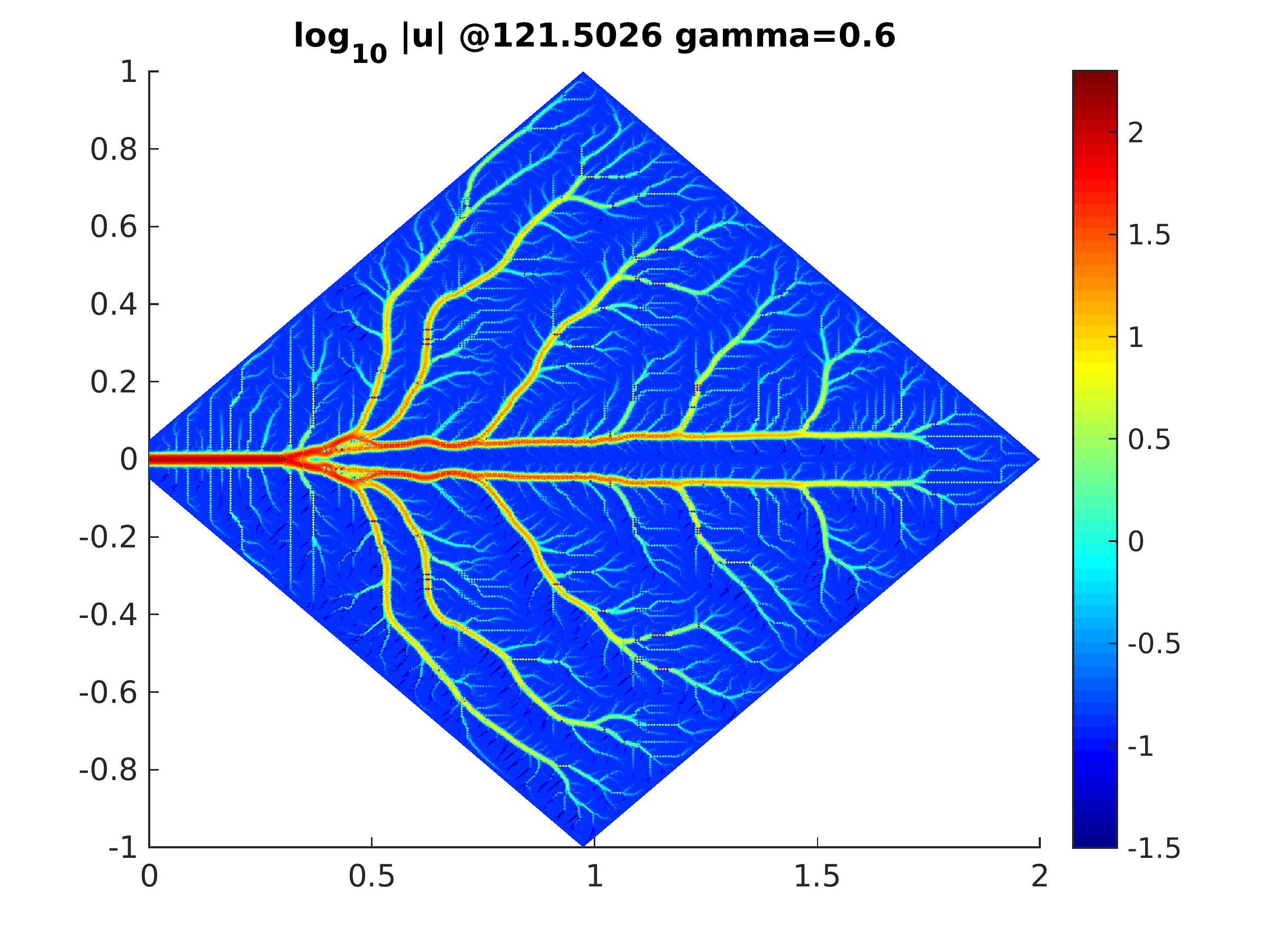}
 \includegraphics[width=.32\textwidth]{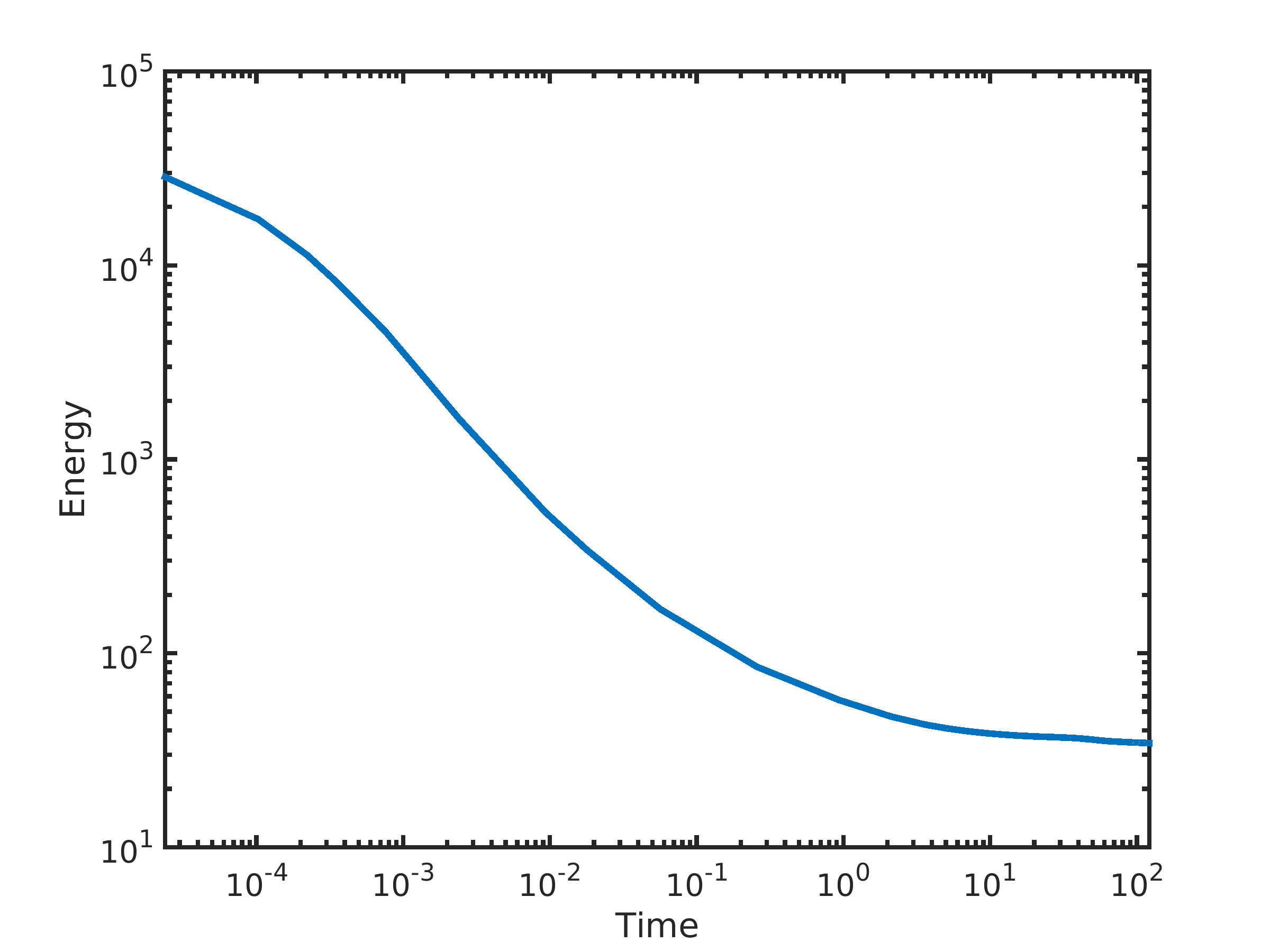}
  \includegraphics[width=.32\textwidth]{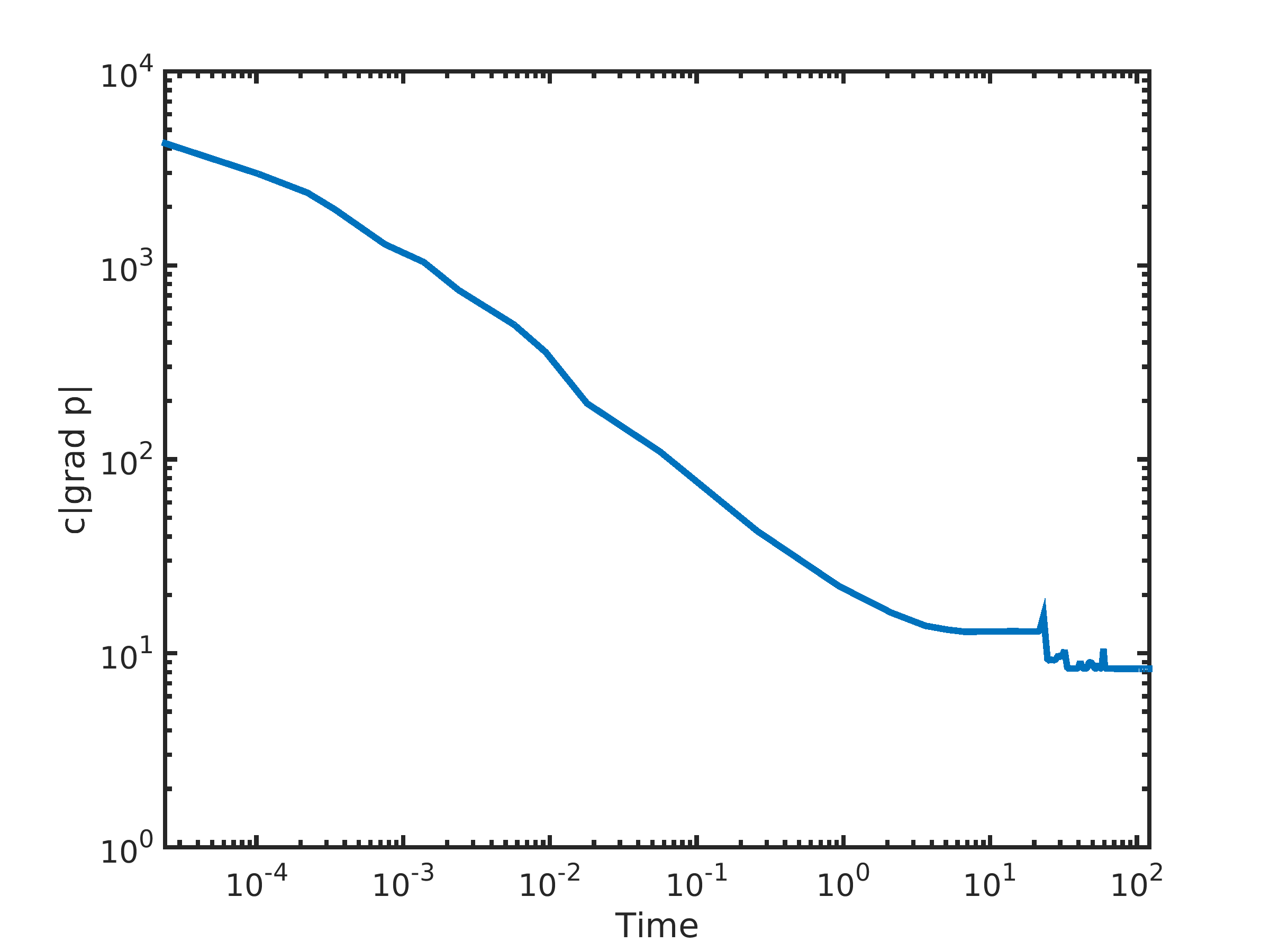}
 \caption{Near stationary velocity $|u_h^k|$ for $\gamma=\frac{3}{5}$, $D=\frac{1}{1000}$ in a Log$_{10}$-scale for different times, and corresponding evolution of the $\mathcal{E}_h(m_h^k)$ and $\|c|\nabla p|\|_{L^\infty(\Omega)}$ vs. time in a logarithmic scaling. 
 \label{fig:evolution_gamma06}}
\end{figure}

\begin{figure} \centering
  \includegraphics[width=.32\textwidth]{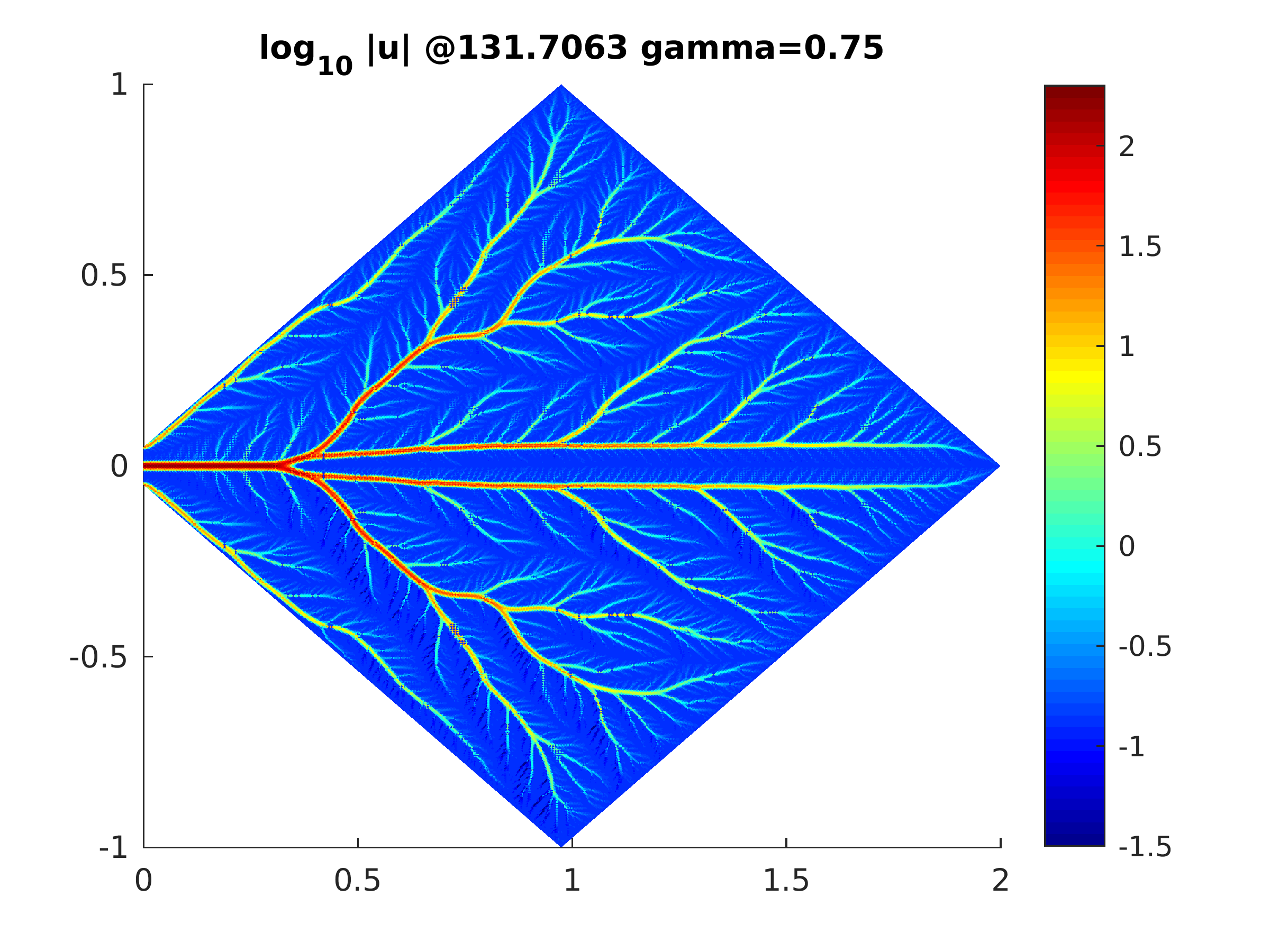}
  \includegraphics[width=.32\textwidth]{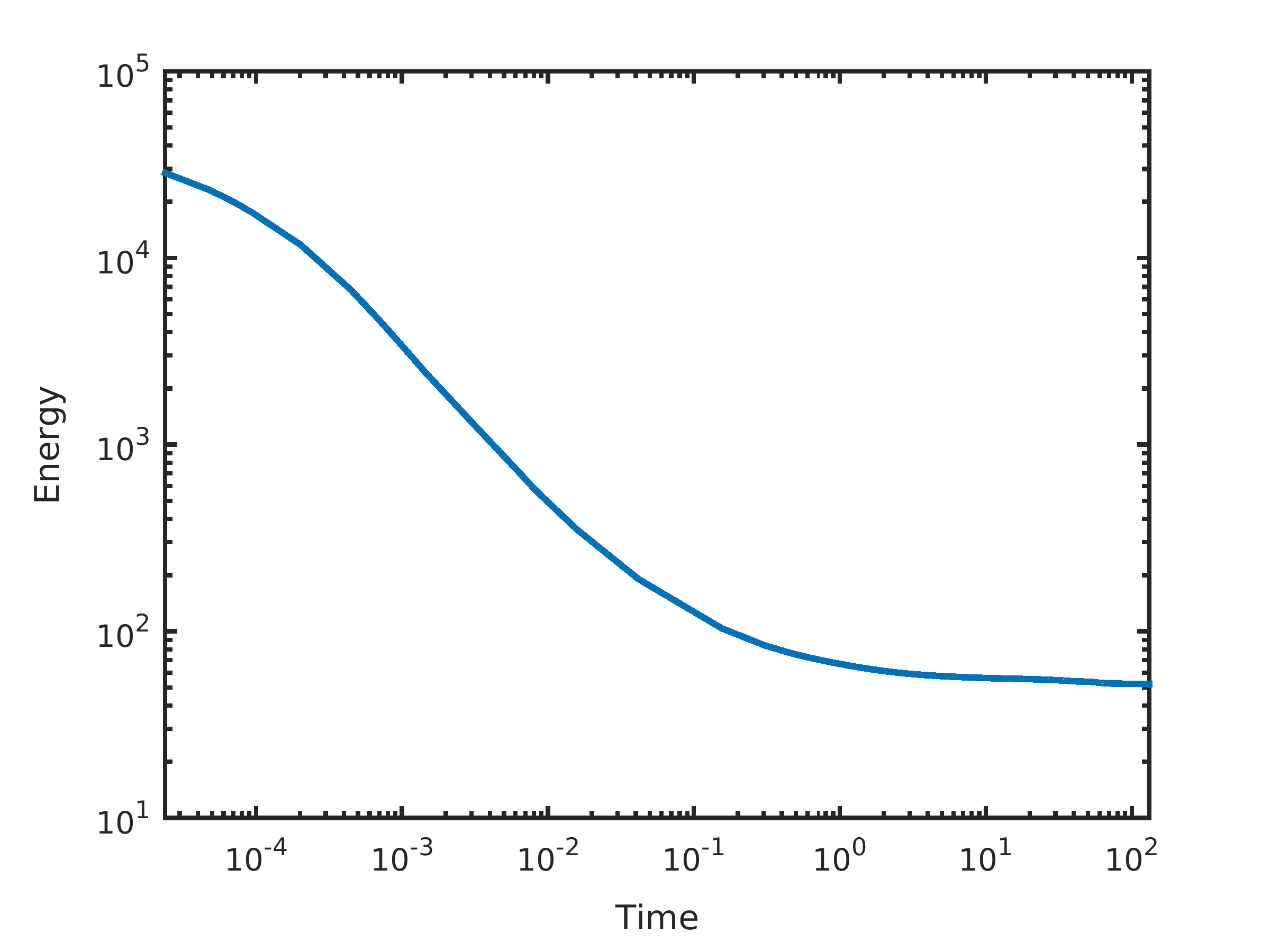}
  \includegraphics[width=.32\textwidth]{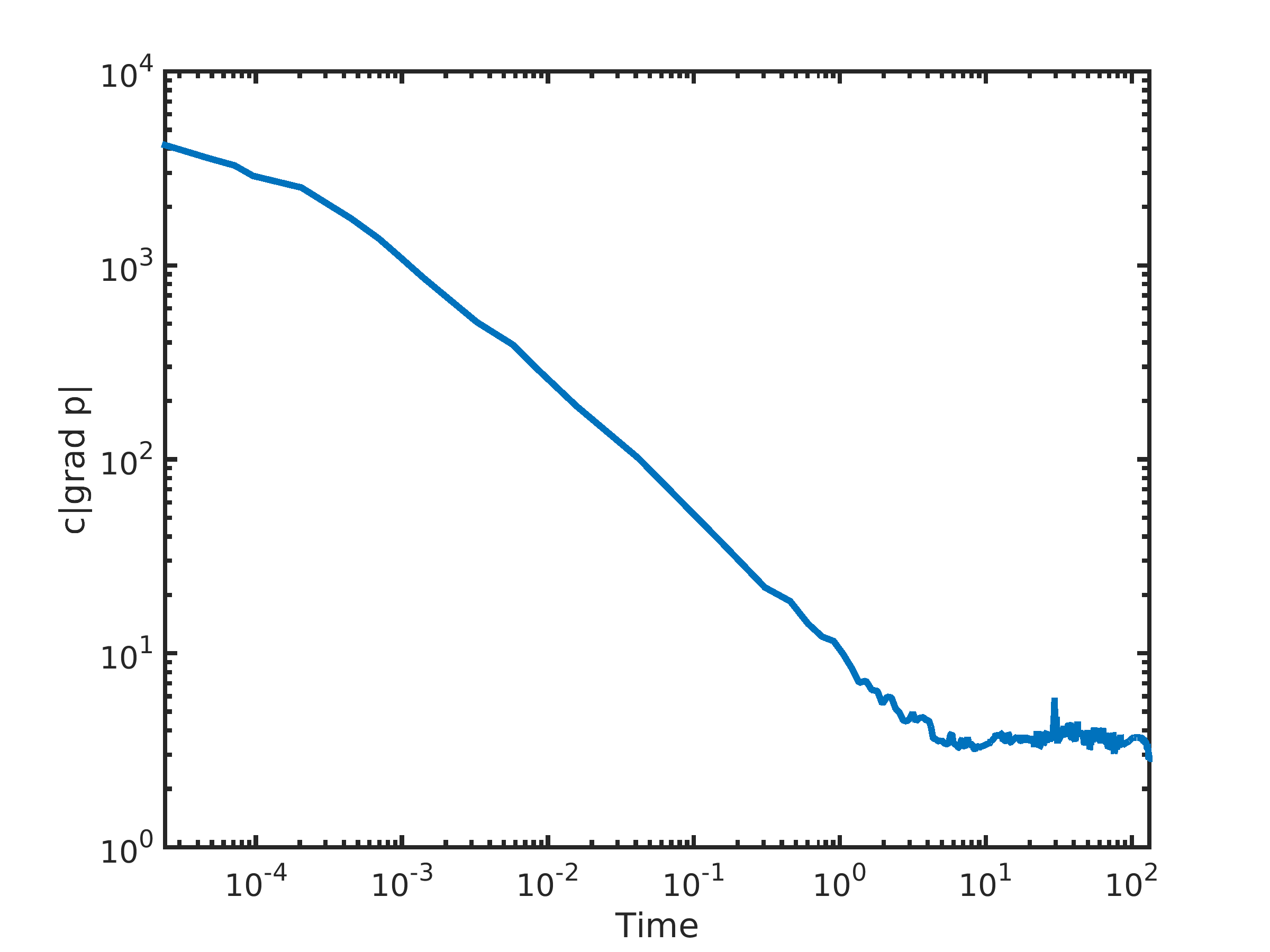}
 \caption{Velocity $|u_h^k|$ for $\gamma=\frac{3}{4}$, $D=\frac{1}{1000}$ in a Log$_{10}$-scale, and corresponding evolution of the $\mathcal{E}_h(m_h^k)$ and $\|c|\nabla p|\|_{L^\infty(\Omega)}$ vs. time in a logarithmic scaling. 
 \label{fig:evolution_gamma075}}
\end{figure}


\begin{figure} \centering
 \includegraphics[width=.32\textwidth]{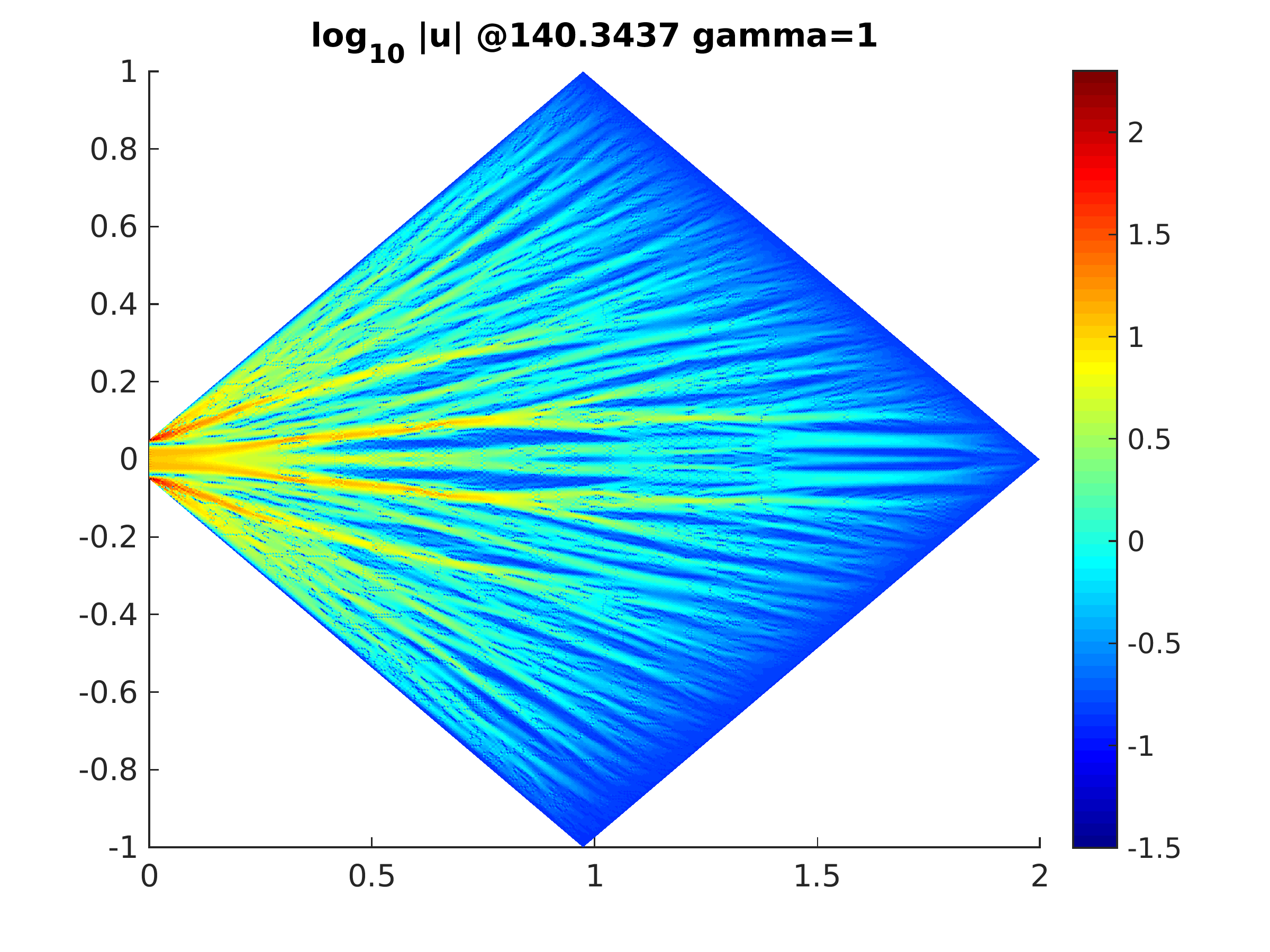}
 \includegraphics[width=.32\textwidth]{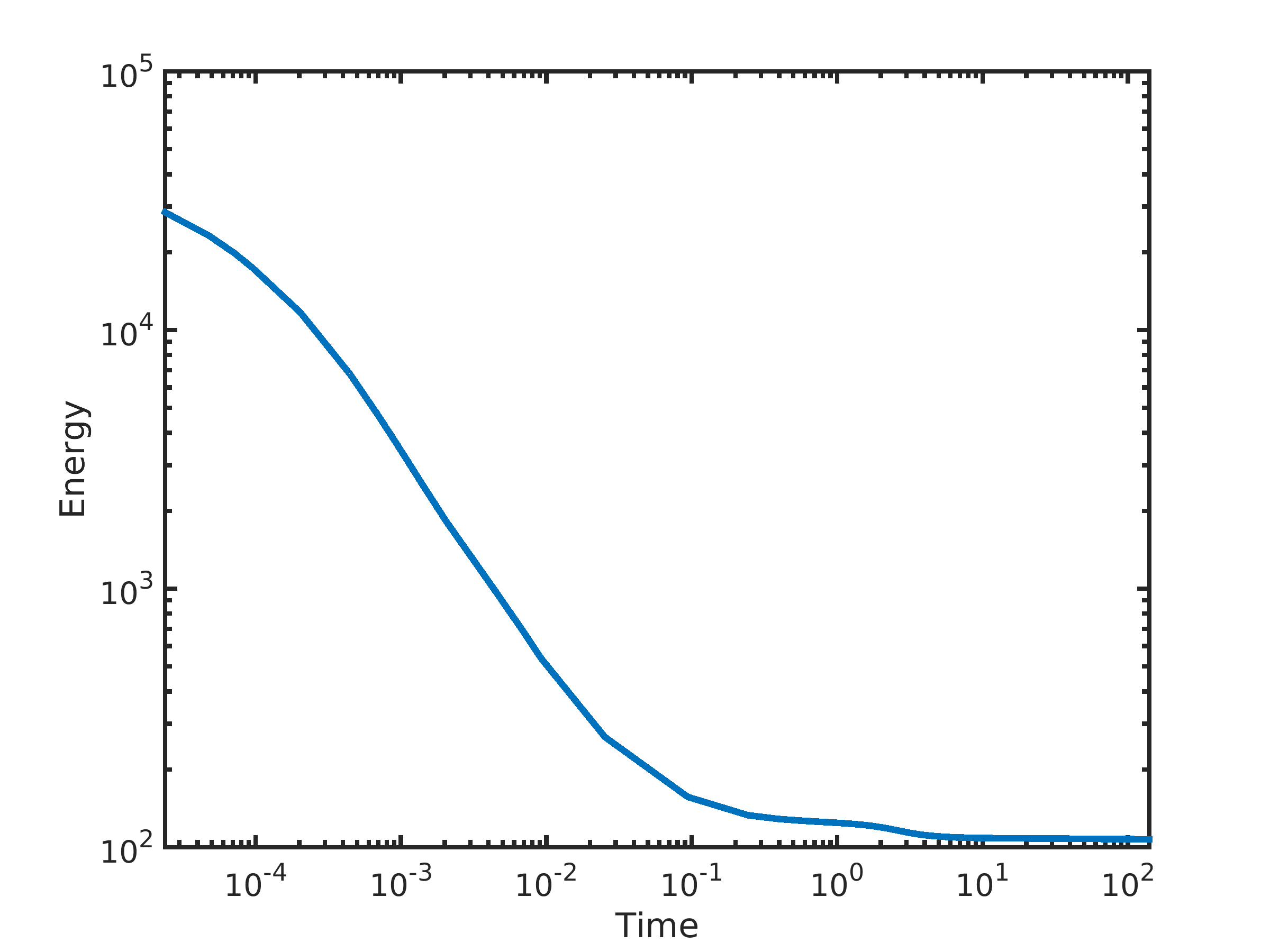}
 \includegraphics[width=.32\textwidth]{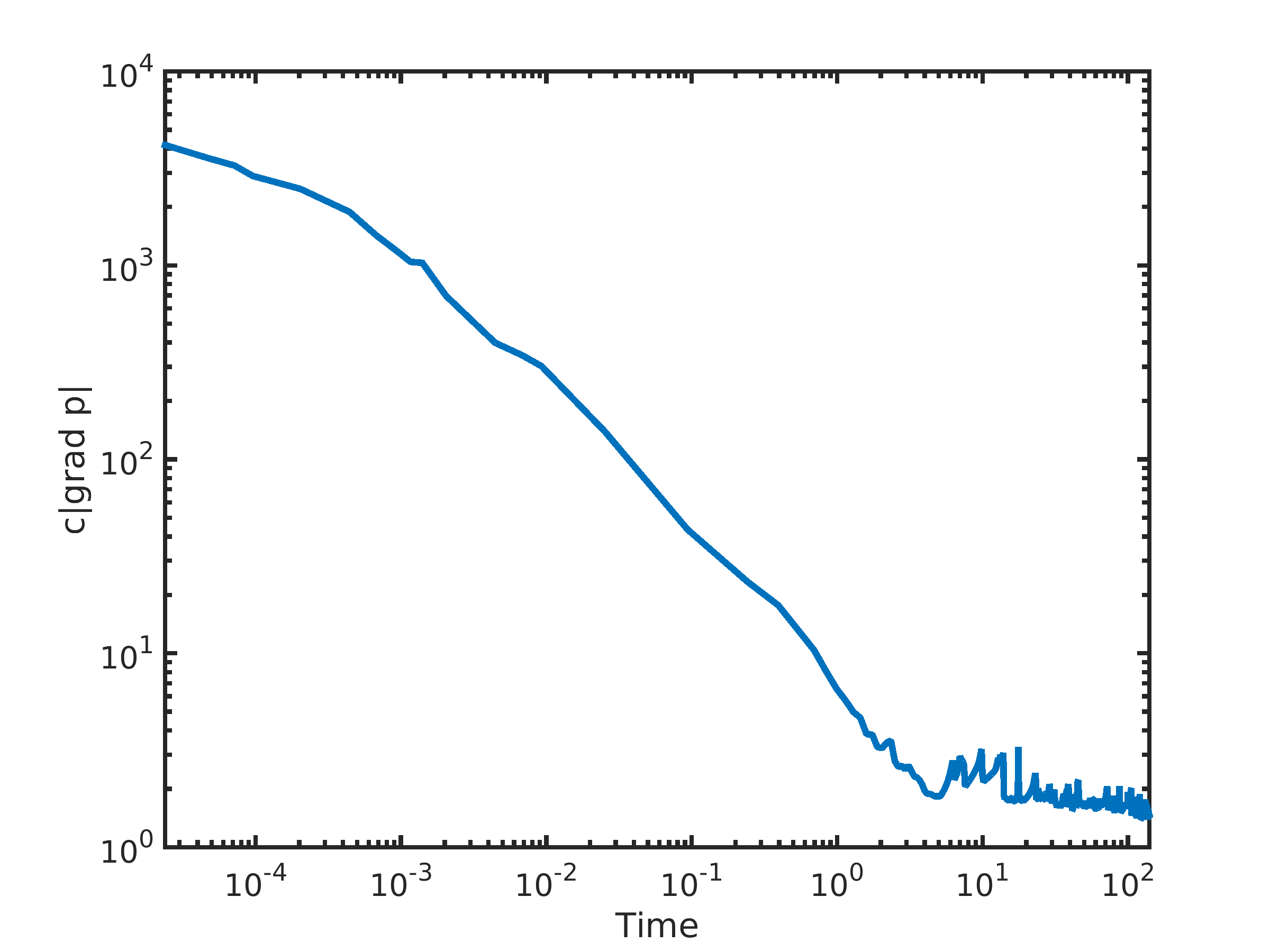}
 \caption{Velocity $|u_h^k|$ for $\gamma=1$, $D=\frac{1}{1000}$ in a Log$_{10}$-scale, and corresponding evolution of the $\mathcal{E}_h(m_h^k)$ and $\|c|\nabla p|\|_{L^\infty(\Omega)}$ vs. time in a logarithmic scaling.  \label{fig:evolution_gamma1}}
\end{figure}


\begin{figure} \centering
 \includegraphics[width=.32\textwidth]{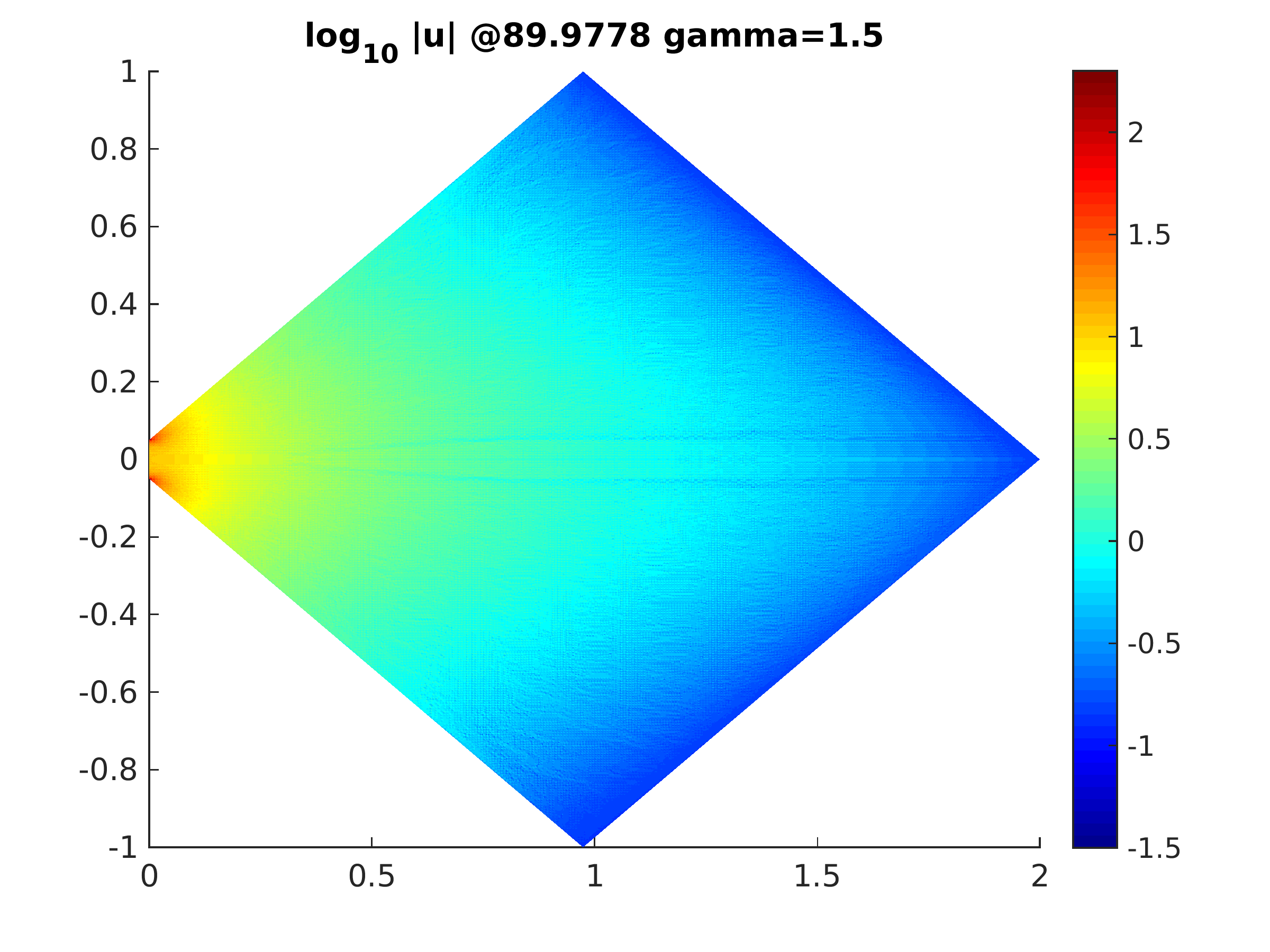}
 \includegraphics[width=.32\textwidth]{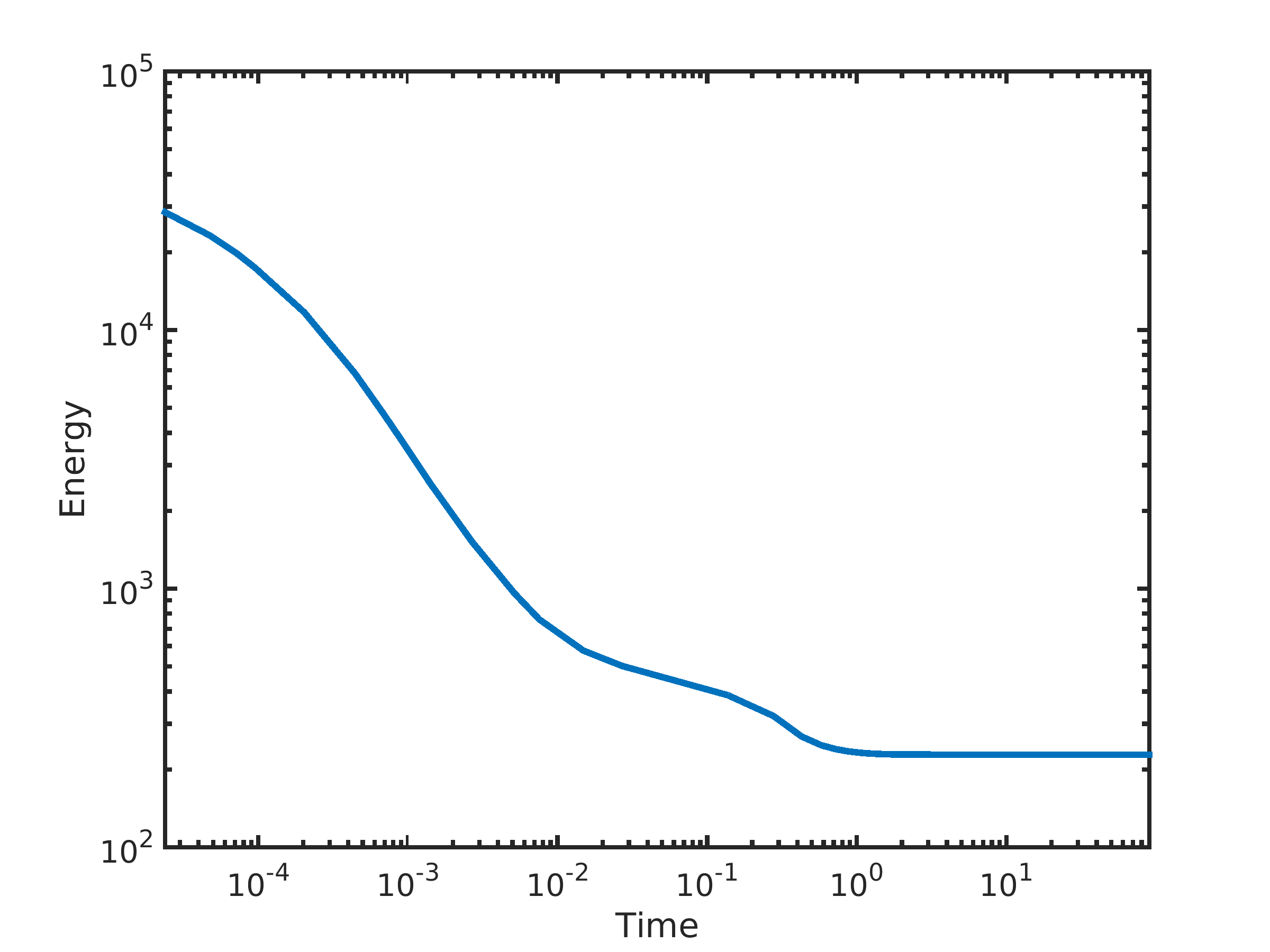}
 \includegraphics[width=.32\textwidth]{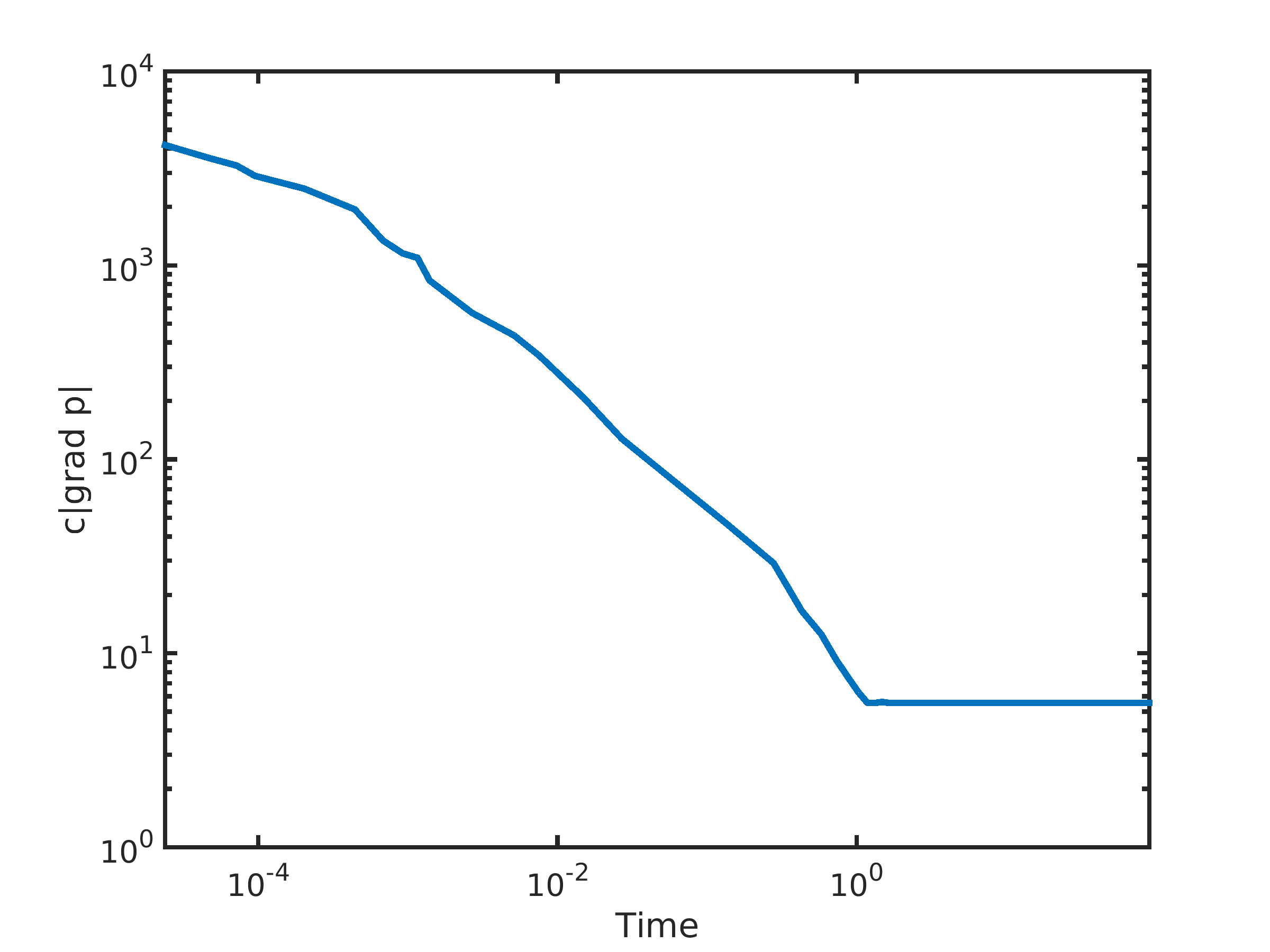}
 \caption{Velocity $|u_h^k|$ for $\gamma=\frac{3}{2}$, $D=\frac{1}{1000}$ in a Log$_{10}$-scale, and corresponding evolution of the $\mathcal{E}_h(m_h^k)$ and $\|c|\nabla p|\|_{L^\infty(\Omega)}$ vs. time in a logarithmic scaling.\label{fig:evolution_gamma15}}
\end{figure}


\begin{figure} \centering
 \includegraphics[width=.32\textwidth]{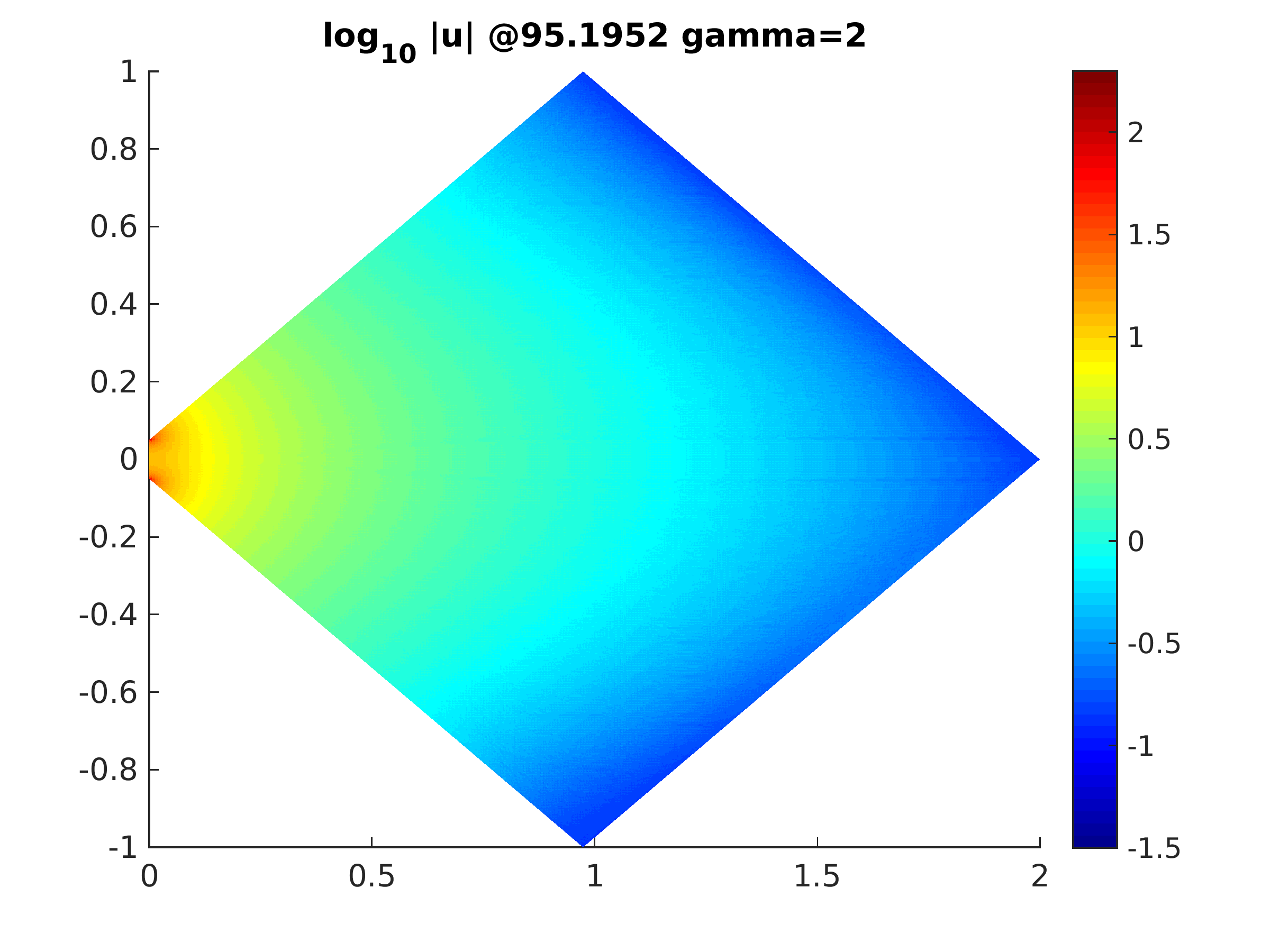}
 \includegraphics[width=.32\textwidth]{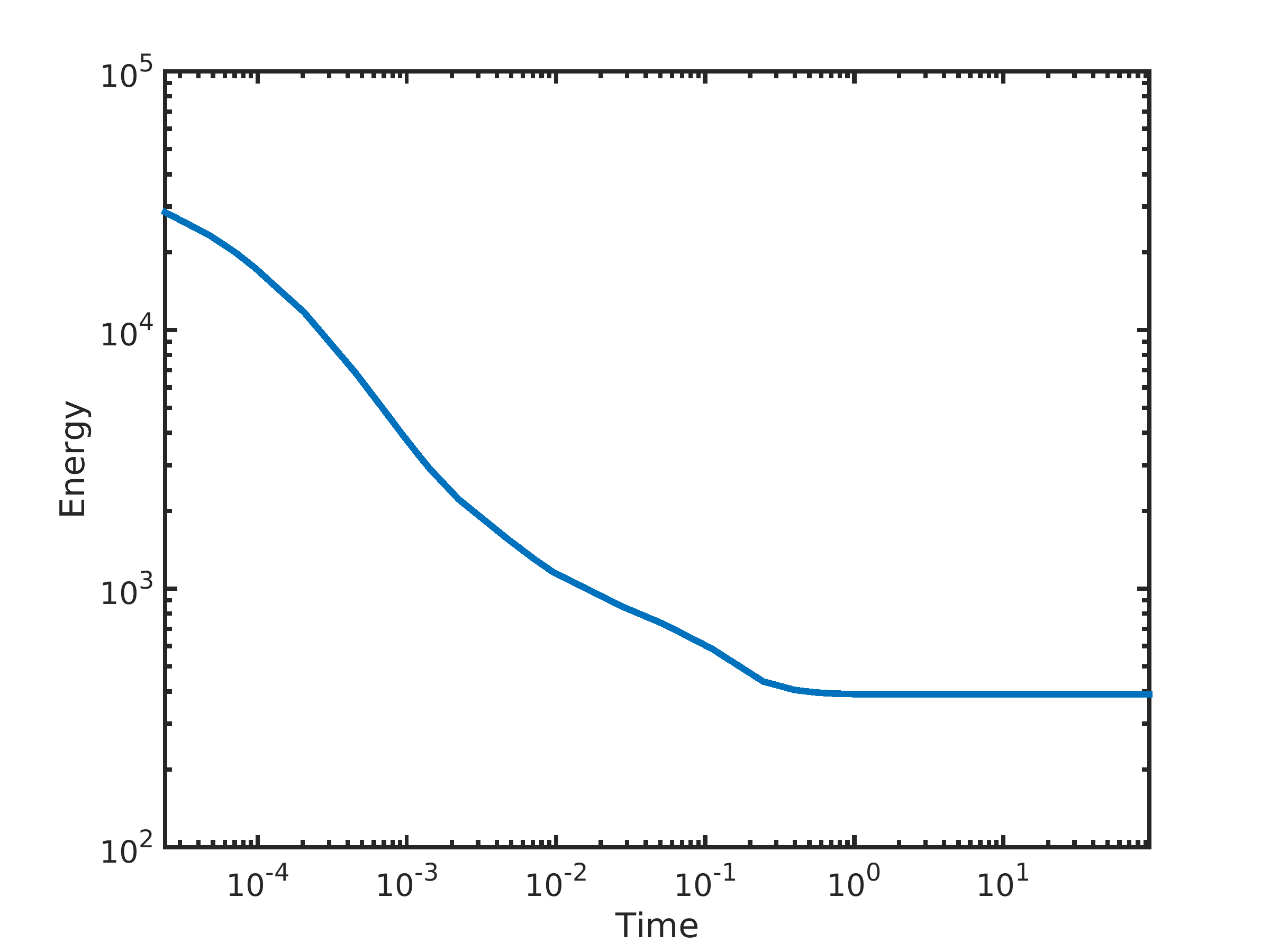}
 \includegraphics[width=.32\textwidth]{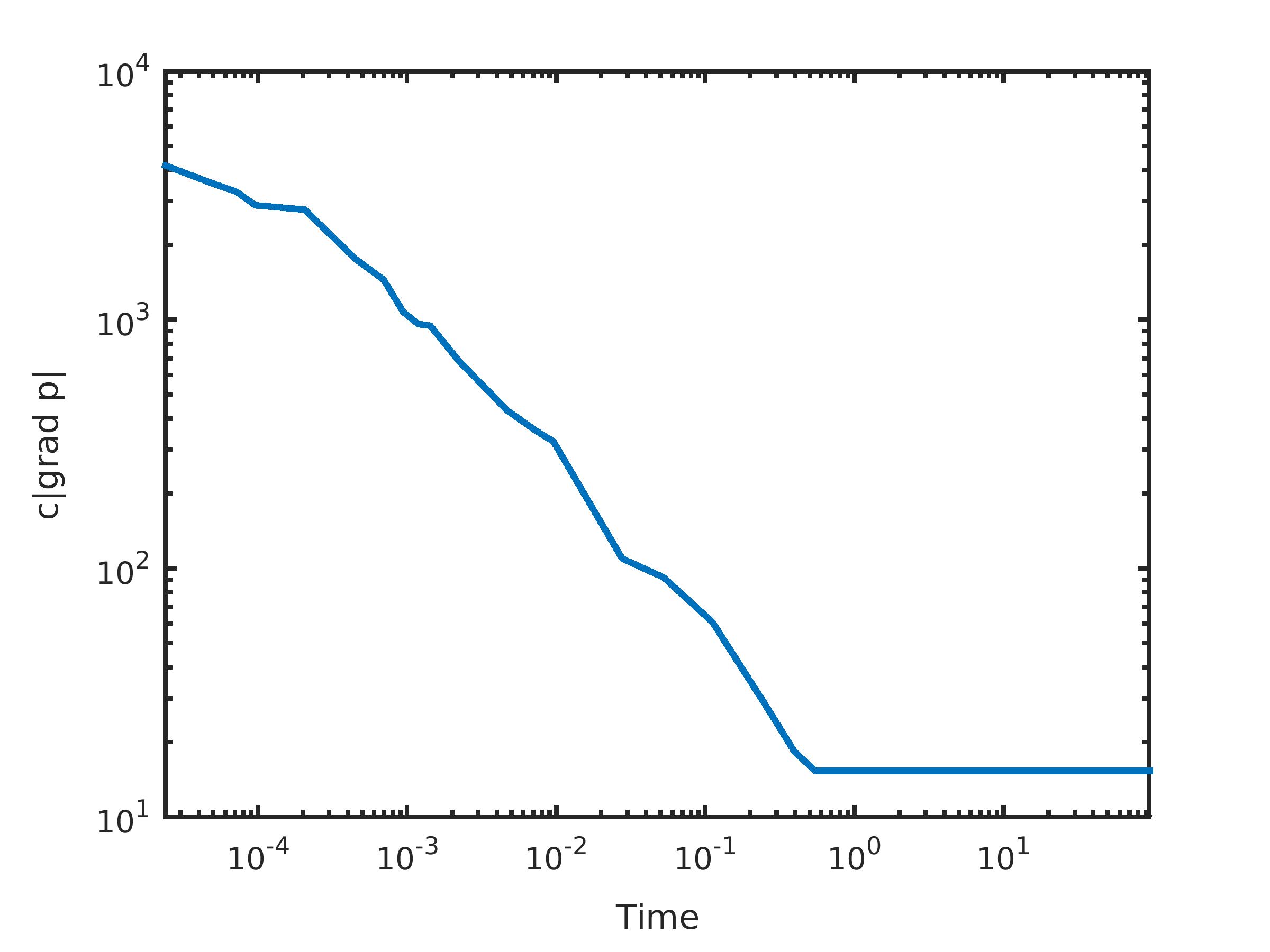}
 \caption{Velocity $|u_h^k|$ for $\gamma=2$, $D=\frac{1}{1000}$ in a Log$_{10}$-scale, and corresponding evolution of the $\mathcal{E}_h(m_h^k)$ and $\|c|\nabla p|\|_{L^\infty(\Omega)}$ vs. time in a logarithmic scaling. \label{fig:evolution_gamma2}}
\end{figure}

\begin{table}
\centering
 \begin{tabular}{c| r r r r r}
  $\gamma$ & $ \frac{3}{5}$ & $\frac{3}{4}$ & $1$ & $\frac{3}{2}$ & $2$\\
  \hline
  $t^k$           &    $121.5026$        &     $131.7063$  &      $140.3437$         &  $89.9778$ & $95.1952$ \\
  $\E_{h,t}^{k}$ & $-1.3\times 10^{-2}$ & $-3.4\times 10^{-3}$ & $-2.1\times 10^{-3}$  &  $-4.2\times 10^{-7}$& $-9.8 \times 10^{-7}$\\
  $m_{h,t}^k$    & $2.1\times 10^{-1}$& $9.1\times 10^{-2}$  & $1.1\times 10^{-1}$   &  $1.7\times 10^{-4}$& $1.1\times 10^{-5}$ \\
  $s_k$          & $3.547262$ &  $4.153456$         &    $1.555159$   &  $1.181271$          & $1.161117$ 
 \end{tabular}
 \caption{Stationarity and sparsity measures for different values of $\gamma$ and $D=\frac{1}{1000}$, see Section~\ref{sec:varying_gamma}.
 For $\gamma\in \{\frac{3}{2},2\}$ the change in the energy $\E_h$ is already within machine accuracy.
 \label{tab:sparsity_gamma}}
\end{table}

\subsection{Unstable stationary solutions for $D=0$ and $\frac{1}{2}\leq \gamma <1$}\label{sec:num_instability}
In Section~\ref{sec:instable_gamma_less_1} we have constructed stationary solutions for $D=0$ and $\frac{1}{2}\leq \gamma<1$.
In one dimension our stability analysis shows that these stationary states are not stable.
In the following, we indicate that these stationary states are unstable also in two dimensions.
To do so, we compute the minimizer of the functional $\mathcal{F}_\alpha$ defined in \eqref{fF} with 
$$\alpha=\alpha_\gamma=c^{-\frac{1}{4}} \Big(\frac{1-\gamma}{1+\gamma}\Big)^{\frac{\gamma-1}{2}},
$$
which is \eqref{alpha_gamma} for general values of $c$.
For the minimization we use a gradient descent method with step-sizes chosen by the Armijo rule \cite{NocedalWright}. The iteration is stopped as soon as two subsequent iterates of the pressure, say $p^k$ and $p^{k+1}$, satisfy $\|p^{k}-p^{k+1}\|_{H^1(\Omega)}/\|p^k\|_{H^1(\Omega)} < 10^{-15}$, i.e. they coincide up to round-off errors.
Since the derivative of $\mathcal{F}_\alpha$ is discontinuous, one should in general use more general methods from convex optimization to ensure convergence of the minimization scheme, for instance proximal point methods \cite{Rockafellar}. However, in our example also the gradient descent method converged.
We set $\gamma=1/2$, $r=1$ and $c=50$. Moreover, we let $\mathcal{A}=\{x\in\RR^2: (x_1-1)^2-x_2^2<1/4\}\subset\Omega$, see Figure~\ref{fig:stationary_instable}.
The stationary conductance $m_0$ is computed via \eqref{m0_cutoff} and satisfies
\begin{align*}
 \| c\nabla p_0\otimes c\nabla p_0 m_0 - |m_0|^{2(\gamma-1)} m_0\|_{L^2(\Omega)} \approx 2 \times 10^{-16},
\end{align*}
which shows stationarity of the resulting solution $(p_0,m_0)$ up to machine precision.
The resulting stationary pressure and conductances are depicted in Figure~\ref{fig:stationary_instable}. 

\begin{figure} \centering
 \includegraphics[width=.24\textwidth]{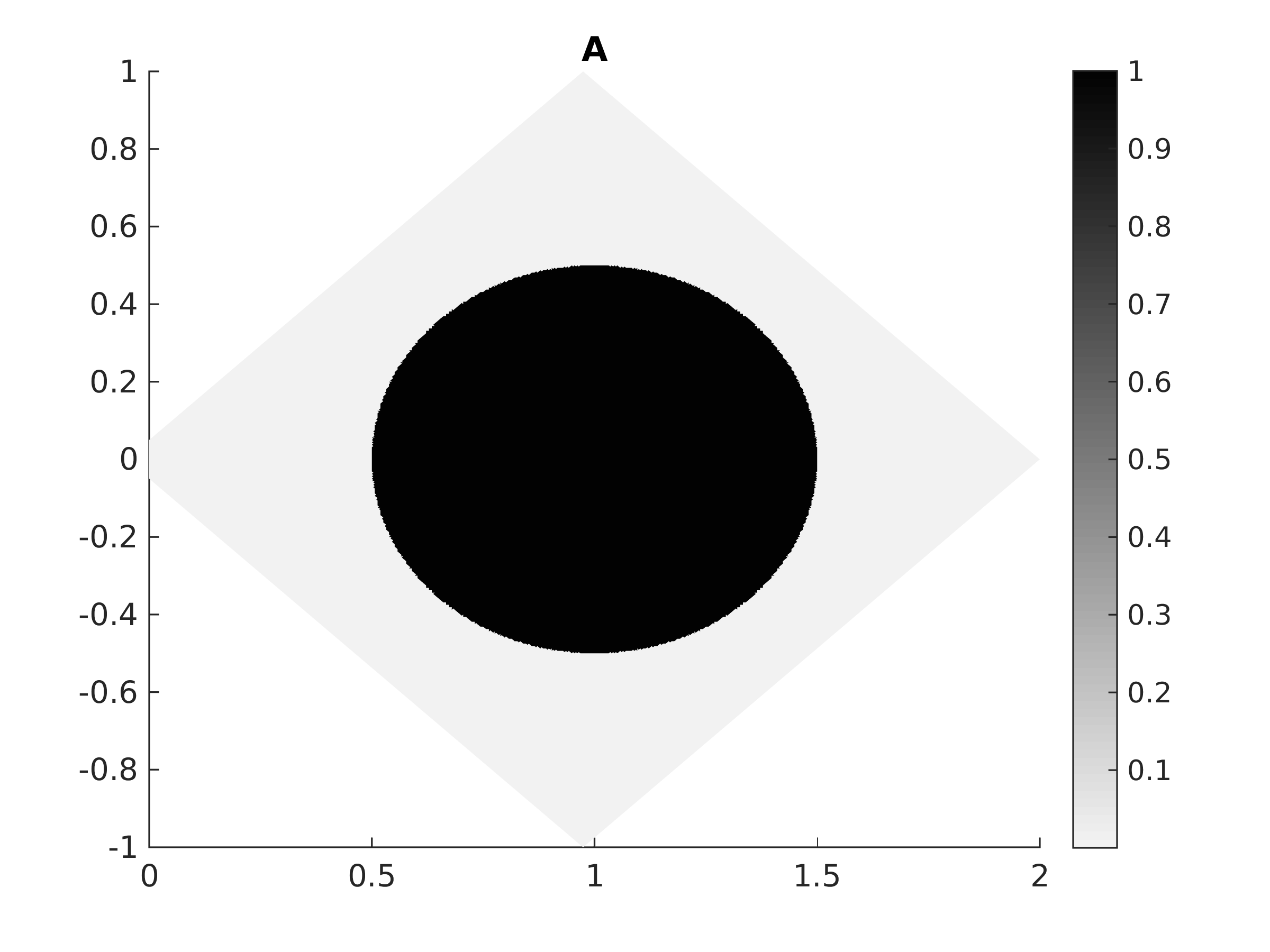}
 \includegraphics[width=.24\textwidth]{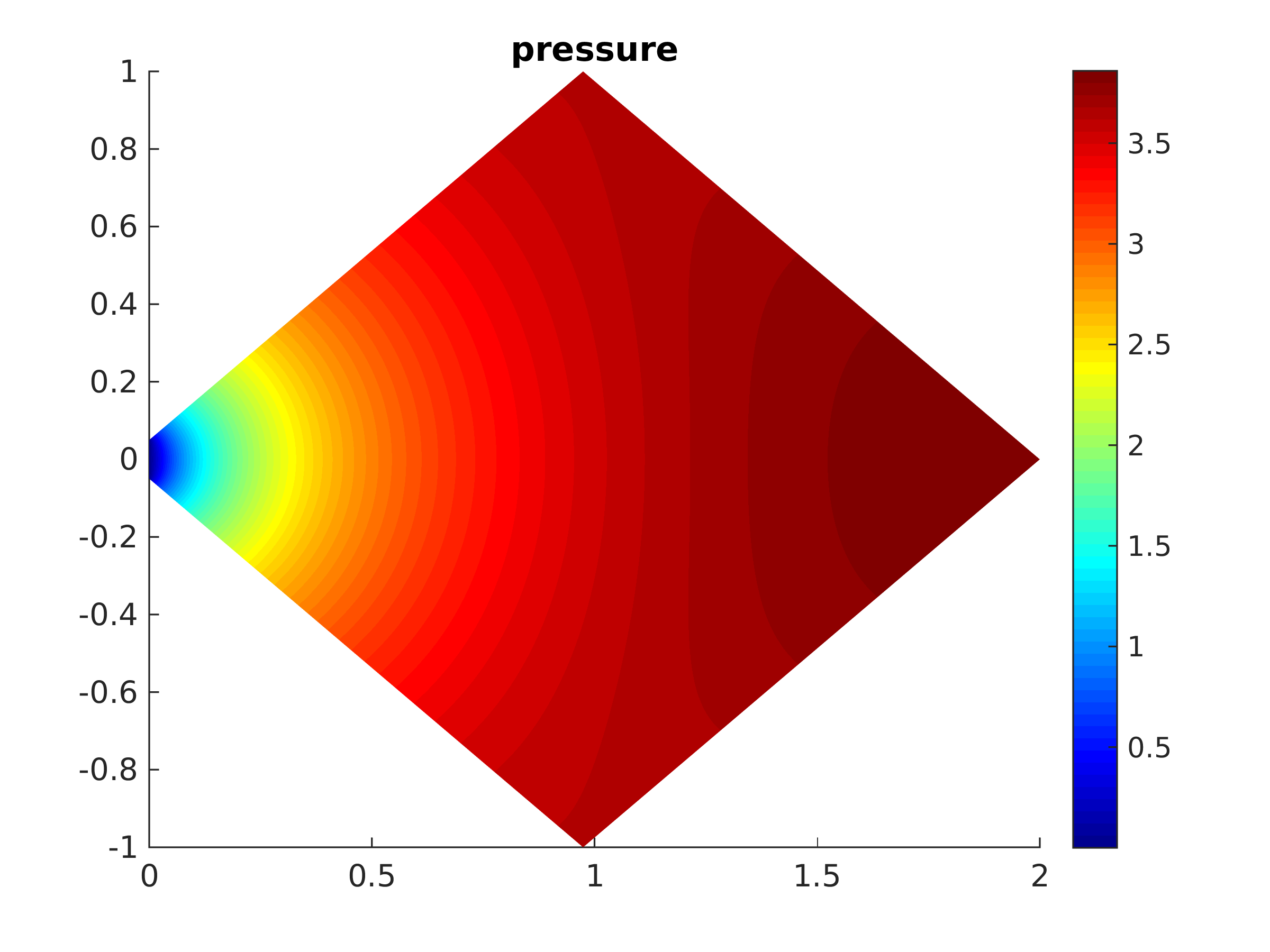}
 \includegraphics[width=.24\textwidth]{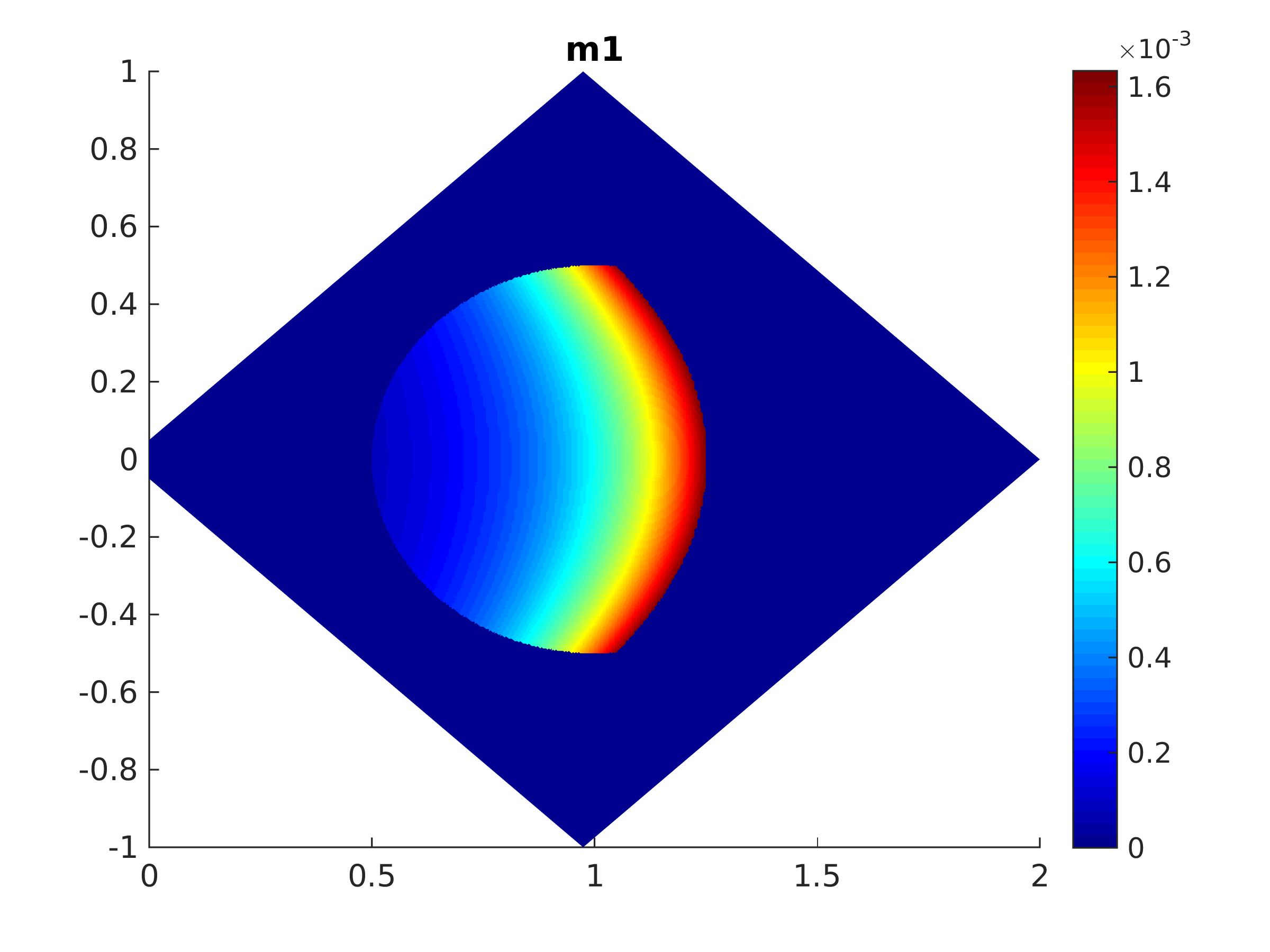}
 \includegraphics[width=.24\textwidth]{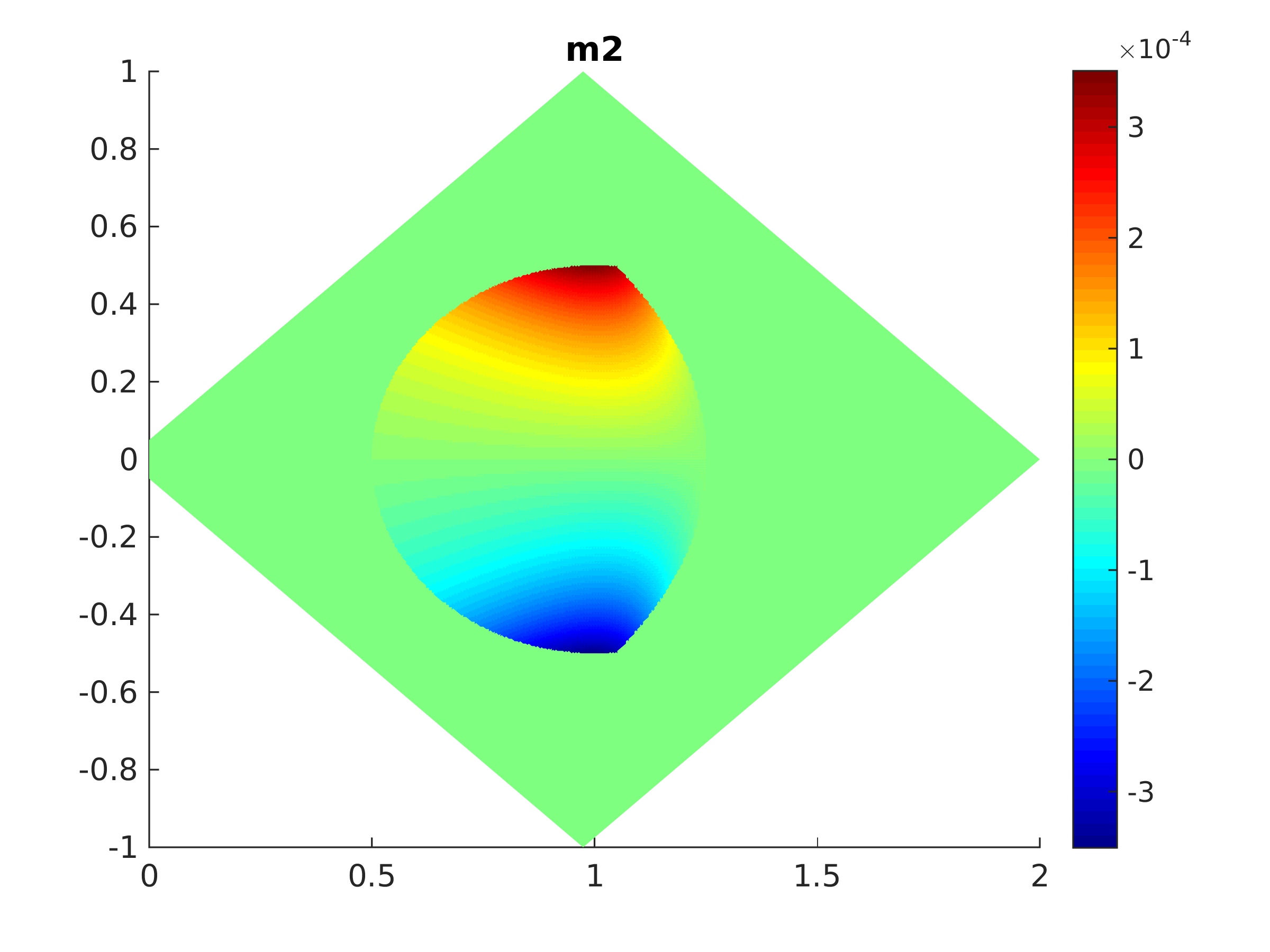}
 \caption{Stationary solution for $D=0$ and $\gamma=1/2$ via minimization of $\mathcal{F}_\alpha$ defined in \eqref{fF}. From left to right: Indicator function $\chi_\mathcal{A}$ of the set $\mathcal{A}$. Stationary pressure $p_0$. Stationary conductances $m_{0,1}$ and $m_{0,2}$.\label{fig:stationary_instable}}
\end{figure}

In order to investigate the stability of the stationary state, we let $\eta$ denote uniformly distributed random noise on $[-\frac{1}{2},\frac{1}{2}]$, which we normalize such that $\|\eta\|_{L^2(\Omega)}=1$. We set
\begin{align*}
 m^0_\eta = (1+ \frac{\eta}{1000}) m_0
\end{align*}
as initial datum for our time-stepping scheme. The resulting evolution is depicted in Figure~\ref{fig:stationary_instable_evolution}.
Since the added noise is very small, the first picture in Figure~\ref{fig:stationary_instable_evolution} is visually identical to the absolute value $|m_0|$ of the unperturbed stationary solution.
We observe that $m_\eta(t)$ does not converge to $m_0$, showing instability of the stationary state $(p_0,m_0)$. 

\begin{figure} \centering
 \includegraphics[width=.24\textwidth]{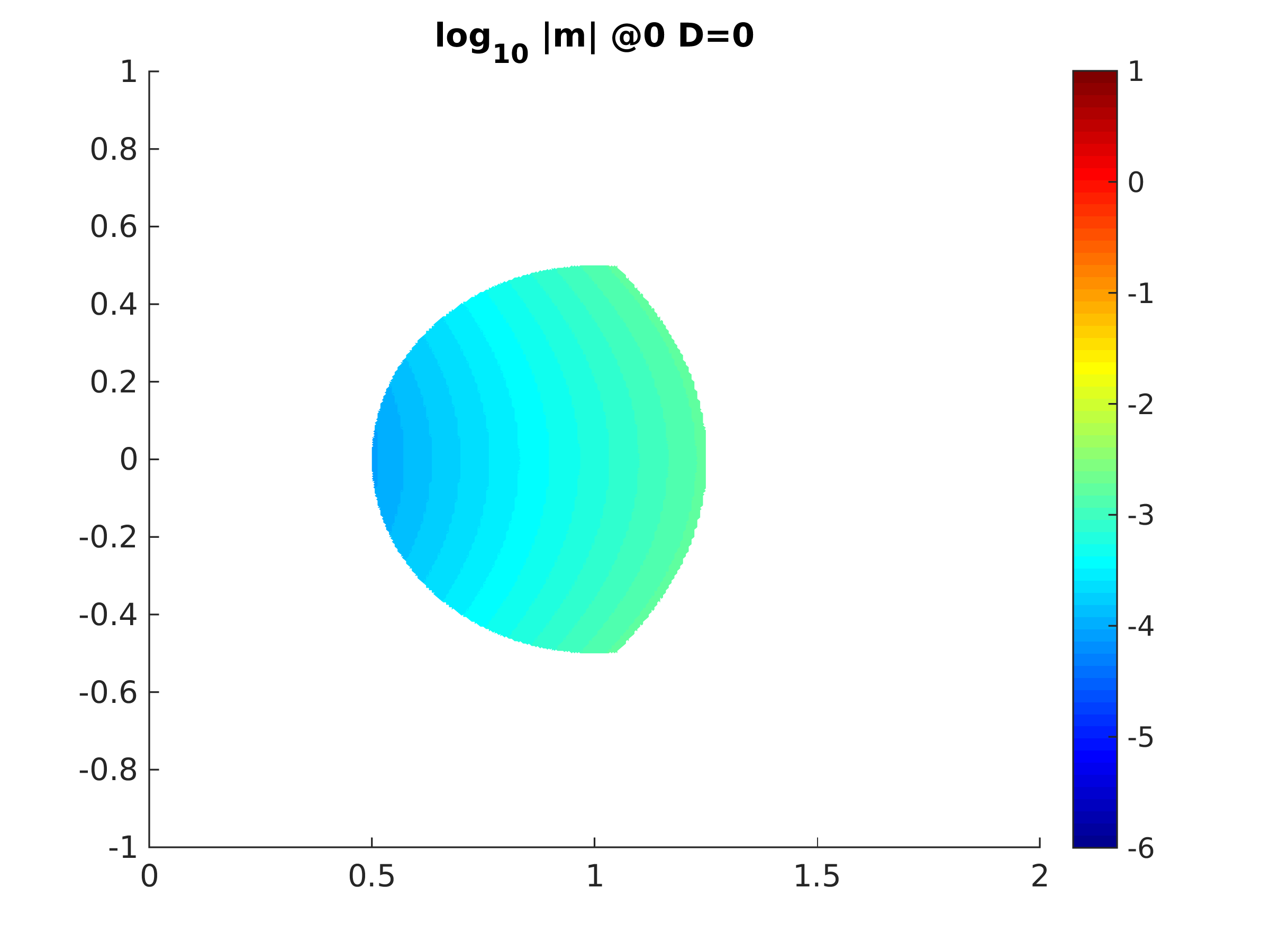}
 \includegraphics[width=.24\textwidth]{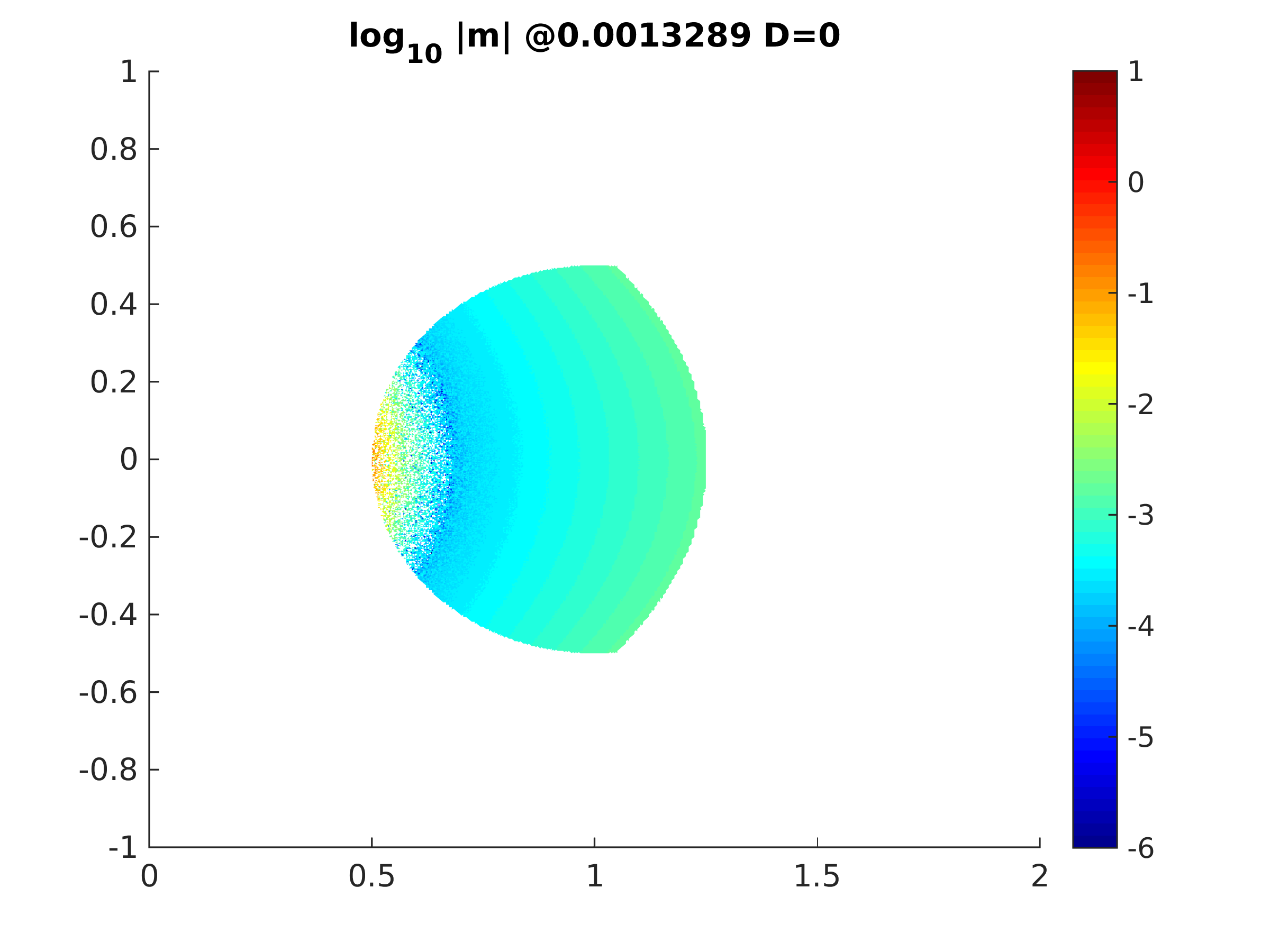}
 \includegraphics[width=.24\textwidth]{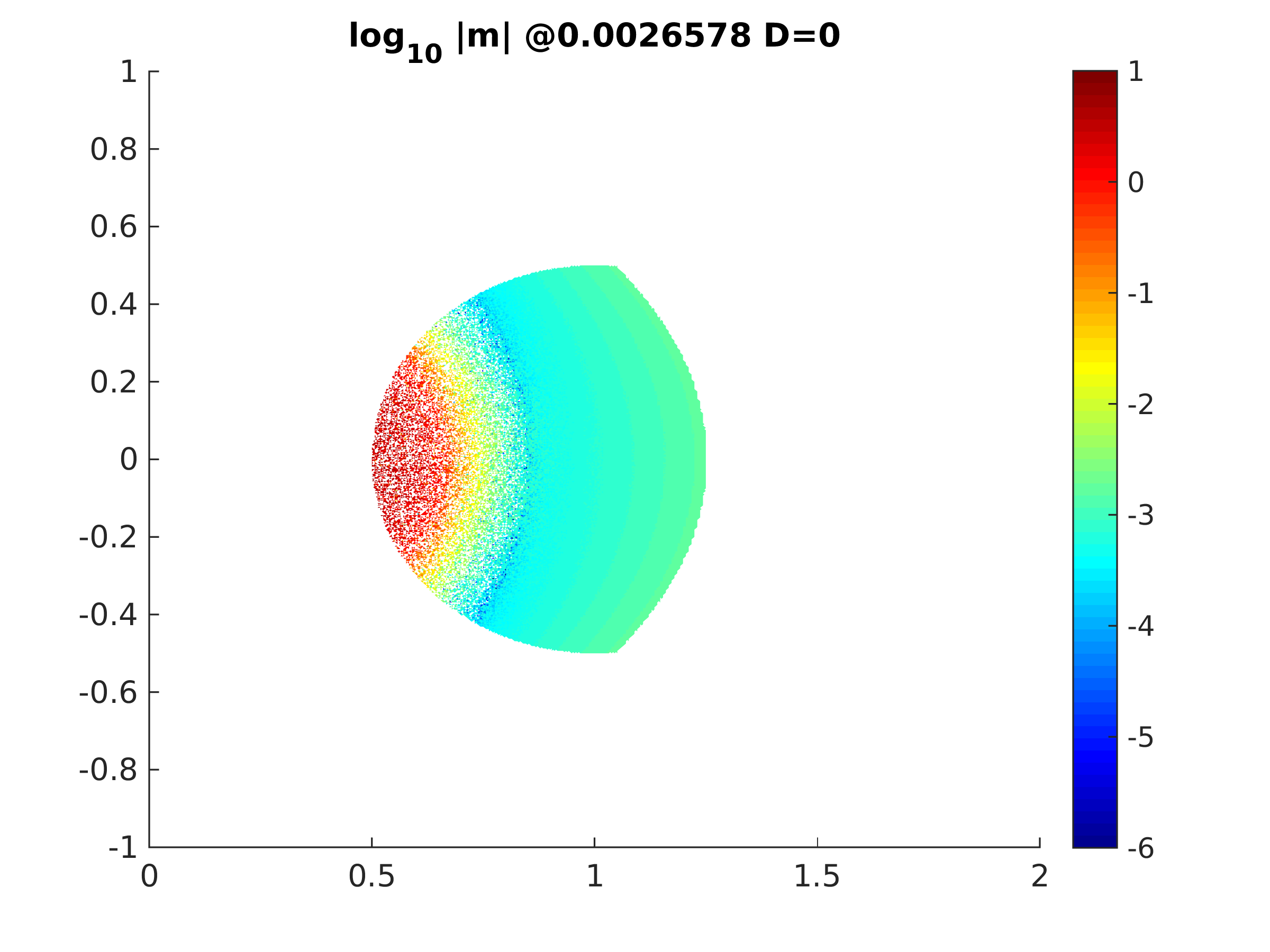}
 \includegraphics[width=.24\textwidth]{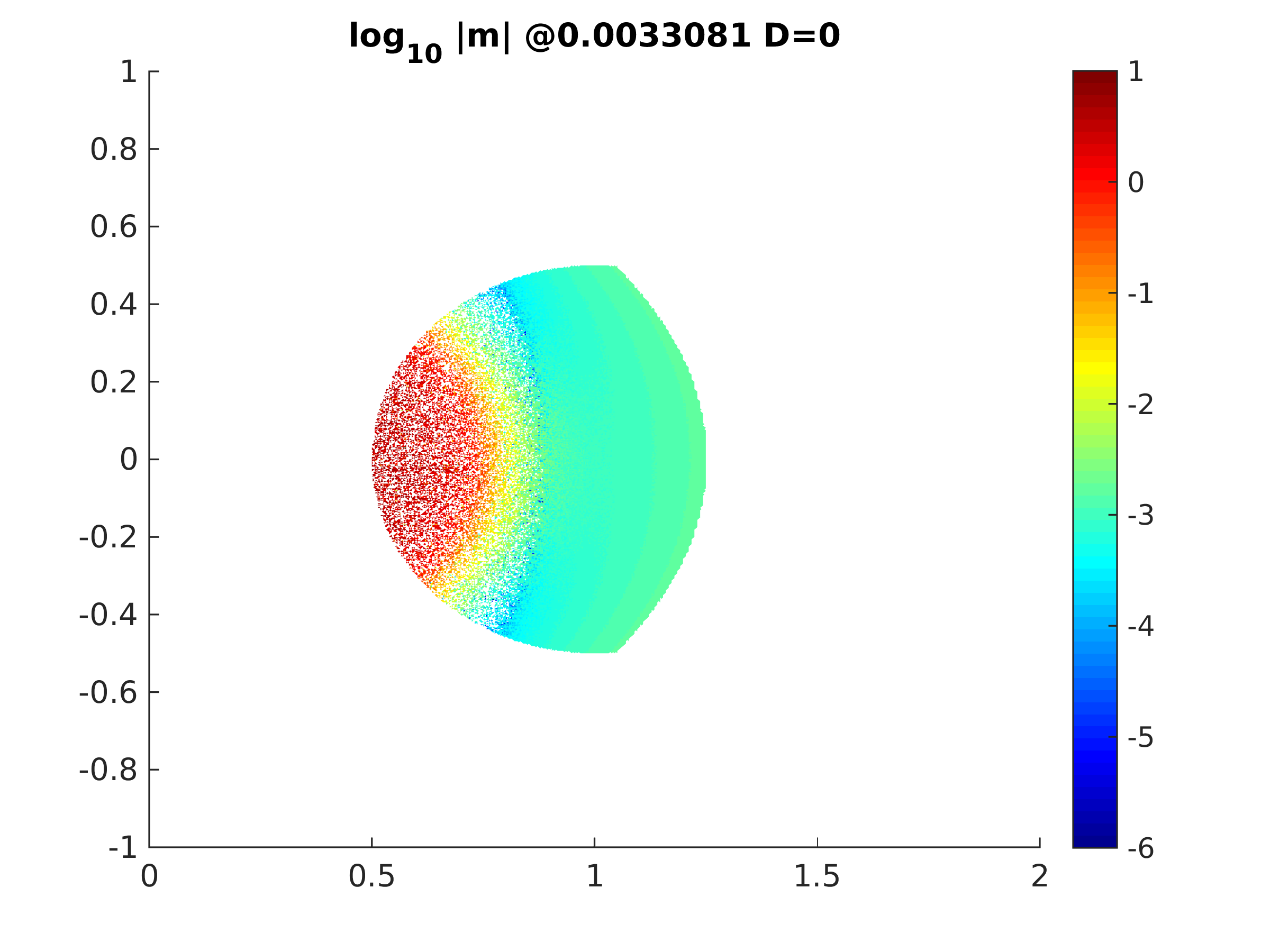}
 \caption{Evolution of the absolute value of the perturbed stationary solution $m^0_\eta$ computed in Section~\ref{sec:num_instability} for different times in a logarithmic scale.\label{fig:stationary_instable_evolution}}
\end{figure}

\subsection{Finite time break-down for $\gamma<1/2$}
In Section~\ref{sec:finite_time_breakdown} we have proven that $m$ decays exponentially to zero for $-1\leq \gamma\leq 1$, and that $m$ becomes zero after a finite time if $\gamma<1/2$ and the quantity $cS$ is sufficiently small.
In this case, the relaxation term $|m|^{2(\gamma-1)} m$ develops a singularity and is meaningless in the limit $|m|\to 0$. 
In the following we demonstrate numerically that this happens also in two dimensions, which complements the one-dimensional analysis of Section~\ref{sec:finite_time_breakdown}.
In view of Section~\ref{sec:finite_time_breakdown}, we replace the homogeneous Dirichlet conditions for $m$ by homogeneous Neumann conditions.
For our test we set $c=1$ and $\rho=0$. To start with a strictly positive initial datum, we modify $m^0$ as follows
\begin{align*}
 \tilde m^0_1 = m^0_1 + 10^{-3},\quad \tilde m^0_2=m^0_2,
\end{align*}
and we define the extinction time as
$$
T_{ex,\gamma} = \min\{t=t^k: \min_{x\in\Omega} |m_h^k(x)|< 10^{-8}\}.
$$
The results are depicted in Table~\ref{tab:finite_time_extinction}. We observe that the smaller $\gamma$ the shorter the extinction time is.
Monotone increase of $\gamma\mapsto T_{ex,\gamma}$ might be expected since for smaller values of $\gamma$ the relaxation term becomes more singular as $m\to 0$. Therefore, for smaller $\gamma$ the relaxation term dominates the activation term already for smaller times. In particular, $f_{\gamma,c}(m_h^k,\nabla p_h^k)$ acts like a sink. 
The threshold $10^{-8}$ seems to be somewhat arbitrary in the first place. However, in all our numerical simulations $\min_{x\in\Omega}|m_h^k(x)|$ decreased in a continuous fashion to values being approximately $10^{-6}$, and then dropped below the threshold in one step. Therefore, all threshold values in the interval $[10^{-8},10^{-6}]$ would yield the same $T_{ex,\gamma}$ in these examples.
This result indicates that Lemma~\ref{lem:extFT} might be extended to multiple dimensions.

\begin{table}
\centering
 \begin{tabular}{c| c c c c}
  $\gamma$ & 0 & $\frac{1}{10}$ & $\frac{1}{4}$ & $\frac{2}{5}$\\
  \hline
  $T_{ex,\gamma}$ &  $5.3\times 10^{-7}$ & $2.6\times 10^{-6}$& $2.1 \times 10^{-5}$ &  $2.1\times 10^{-4}$\\
  $\min_{x\in\Omega} |m_h^k(x)|$ & $5.9\times 10^{-9}$& $5.4\times 10^{-9}$& $3.7\times 10^{-12}$& $2.8\times 10^{-10}$
 \end{tabular}
 \caption{Extinction times $T_{ex,\gamma}$ for different values of $\gamma$.\label{tab:finite_time_extinction}}
\end{table}

\medskip

\noindent{\bf Acknowledgment.}
BP is  (partially) funded by the french "ANR blanche" project Kibord:  ANR-13-BS01-0004" and by Institut Universitaire de France. 
PM acknowledges support of the Fondation Sciences Mathematiques de Paris in form of his Excellence Chair 2011.
MS acknowledges support by ERC via Grant EU FP 7 - ERC Consolidator Grant 615216 LifeInverse.


\end{document}